\documentclass[11pt,a4paper]{article}

\usepackage[margin=1in]{geometry}
\usepackage{amsmath,amssymb,amsthm,mathtools,bm,mathrsfs}
\usepackage{booktabs,array,graphicx}
\usepackage{enumitem}
\usepackage{microtype}
\usepackage{xcolor}
\usepackage{tikz}
\usetikzlibrary{arrows.meta,positioning,fit,calc,shapes.geometric}
\usepackage{algorithm}
\usepackage{algpseudocode}
\usepackage[numbers,sort&compress]{natbib}
\usepackage[hidelinks]{hyperref}
\usepackage[capitalize,nameinlink]{cleveref}
\makeatletter
\renewcommand{\theHALG@line}{\thealgorithm.\arabic{ALG@line}}
\makeatother
\allowdisplaybreaks[3]
\usepackage{array}
\usepackage{makecell}
\newcolumntype{C}[1]{>{\centering\arraybackslash}m{#1}}

\theoremstyle{plain}
\newtheorem{theorem}{Theorem}[section]
\newtheorem{lemma}[theorem]{Lemma}
\newtheorem{proposition}[theorem]{Proposition}
\newtheorem{corollary}[theorem]{Corollary}
\theoremstyle{definition}

\theoremstyle{plain}

\newcommand{\newsiamremark}[2]{%
  \theoremstyle{remark}%
  \newtheorem{#1}[theorem]{#2}%
  \theoremstyle{plain}%
}

\newenvironment{keywords}
  {\par\smallskip\begingroup\small\noindent\textbf{Keywords.}\ \ignorespaces}
  {\par\endgroup}

\crefname{section}{section}{sections}
\Crefname{section}{Section}{Sections}
\crefname{subsection}{subsection}{subsections}
\Crefname{subsection}{Subsection}{Subsections}
\crefname{theorem}{Theorem}{Theorems}
\Crefname{theorem}{Theorem}{Theorems}
\crefname{lemma}{Lemma}{Lemmas}
\Crefname{lemma}{Lemma}{Lemmas}
\crefname{algorithm}{Algorithm}{Algorithms}
\Crefname{algorithm}{Algorithm}{Algorithms}

\newsiamremark{assumption}{Assumption}
\newsiamremark{remark}{Remark}
\crefname{assumption}{Assumption}{Assumptions}
\Crefname{assumption}{Assumption}{Assumptions}
\crefname{remark}{Remark}{Remarks}
\Crefname{remark}{Remark}{Remarks}

\newcommand{\R}{\mathbb R}
\newcommand{\E}{\mathbb E}
\newcommand{\Ical}{\mathcal I}
\newcommand{\Pcal}{\mathcal P}
\newcommand{\Rcal}{\mathcal R}
\newcommand{\Acal}{\mathcal A}
\newcommand{\Hcal}{\mathcal H}
\newcommand{\Tcal}{\mathcal T}
\newcommand{\Mcal}{\mathcal M}
\newcommand{\F}{\mathrm F}
\newcommand{\op}{\mathrm{op}}
\newcommand{\rank}{\operatorname{rank}}
\newcommand{\range}{\operatorname{range}}
\newcommand{\inner}[2]{\left\langle #1,#2\right\rangle}
\newcommand{\norm}[1]{\left\lVert #1\right\rVert}
\newcommand{\abs}[1]{\left\lvert #1\right\rvert}
\newcommand{\ceil}[1]{\left\lceil #1\right\rceil}
\newcommand{\bX}{\bm X}
\newcommand{\bY}{\bm Y}
\newcommand{\bZ}{\bm Z}
\newcommand{\bE}{\bm E}
\newcommand{\bN}{\bm N}
\newcommand{\bU}{\bm U}
\newcommand{\bV}{\bm V}
\newcommand{\bA}{\bm A}
\newcommand{\bB}{\bm B}
\newcommand{\bC}{\bm C}
\newcommand{\bM}{\bm M}
\newcommand{\bL}{\bm L}
\newcommand{\bR}{\bm R}
\newcommand{\bI}{\bm I}
\newcommand{\bG}{\bm G}
\newcommand{\bQ}{\bm Q}

\newcommand{\bS}{\bm S}
\newcommand{\bW}{\bm W}
\newcommand{\sharpnorm}[1]{\norm{#1}_{\sharp}}

\tikzset{
  diagramnode/.style={draw=black!65, rounded corners=2pt, fill=black!3,
    align=center, inner sep=5pt, font=\small},
  proofnode/.style={draw=black!60, rounded corners=2pt, fill=black!2,
    align=center, inner sep=4pt, font=\scriptsize},
  diagramarrow/.style={-{Latex[length=2.2mm]}, thick, draw=black!70}
}

\title{Near-Optimal Nonconvex Matrix Completion}
\author{%
Jian-Feng Cai\thanks{Department of Mathematics, Hong Kong University of Science and Technology. E-mail: \texttt{jfcai@ust.hk}.}
\and
Xiliang Lu\thanks{School of Mathematics and Statistics, Hubei Center for Applied Mathematics, and Hubei Key Laboratory of Computational Science, Wuhan University, Wuhan 430072, China. E-mail: \texttt{xllv.math@whu.edu.cn}.}
\and
Juntao You\thanks{(corresponding author.) School of Artificial Intelligence, Hubei Center for Applied Mathematics, and Hubei Key Laboratory of Computational Science, Wuhan University, Wuhan 430072, China.  E-mail: \texttt{youjuntao@whu.edu.cn}. }
}
\date{}
\hypersetup{
  pdftitle={Near-Optimal Nonconvex Matrix Completion},
  pdfauthor={Jian-Feng Cai, Xiliang Lu, and Juntao You},
  pdfkeywords={matrix completion, sample complexity, Riemannian optimization, Riemannian gradient descent, Riemannian Gauss--Newton method, leave-one-out analysis}
}

\begin{document}
\maketitle

\begin{abstract}
We study nonconvex methods for matrix completion, the problem of recovering a low-rank matrix from a subset of its entries. Convex methods achieve sample complexity linear in the matrix dimension and the rank, up to logarithmic factors, whereas global guarantees for commonly used nonconvex methods require a higher polynomial dependence on the rank. We close this gap by analyzing Riemannian gradient descent (RGD) and Riemannian Gauss--Newton (RGN) methods. For an $n\times n$ matrix of rank $r$ with incoherence parameter $\mu$ and condition number $\kappa$, the two methods achieve exact recovery with high probability from $O(\mu nr\log n\log(n\kappa))$ and $O(\mu nr\log n\log(2\mu r\kappa))$ observations, respectively. The methods use a multiscale residual initialization, while the analysis simultaneously controls the spectral error and incoherence. The resulting RGD iterates converge linearly, whereas RGN eventually converges Q-quadratically.
\end{abstract}

\begin{keywords}
nonconvex matrix completion, sample complexity, Riemannian gradient descent, Riemannian Gauss--Newton method, multiscale  initialization, leave-one-out analysis
\end{keywords}

\section{Introduction}\label{sec:introduction}

Matrix completion seeks to recover a low-rank matrix from a subset of its entries and arises in a broad range of problems in machine learning and data analysis, including collaborative filtering~\cite{KeshavanMontanariOh2010}, dimensionality reduction and clustering~\cite{ChenBhojanapalliSanghaviWard2015}, model reduction and system identification~\cite{LiuVandenberghe2009}, and sensor network localization~\cite{ChenBhojanapalliSanghaviWard2015}. Given an unknown matrix $\bX_\star\in\R^{n\times n}$ of rank $r$ and an observed index set $\Omega$, a natural formulation is
\begin{equation}\label{eq:matrix-completion}
\min_{\bX\in\R^{n\times n}}\rank(\bX) \quad\text{subject to}\quad \Pcal_{\Omega}(\bX)=\Pcal_{\Omega}(\bX_\star),
\end{equation}
where $\Pcal_\Omega$ retains the entries indexed by $\Omega$.

Since rank minimization is computationally intractable in general, the seminal work of Cand\`es and Recht~\cite{CandesRecht2009} studied the convex relaxation obtained by replacing the rank with the nuclear norm. Cand\`es and Tao~\cite{CandesTao2010} proved exact recovery from $O(\mu_{\mathrm{s}}^2nr\log^6 n)$ randomly observed entries under the strong incoherence condition. Chen~\cite{Chen2015} later showed that $O(\mu nr\log^2 n)$ observations suffice under standard incoherence. This bound is optimal in its dependence on $\mu$, $n$, and $r$, up to logarithmic factors. Nevertheless, nuclear norm minimization can be computationally expensive. To reduce this computational cost, a number of efficient nonconvex methods have been developed that exploit the rank constraint directly. Their global recovery guarantees, however, generally require a higher polynomial dependence on $r$; see Table~\ref{tab:theory-comparison}.  A natural question is \emph{whether efficient nonconvex methods for matrix completion can attain a sample complexity linear in \(n\) and \(r\), up to logarithmic factors.}

We address this question by establishing near-optimal sample complexity bounds for RGD and RGN equipped with a multiscale residual initialization. Under standard incoherence condition, we establish global recovery guarantees for RGD and RGN, with RGD converging linearly and RGN eventually converging Q-quadratically. With high probability, the two methods recover $\bX_\star$ with sample complexities
\[
O(\mu nr\log n\log(n\kappa)) \qquad\text{and}\qquad O(\mu nr\log n\log(2\mu r\kappa)),
\]
respectively, where $\kappa=\sigma_1(\bX_\star)/\sigma_r(\bX_\star)$ is the condition number.

\paragraph{Related work.}
A common nonconvex approach to matrix completion is to factorize
$\bX=\bL\bR^\top$ and optimize over the factors, as in
OptSpace~\cite{KeshavanMontanariOh2010},
alternating minimization~\cite{JainNetrapalliSanghavi2013,Hardt2014},
and gradient-based methods~\cite{SunLuo2016,WangZhangGu2017,CaiCaiYou2023}. For incoherent positive semidefinite matrices, gradient descent~\cite{MaWangChiChen2020} converges linearly without explicit regularization while maintaining incoherence along the iterates, and scaled projected gradient descent~\cite{NgoSaad2012,TongMaChi2021,XuShenChiMa2023,CaiHuangLuYou2025} further removes the dependence of the convergence rate on $\kappa$. Projected-gradient and hard-thresholding methods~\cite{JainMekaDhillon2010,JainNetrapalli2015,TannerWei2013,BlanchardTannerWei2015} instead work directly with the matrix variable, but require a rank-$r$ approximation after each gradient step. Riemannian methods~\cite{Vandereycken2013,WeiCaiChanLeung2020} avoid this large-scale truncation by restricting the search direction to the tangent space of the fixed-rank manifold. Second-order methods have also been studied, including MatrixIRLS~\cite{KummerleVerdun2021} and Gauss--Newton methods~\cite{ZilberNadler2022,ZhuYuan2025,VakninLauferNadler2025}. Despite these developments, the known global recovery guarantees generally retain a higher polynomial dependence on the rank; for example, Riemannian gradient descent~\cite{WeiCaiChanLeung2020} requires $O(\mu\kappa^6nr^2\log^2 n)$ observations under certain conditions. In contrast, for Gaussian matrix sensing, Riemannian gradient descent~\cite{CaiWuXia2025} attains the optimal $O(nr)$ sample complexity. For matrix completion, however, the sampling operator does not satisfy a uniform restricted isometry over low-rank matrices, and the iterates must be shown to remain incoherent. This makes the analysis more challenging.

Another issue is the dependence on the condition number in initialization. The ordinary spectral estimator, widely used in nonconvex methods, can incur an additional factor $\kappa^2$ in the sampling requirement (see \Cref{thm:necessary-scales} in this work). Stagewise methods such as SoftDeflate~\cite{HardtWootters2014} and the projected-gradient method~\cite{JainNetrapalli2015} reduce or remove this dependence, but the recovery guarantees still have a higher polynomial dependence on $r$. Our multiscale initialization instead controls the spectral, row, column, and entrywise errors simultaneously, which avoids an additional polynomial dependence on $r$ in the sampling requirement.

\begin{table}[t]
\caption{Complexity comparison for matrix completion under the stated recovery guarantees. All results assume incoherence, while Factorized GD~\cite{MaWangChiChen2020} also assumes PSD and $\kappa=O(1)$; Riemannian GD~\cite{WeiCaiChanLeung2020} additionally assumes spikiness.}
\label{tab:theory-comparison}
\centering
\footnotesize
\setlength{\tabcolsep}{2.5pt}
\renewcommand{\arraystretch}{1.15}
\setlength{\extrarowheight}{0pt}

\begin{tabular}{|
>{\centering\arraybackslash}m{\dimexpr0.125\textwidth-2\tabcolsep-1.2\arrayrulewidth\relax}|
>{\centering\arraybackslash}m{\dimexpr0.125\textwidth-2\tabcolsep-1.2\arrayrulewidth\relax}|
>{\centering\arraybackslash}m{\dimexpr0.165\textwidth-2\tabcolsep-1.2\arrayrulewidth\relax}|
>{\centering\arraybackslash}m{\dimexpr0.250\textwidth-2\tabcolsep-1.2\arrayrulewidth\relax}|
>{\centering\arraybackslash}m{\dimexpr0.295\textwidth-2\tabcolsep-1.2\arrayrulewidth\relax}|}
\hline

\Gape[4pt][4pt]{\textbf{Method}}
& \Gape[4pt][4pt]{\shortstack[c]{\textbf{Iteration}\\\textbf{complexity}}}
& \Gape[4pt][4pt]{\shortstack[c]{\textbf{Cost per}\\\textbf{iteration}}}
& \Gape[4pt][4pt]{\shortstack[c]{\textbf{Sample}\\\textbf{complexity}}}
& \Gape[4pt][4pt]{\shortstack[c]{\textbf{Total computational}\\\textbf{complexity}}} \\
\hline
\Gape[4pt][4pt]{\shortstack[c]{Nuclear norm\\minimization\\\cite{CandesTao2010,Chen2015,DingChen2020}}}
&\Gape[4pt][4pt]{\shortstack[c]{Polynomial\\ time; solver\\ dependent}}
&\Gape[4pt][4pt]{Solver dependent}
&\Gape[4pt][4pt]{$O\!\left(\mu nr\log(2\mu r)\log n\right)$}
&\Gape[4pt][4pt]{Solver dependent}  \\
\hline

\Gape[4pt][4pt]{\shortstack[c]{Factorized \\ GD\cite{MaWangChiChen2020}}}
&
\Gape[4pt][4pt]{$O\!\left(\kappa^2\log\frac{1}{\varepsilon}\right)$}
&
\Gape[4pt][4pt]{$O(|\Omega|r+nr)$}
&
\Gape[4pt][4pt]{$O(\mu^3nr^3\log^3 n)$}
&
\Gape[4pt][4pt]{$O\!\left(\mu^3\kappa^2nr^4\log^3 n\log\frac{1}{\varepsilon}\right)$}
\\
\hline

\Gape[4pt][4pt]{\shortstack[c]{SVP/PGD\\ \cite{JainMekaDhillon2010,DingChen2020,WangWei2026}}}
&
\Gape[4pt][4pt]{$O\!\left(\log\frac{1}{\varepsilon}\right)$}
&
\Gape[4pt][4pt]{$O(n^3)$}
&
\Gape[4pt][4pt]{$O(\mu^2\kappa^4nr^2\log n)$}
&
\Gape[4pt][4pt]{$O\!\left(n^3\log\frac{1}{\varepsilon}\right)$}
\\
\hline

\Gape[4pt][4pt]{\shortstack[c]{ScaledPGD\\ \cite{TongMaChi2021}}}
&
\Gape[4pt][4pt]{$O\!\left(\log\frac{1}{\varepsilon}\right)$}
&
\Gape[4pt][4pt]{$O(|\Omega|r+nr^2)$}
&
\Gape[4pt][4pt]{$O\!\left(\mu\kappa^2nr^2(\mu\kappa^2\vee\log n)\right)$}
&
\Gape[4pt][4pt]{$O\!\left(\mu\kappa^2nr^3(\mu\kappa^2\vee\log n)\log\frac{1}{\varepsilon}\right)$}
\\
\hline

\Gape[4pt][4pt]{\shortstack[c]{Riemannian \\ GD\cite{WeiCaiChanLeung2020}}}
&
\Gape[4pt][4pt]{$O\!\left(\log\frac{1}{\varepsilon}\right)$}
&
\Gape[4pt][4pt]{$O(|\Omega|r+nr^2)$}
&
\Gape[4pt][4pt]{$O\!\left(\max\{\mu_0,\mu_1^2\}\kappa^6nr^2\log^2 n\right)$}
&
\Gape[4pt][4pt]{$O\!\left(\max\{\mu_0,\mu_1^2\}\kappa^6nr^3\log^2 n\log\frac{1}{\varepsilon}\right)$}
\\
\hline

\Gape[4pt][4pt]{\shortstack[c]{RGD\\(this paper)}}
&
\Gape[4pt][4pt]{$O\!\left(\log\frac{1}{\varepsilon}\right)$}
&
\Gape[4pt][4pt]{$O(|\Omega|r+nr^2)$}
&
\Gape[4pt][4pt]{$O\!\left(\mu nr\log n\log(n\kappa)\right)$}
&
\Gape[4pt][4pt]{$O\!\left(\mu nr^2\log^2(n\kappa)\left(\log n+\log\frac{1}{\varepsilon}\right)\right)$}
\\
\hline

\Gape[4pt][4pt]{\shortstack[c]{RGN\\(this paper)}}
&
\Gape[4pt][4pt]{$O\!\left(\log\log\frac{1}{\varepsilon}\right)$}
&
\Gape[4pt][4pt]{$O\!\left(J_k(|\widehat{\Omega}|r+nr^2)\right)$}
&
\Gape[4pt][4pt]{$O\!\left(\mu nr\log n\log(2\mu r\kappa)\right)$}
&
\Gape[4pt][4pt]{$O\!\left(\mu Jnr^2\log^2 n\log(2\mu r\kappa)\log\frac{1}{\varepsilon}\right)$}
\\
\hline

\end{tabular}
\end{table}

\paragraph{Our contributions and proof strategy.}
Our main contribution is to establish near-optimal global sample complexity bounds for efficient nonconvex RGD and RGN methods in matrix completion. As discussed above, two sources of additional sample complexity arise in the usual analysis: the conversion from spectral to Frobenius error can introduce an additional factor $r$, while ordinary spectral initialization can incur a factor $\kappa^2$.

To control the loss in rank $r$, we keep track of the spectral, row, column, and entrywise errors simultaneously. For $\bZ\in\R^{n\times n}$, define
\[
\norm{\bZ}_{2,\infty}:=\max_i\norm{\bm e_i^T\bZ}_2, \qquad \norm{\bZ^T}_{2,\infty}:=\max_j\norm{\bZ\bm e_j}_2, \qquad \norm{\bZ}_{\infty}:=\max_{i,j}|(\bZ)_{ij}|,
\]
and introduce the \emph{sharp-norm} as
\begin{equation}\label{eq:sharp-main}
\sharpnorm{\bZ}:=\max\left\{\norm{\bZ}_{\op},\frac12\sqrt{\frac{n}{\mu r}}\norm{\bZ}_{2,\infty},\frac12\sqrt{\frac{n}{\mu r}}\norm{\bZ^T}_{2,\infty},\frac{n}{4\mu r}\norm{\bZ}_{\infty}\right\}.
\end{equation}
The multiscale initialization maintains this sharp-norm control and reaches the required local regions without introducing an additional polynomial dependence on $r$ in the sampling requirement. For RGD, the eventual conversion to the Frobenius norm affects only the number of initialization steps; for RGN, the sharp-norm control is continued through the initial iterations before the analysis enters the Frobenius regime.

To control the loss in $\kappa^2$, we use a multiscale residual initialization in place of the ordinary spectral estimator. Successive residual reconstructions reduce the sharp-norm error by a fixed factor at the sampling level $p\gtrsim\mu r\log n/n$, and we establish both the sampling and computational complexities of this procedure. We further show that the $\mu\kappa^2r/n$ sampling scale of the ordinary spectral estimator is necessary over an explicit family of incoherent matrices.

The proof combines these initialization estimates with the local convergence analysis. For RGD, the initialization reaches the required Frobenius neighborhood, where a uniform tangent-space sampling estimate yields linear convergence. For RGN, a finite leave-one-out argument propagates the sharp-norm control through the initial iterations; once the iterates enter a sufficiently small Frobenius neighborhood, a local deterministic argument yields quadratic convergence.

\paragraph{Organization and notation.}
Section~\ref{sec:model} introduces the observation model, the Riemannian algorithms, and the multiscale initialization. Section~\ref{sec:main} states the global recovery guarantees and the lower bound for ordinary spectral initialization. Section~\ref{sec:interfaces} presents the proof framework, including the local convergence and initialization results. Their proofs, together with the proofs of the global theorems, are given in Sections~\ref{sec:local-rgd}--\ref{sec:local-rgn}. Sections~\ref{sec:numerics} and~\ref{sec:conclusion} present the numerical experiments and concluding remarks. The appendices collect the supporting probabilistic and geometric estimates, the proof of the spectral lower bound, the sharp-norm estimates for spectral reconstruction, the leave-one-out analysis, and the implementation and complexity analysis.

Throughout the paper, bold lowercase letters denote vectors and bold uppercase letters denote matrices, while scalars are written in ordinary type. The vector $\bm e_i$ denotes the $i$th standard basis vector. For a vector $\bm x$, $\norm{\bm x}_2$ denotes the Euclidean norm. For a matrix $\bZ$, $\norm{\bZ}_{\op}$ and $\norm{\bZ}_{\F}$ denote the operator norm and Frobenius norm respectively. We write $\sigma_i(\bZ)$ for the $i$-th largest singular value of $\bZ$, and $\rank(\bZ)$ and $\range(\bZ)$ for its rank and column space. The symbols $\bI$ and $\Ical$ denote the identity matrix and identity operator, respectively. We use $O(\cdot)$ for bounds up to an absolute numerical constant independent of the problem parameters.

\section{Problem Formulation and Riemannian Algorithms}\label{sec:model}

We first formulate the matrix completion problem and then describe the Riemannian algorithms considered in this paper.

\subsection{Problem setup}
Suppose that $\bX_\star\in\R^{n\times n}$,  $n\ge2$ is an unknown matrix of rank $r$, where $1\le r<n$. Assume we observe its entries independently according to the Bernoulli sampling model:
\[
Y_{ij}=
\begin{cases}
(\bX_\star)_{ij}, & \text{with probability }p,\\
\ast, & \text{with probability }1-p,
\end{cases}
\qquad 1\le i,j\le n,
\]
where $0<p\le1$. Let $\Omega:=\{(i,j):Y_{ij}\neq\ast\}$ denote the set of observed indices, also denoted by $\Omega\sim\operatorname{Bernoulli}(p)$. The associated sampling operator $\Pcal_\Omega$ is defined by
\[
(\Pcal_\Omega(\bZ))_{ij}=
\begin{cases}
Z_{ij}, & (i,j)\in\Omega,\\
0, & (i,j)\notin\Omega.
\end{cases}
\]
The matrix completion problem is to recover $\bX_\star$ from the observed entries $\Pcal_\Omega(\bX_\star)$. Let
\[
\bX_\star=\bU_\star\bm\Sigma_\star\bV_\star^T, \qquad \bm\Sigma_\star=\operatorname{diag}(\sigma_1,\ldots,\sigma_r), \qquad \sigma_1\ge\cdots\ge\sigma_r>0,
\]
be a compact singular value decomposition of $\bX_\star$. We assume that $\bX_\star$ satisfies the following standard incoherence condition.
\begin{assumption}[Incoherence~\cite{CandesRecht2009,CandesTao2010,Chen2015}]\label[assumption]{ass:incoherence}
For some $1\le\mu\le n/r$, it holds
$$
\max\left\{\max_i\norm{\bU_\star^T\bm e_i}_2^2,\max_j\norm{\bV_\star^T\bm e_j}_2^2\right\}\le\frac{\mu r}{n}.
$$
\end{assumption}
The standard incoherence condition was introduced by Cand\`es and Recht \cite{CandesRecht2009} for low-rank matrix completion. It requires the left and right singular spaces of $\bX_\star$ to be weakly correlated with the canonical basis, preventing the matrix from being concentrated on only a few entries. For the completion problem, we consider the following nonconvex formulation:
\begin{equation}\label{eq:ls-objective}
\min_{\bX\in\Mcal_r} f_\Omega(\bX),\qquad f_\Omega(\bX):=\frac{1}{2p}\norm{\Pcal_\Omega(\bX-\bX_\star)}_{\F}^2,
\end{equation}
where
\[
\Mcal_r:=\{\bX\in\R^{n\times n}:\rank(\bX)=r\}
\]
is the manifold of rank-$r$ matrices.

\subsection{Riemannian gradient descent}\label{sec:opt-formulation}

We first present the RGD method for \eqref{eq:ls-objective}. The algorithm updates the current iterate along the negative gradient direction in the tangent space of the fixed-rank manifold and then retracts the tangent update back onto the manifold. We first recall the geometry of the fixed low-rank manifold.  Let $\bX=\bU\bm\Sigma\bV^T\in\Mcal_r$ be a compact singular value decomposition. The tangent space of $\Mcal_r$ at $\bX$ is \cite{Vandereycken2013}
\[
T_{\bX}\Mcal_r = \left\{ \bU\bZ_1^T+\bZ_2\bV^T: \bZ_1,\bZ_2\in\R^{n\times r} \right\},
\]
and the orthogonal projection onto $T_{\bX}\Mcal_r$ is
\begin{equation}\label{eq:tangent-projector}
\Pcal_{T_{\bX}}(\bZ) = \bU\bU^T\bZ+\bZ\bV\bV^T-\bU\bU^T\bZ\bV\bV^T.
\end{equation}
Since $\nabla f_\Omega(\bX) = p^{-1}\Pcal_\Omega(\bX-\bX_\star)$, the Riemannian gradient of $f_\Omega$ at $\bX$ is
\[
\operatorname{grad}f_\Omega(\bX) = \Pcal_{T_{\bX}}\nabla f_\Omega(\bX) = p^{-1}\Pcal_{T_{\bX}}\Pcal_\Omega(\bX-\bX_\star).
\]
A tangent update does not in general belong to $\Mcal_r$. We use the orthographic retraction \cite{AbsilOseledets2015} to map it back onto the fixed-rank manifold. For $\bm\xi\in T_{\bX}\Mcal_r$, define
\begin{equation}\label{eq:graph-retraction-opt}
\operatorname{Retr}_{\bX}(\bm\xi) := (\bX+\bm\xi)\bV \bigl[\bU^T(\bX+\bm\xi)\bV\bigr]^{-1} \bU^T(\bX+\bm\xi),
\end{equation}
whenever $\bU^T(\bX+\bm\xi)\bV$ is nonsingular. Let $\bm\xi_k=-\operatorname{grad}f_\Omega(\bX_k)$. Since $\inner{\nabla f_\Omega(\bX_k)}{\bm\xi_k}=-\norm{\bm\xi_k}_{\F}^2$, exact line search along this tangent direction \cite{WeiCaiChanLeung2020} gives
\begin{equation}\label{eq:rgd-exact-line-search}
\alpha_k =\arg\min_{\alpha\in\R}f_\Omega(\bX_k+\alpha\bm\xi_k) =\frac{\norm{\bm\xi_k}_{\F}^2}{p^{-1}\norm{\Pcal_\Omega(\bm\xi_k)}_{\F}^2},
\end{equation}
whenever $\bm\xi_k\ne\bm0$. The RGD update is then
\[
\bX_{k+1}=\operatorname{Retr}_{\bX_k}(\alpha_k\bm\xi_k),
\]
as summarized in \Cref{alg:fixed-rgd}. If $\bm\xi_k=\bm0$, the algorithm terminates.

\begin{algorithm}[H]
\caption{Riemannian gradient descent (RGD)}
\label{alg:fixed-rgd}
\begin{algorithmic}[1]
\Statex \textbf{Input:} $p$, $\Pcal_\Omega(\bX_\star)$, and $\bX_0\in\Mcal_r$.
\For{$k=0,1,2,\ldots$}
  \State $\bm\xi_k\gets-p^{-1}\Pcal_{T_{\bX_k}}
  \bigl(\Pcal_\Omega(\bX_k)-\Pcal_\Omega(\bX_\star)\bigr)$.
  \State $\alpha_k\gets\norm{\bm\xi_k}_{\F}^2/\bigl(p^{-1}\norm{\Pcal_\Omega(\bm\xi_k)}_{\F}^2\bigr)$ if $\bm\xi_k\neq\bm0$; otherwise, stop.
  \State $\bX_{k+1}\gets\operatorname{Retr}_{\bX_k}(\alpha_k\bm\xi_k)$.
\EndFor
\end{algorithmic}
\end{algorithm}

The tangent gradient can be evaluated using sparse matrix--factor products in $O(|\Omega|r+nr^2)$ operations. The exact line-search stepsize can be evaluated in the same order by computing $\norm{\bm\xi_k}_{\F}$ and the entries of $\bm\xi_k$ on $\Omega$. The orthographic retraction can be computed from low-rank factors using two thin QR factorizations and an $r\times r$ singular value decomposition in $O(nr^2+r^3)$ operations \cite{AbsilOseledets2015}. Thus, since $r<n$, each RGD iteration costs $O(|\Omega|r+nr^2)$ operations. The implementation details are given in \Cref{app:implementation}.

\subsection{Riemannian Gauss--Newton}

Riemannian Gauss--Newton first computes a search direction in the tangent space and then retracts the tangent update back onto $\Mcal_r$. Let $\widehat{\Omega}\subset\Omega$ denote the subset of observations used for the RGN iterations. At $\bX\in\Mcal_r$, since
\[
D\operatorname{Retr}_{\bX}(0)[\bm\xi]=\bm\xi, \qquad \bm\xi\in T_{\bX}\Mcal_r,
\]
linearizing the sampled residual along the retraction gives the Gauss--Newton subproblem
\begin{equation}\label{eq:gn-model}
\min_{\bm\xi\in T_{\bX}\Mcal_r} \frac12\norm{\Pcal_{\widehat{\Omega}}(\bX-\bX_\star+\bm\xi)}_{\F}^2.
\end{equation}
Its first-order optimality condition is the tangent normal equation
\begin{equation}\label{eq:ambient-normal-equation}
\Pcal_{T_{\bX}}\Pcal_{\widehat{\Omega}}\Pcal_{T_{\bX}}\bm\xi = \Pcal_{T_{\bX}}\Pcal_{\widehat{\Omega}}(\bX_\star-\bX).
\end{equation}
When the sampled tangent normal operator is positive definite on $T_{\bX}\Mcal_r$, \eqref{eq:gn-model} has a unique minimizer. We compute this direction by applying the conjugate gradient method to \eqref{eq:ambient-normal-equation} on $T_{\bX}\Mcal_r$. At the $k$-th nonterminal RGN iteration, let $J_k\ge1$ denote the number of CG iterations used to solve the normal equation. CG is started from zero and run to the exact solution. If the right-hand side is zero, the algorithm terminates; the sequence is then continued by its final iterate for the convergence statements. In exact case, the finite-termination property of CG \cite[Theorems~2.3.2 and~3.1.1]{Greenbaum1997} gives
\[
J_k\le\dim(T_{\bX_k}\Mcal_r)=r(2n-r).
\]
The resulting iteration is summarized in \Cref{alg:fixed}.

\begin{algorithm}[H]
\caption{Riemannian Gauss--Newton (RGN)}
\label{alg:fixed}
\begin{algorithmic}[1]
\Statex \textbf{Input:} $\Pcal_{\widehat{\Omega}}(\bX_\star)$ and $\bX_0\in\Mcal_r$.
\For{$k=0,1,2,\ldots$}
  \State Apply $J_k$ CG iterations, starting from zero, to solve the following normal equation for $\bm\xi_k$:
  $$\Pcal_{T_{\bX_k}}\Pcal_{\widehat{\Omega}}\Pcal_{T_{\bX_k}}\bm\xi
  =\Pcal_{T_{\bX_k}}\Pcal_{\widehat{\Omega}}(\bX_\star-\bX_k).$$
  \State $\bX_{k+1}\gets\operatorname{Retr}_{\bX_k}(\bm\xi_k)$.
\EndFor
\end{algorithmic}
\end{algorithm}

We also consider a regularized variant of RGN. At the $k$-th iteration, define
\begin{equation}\label{eq:rgn-lambda}
\lambda_k =\frac{1}{\sigma_r(\bX_k)}\norm{\Pcal_{T_{\bX_k}}\Pcal_{\widehat{\Omega}}(\bX_\star-\bX_k)}_{\F}
\end{equation}
and add $\lambda_k\bm\xi$ to the left-hand side of the normal equation \eqref{eq:ambient-normal-equation}. At every nonterminal iteration, $\lambda_k>0$, so the regularized normal equation is positive definite on $T_{\bX_k}\Mcal_r$. Using the same CG and termination conventions as above gives \Cref{alg:regularized-rgn}.

\begin{algorithm}[H]
\caption{Regularized Riemannian Gauss--Newton}
\label{alg:regularized-rgn}
\begin{algorithmic}[1]
\Statex \textbf{Input:} $\Pcal_{\widehat{\Omega}}(\bX_\star)$ and $\bX_0\in\Mcal_r$.
\For{$k=0,1,2,\ldots$}
  \State $\displaystyle
  \lambda_k\gets
  \frac{1}
  {\sigma_r(\bX_k)}\norm{\Pcal_{T_{\bX_k}}\Pcal_{\widehat{\Omega}}(\bX_\star-\bX_k)}_{\F}.$
  \State Apply $J_k$ CG iterations, starting from zero, to solve the following normal equation for $\bm\xi_k$:
  $$\Pcal_{T_{\bX_k}}\Pcal_{\widehat{\Omega}}\Pcal_{T_{\bX_k}}\bm\xi
  +\lambda_k\bm\xi
  =
  \Pcal_{T_{\bX_k}}\Pcal_{\widehat{\Omega}}(\bX_\star-\bX_k).$$
  \State $\bX_{k+1}\gets\operatorname{Retr}_{\bX_k}(\bm\xi_k)$.
\EndFor
\end{algorithmic}
\end{algorithm}

The normal equations in \Cref{alg:fixed,alg:regularized-rgn} can be solved matrix-free without forming the sampled tangent normal matrix. Each normal-operator application costs $O(|\widehat{\Omega}|r+nr^2)$ operations, and the regularization term in \Cref{alg:regularized-rgn} does not change this order. The orthographic retraction costs $O(nr^2+r^3)$ operations. Hence, since $r<n$ and $J_k\ge1$, the $k$-th iteration of either method costs
\[
O\!\left(J_k(|\widehat{\Omega}|r+nr^2)\right)
\]
operations. The matrix-free implementation and detailed complexity analysis are given in \Cref{app:implementation}.

\subsection{Multiscale initialization}\label{sec:sample-split}

The local convergence results for both RGD and RGN require an initial point sufficiently close to $\bX_\star$. By \Cref{thm:necessary-scales}, the standard spectral estimator can require a sampling probability of order $\mu\kappa^2r/n$ to reach a constant sharp-norm neighborhood. We instead use a multiscale residual spectral initialization, related to residual spectral updates in singular value projection and iterative hard thresholding~\cite{JainMekaDhillon2010,TannerWei2013,BlanchardTannerWei2015} and to stagewise constructions for matrix completion~\cite{HardtWootters2014,JainNetrapalli2015}.

At each scale, the current approximation is corrected by the observed residual and then truncated at a decreasing spectral level. Let $\Omega^{(\ell)}\subset\Omega$ denote the observations used at scale $\ell$, with sampling probability $q$. Starting from $\bZ_0=\bm0$, set $\tau_0:=2q^{-1/2}\norm{\Pcal_{\Omega^{(0)}}(\bX_\star)}_{\F}$ and $\tau_\ell:=4^{-\ell}\tau_0$. At iteration $\ell$, the residual correction satisfies
\[
\bZ_\ell+q^{-1}\Pcal_{\Omega^{(\ell+1)}}(\bX_\star-\bZ_\ell) = \bX_\star+\bigl(q^{-1}\Pcal_{\Omega^{(\ell+1)}}-\Ical\bigr)(\bX_\star-\bZ_\ell).
\]
Thus, the sampling perturbation acts on the current error $\bX_\star-\bZ_\ell$. Accordingly, we use the following update:
\begin{equation}\label{eq:multiscale-update}
\bZ_{\ell+1} = \Tcal_{\tau_\ell}\!\left( \bZ_\ell+q^{-1}\Pcal_{\Omega^{(\ell+1)}}(\bX_\star-\bZ_\ell) \right), \qquad 0\le\ell<K,
\end{equation}
where $\Tcal_{\tau_\ell}$ denotes the spectral truncation operator defined below. For a singular value decomposition $\bY=\sum_j\sigma_j\bm u_j\bm v_j^T$, first define the hard spectral thresholding operator
\begin{equation*}
\Hcal_{\ge\lambda}(\bY):=\sum_{\sigma_j\ge\lambda}\sigma_j\bm u_j\bm v_j^T.
\end{equation*}
For $\tau>0$, $\Tcal_\tau(\bY)$ is computed as follows: Starting from $\bQ_0=\operatorname{qf}(\bG)$, where $\bG\in\R^{n\times r}$ has independent standard Gaussian entries, compute
\begin{equation}\label{eq:threshold-iteration}
\bQ_{t+1}=\operatorname{qf}\!\left(\bY(\bY^T\bQ_t)+\frac{\tau^2}{4096}\bQ_t\right), \qquad 0\le t<\ceil{12\log n},
\end{equation}
where $\operatorname{qf}$ denotes the orthogonal factor in a thin QR factorization. The term $\frac{\tau^2}{4096}\bQ_t$  preserves the eigenspaces of $\bY\bY^T$ and keeps the block iteration well defined. Writing $\bQ$ for the final factor, define
\begin{equation}\label{eq:threshold-definition}
\Tcal_\tau(\bY):=\bQ\Hcal_{\ge\tau/8}(\bQ^T\bY).
\end{equation}
Thus each reconstruction uses matrix--factor products and a compressed singular value decomposition. Its sharp-norm approximation property is established in \Cref{thm:sharp-spectral}, and the computational cost of one reconstruction is as follows.
\begin{proposition}\label[proposition]{prop:reconstruction-cost}
Suppose that $\rank(\bZ)\le r$ and $|\Lambda|=m$. Then
\[
\Tcal_\tau\!\left(\bZ+q^{-1}\Pcal_\Lambda(\bX_\star-\bZ)\right)
\]
can be computed in matrix-free form using $O\!\left((mr+nr^2)\log n\right)$  operations. Evaluating the observed residual costs $O(mr)$ additional operations.
\end{proposition}

\begin{proof}
The proof is given in \Cref{proof:reconstruction-cost}.
\end{proof}

As for the topping rule in the above initialization stage, we adopt the following residual test. For the rank-$r$ iterates, we use the observed residual
\[
R_\ell:=q^{-1/2}\norm{\Pcal_{\Omega^{(\ell+1)}}(\bX_\star-\bZ_\ell)}_{\F},
\]
and stop the initialization at $R_\ell$'s first increase or after $K$ reconstructions. The output is denoted by $\bZ_{\widehat K}$. The procedure is summarized in \Cref{alg:init}.

\begin{algorithm}[H]
\caption{Multiscale initialization}
\label{alg:init}
\begin{algorithmic}[1]
\Statex \textbf{Input:} rank $r$, sampling probability $q$, maximum number of reconstructions $K$, and $\{\Pcal_{\Omega^{(\ell)}}(\bX_\star)\}_{\ell=0}^{K}$.
\State $\bZ_0\gets\bm0$, $\tau_0\gets2q^{-1/2}\norm{\Pcal_{\Omega^{(0)}}(\bX_\star)}_{\F}$, $j\gets\varnothing$.
\For{$\ell=0,\ldots,K-1$}
  \State $\bZ_{\ell+1}\gets\Tcal_{\tau_\ell}\!\left(\bZ_\ell+q^{-1}\Pcal_{\Omega^{(\ell+1)}}(\bX_\star-\bZ_\ell)\right)$.
  \State $\tau_{\ell+1}\gets\tau_\ell/4$.
  \State If $\ell+1<K$, set $R_{\ell+1}\gets q^{-1/2}\norm{\Pcal_{\Omega^{(\ell+2)}}(\bX_\star-\bZ_{\ell+1})}_{\F}$.
  \State \textbf{Stop} and return $\bZ_j$ if $\ell+1<K$, $j\neq\varnothing$, $\rank(\bZ_{\ell+1})=r$, and $R_{\ell+1}>R_j$.
  \State Set $j\gets\ell+1$ if $\ell+1<K$ and $\rank(\bZ_{\ell+1})=r$.
\EndFor
\State \Return $\bZ_{K}$.
\end{algorithmic}
\end{algorithm}

We now specify the observation sets used in the analysis. Set
\[
B:=
\begin{cases}
K+1, & \text{for RGD},\\
K+2, & \text{for RGN},
\end{cases}
\qquad q:=1-(1-p)^{1/B}.
\]
The two choices of the upper bound $K$ are specified in \Cref{sec:main}. For each $(i,j)\in\Omega$, independently draw $\bm b_{ij}\in\{0,1\}^B\setminus\{\bm0\}$ from the product Bernoulli$(q)$ distribution conditioned on being nonzero, and set $\Omega^{(\ell)}:=\{(i,j)\in\Omega:(\bm b_{ij})_\ell=1\}$ for $0\le\ell\le B-1$. Under the unconditional law, $\Omega^{(0)},\ldots,\Omega^{(B-1)}$ are mutually independent Bernoulli$(q)$ subsets with $\Omega=\bigcup_{\ell=0}^{B-1}\Omega^{(\ell)}$, where the subsets may not be disjoint. The first $K+1$ components are used for initialization. For RGN, the last component $\widehat{\Omega}:=\Omega^{(K+1)}$ is reserved for the subsequent iterations.

\begin{remark}\label{rem:initialization-sampling}
The above decomposition is used only in the analysis. In numerical implementation, the same observation set $\Omega$ is used throughout the algorithms, with $q$ replaced by $p$.
\end{remark}

\section{Main Results}\label{sec:main}

In this section, we state the global recovery guarantees for RGD and RGN, and then give a lower bound for ordinary spectral initialization.

\subsection{Global convergence of RGD}

For RGD, let $\bX_0$ be the output of \Cref{alg:init} with the upper bound
$$
K=\ceil{5+\log_4(\kappa\sqrt{nr})}.
$$
Starting from $\bX_0$, let $\{\bX_k\}_{k\ge0}$ be the iterates of \Cref{alg:fixed-rgd} on $\Omega$. The following theorem establishes linear convergence from this initialization.

\begin{theorem}[Global convergence of RGD]\label[theorem]{thm:rgd-global}
Suppose that \Cref{ass:incoherence} holds and $\Omega\sim\operatorname{Bernoulli}(p)$. Let $\bX_0$ be the output of \Cref{alg:init} with $K$ specified above, and let $\{\bX_k\}_{k\ge0}$ be generated by \Cref{alg:fixed-rgd}. If
\[
p\ge C_1\frac{\mu r\log n\log(n\kappa)}{n},
\]
where $C_1>0$ is a sufficiently large absolute constant, then, with probability at least $1-n^{-10}$, the initialization and all subsequent iterates are well defined, and
\[
\norm{\bX_k-\bX_\star}_{\F} \le \left(\frac38\right)^k \norm{\bX_0-\bX_\star}_{\F}, \qquad k\ge0.
\]
\end{theorem}
\begin{proof}
The proof is deferred to \Cref{sec:proof-global-rgd}.
\end{proof}

Thus $O(\mu nr\log n\log(n\kappa))$ observations suffice for exact recovery with high probability. After \Cref{alg:init}, relative Frobenius accuracy $\varepsilon$ is attained within $O(\log(1/\varepsilon))$ RGD iterations. The complete initialization and computational costs are given in \Cref{thm:complexity}.

\subsection{Global convergence of RGN}

For RGN, set
\[
K=\ceil{6+\log_4(\mu r^{3/2}\kappa)}, \qquad q=1-(1-p)^{1/(K+2)}, \qquad \bar K:=\ceil{\log_2\log_2(4n)}.
\]
We first run \Cref{alg:init} with this upper bound $K$ using the RGN subsets specified in \Cref{sec:sample-split}, and then run \Cref{alg:fixed} on $\widehat{\Omega}$. The following theorem establishes global convergence and eventual quadratic convergence of the resulting RGN iterates.

\begin{theorem}[Global convergence of RGN]\label[theorem]{thm:end}
Suppose that \Cref{ass:incoherence} holds and $\Omega\sim\operatorname{Bernoulli}(p)$. Let $\bX_0$ denote the output of \Cref{alg:init} with $K$ specified above, and let $\{\bX_k\}_{k\ge0}$ be generated by \Cref{alg:fixed} on $\widehat{\Omega}$. If
\[
p\ge C_2\frac{\mu r\log n\log(2\mu r\kappa)}{n},
\]
where $C_2>0$ is a sufficiently large absolute constant, then, with probability at least $1-n^{-10}$, the initialization and all subsequent iterates are well defined. Moreover,
\[
\norm{\bX_k-\bX_\star}_{\sharp} \le \frac{\sigma_r(\bX_\star)}{280\mu r} \left(\frac7{25}\right)^{2^k}, \qquad 0\le k\le\bar K,
\]
and
\[
\norm{\bX_{k+1}-\bX_\star}_{\F} \le 4\frac{\norm{\bX_k-\bX_\star}_{\F}^2} {\sqrt q\,\sigma_r(\bX_\star)}, \qquad k\ge\bar K.
\]
\end{theorem}

\begin{proof}
The proof is deferred to \Cref{sec:proof-global-rgn}.
\end{proof}

The RGN  can also be regularized and admits a similar global recovery guarantee with the eventual quadratic rate.

\begin{corollary}\label{cor:regularized-rgn}
Under the conditions of \Cref{thm:end}, let $\{\bX_k\}_{0\le k\le\bar K}$ be the RGN iterates in \Cref{thm:end}. Starting from $\bX_{\bar K}$, replace \Cref{alg:fixed} by \Cref{alg:regularized-rgn} on $\widehat{\Omega}$. Then, with probability at least $1-n^{-10}$, all subsequent iterates are well defined and satisfy
\[
\norm{\bX_{k+1}-\bX_\star}_{\F} \le \bar C_2 \frac{\norm{\bX_k-\bX_\star}_{\F}^2} {\sqrt q\,\sigma_r(\bX_\star)}, \qquad k\ge\bar K,
\]
where $\bar C_2>0$ is an absolute constant. 
\end{corollary}

\begin{proof}
The proof is deferred to \Cref{sec:proof-reg-rgn}.
\end{proof}

Thus $O(\mu nr\log n\log(2\mu r\kappa))$ observations suffice for both RGN and its regularized variant to converge to $\bX_\star$ with high probability, with quadratic convergence after finitely many iterations. Relative Frobenius accuracy $\varepsilon$ is attained within $O(\log\log(1/\varepsilon))$ subsequent RGN iterations. The computational cost additionally depends on the CG iteration counts $J_k$; see \Cref{thm:complexity}.

\subsection{Lower bound for ordinary spectral initialization}

We next explain why the multiscale initialization is needed in place of the ordinary spectral estimator. The latter may require an additional factor $\kappa^2$ in the sampling probability to reach the sharp-norm neighborhood required by the local convergence results. For $\bY\in\R^{n\times n}$, let
$$
\Hcal_r(\bY)\in\arg\min_{\rank(\bZ)\le r}\norm{\bY-\bZ}_{\F}
$$
denote a fixed best rank-$r$ approximation. The following theorem shows that a sampling probability of order $\mu\kappa^2r/n$ is necessary over an explicit family of incoherent matrices, with the initialization error measured in the sharp norm \eqref{eq:sharp-main}.

\begin{theorem}
\label[theorem]{thm:necessary-scales}
Let $r,s\ge2$ be integers with $rs\le n$, set $\mu=n/(rs)$, and let $1\le\kappa\le\sqrt{s}$. There exist an absolute constant $C_3>0$ and a rank-$r$ matrix $\bX_\star$ satisfying \Cref{ass:incoherence} with coherence parameter $\mu$ and condition number $\kappa$ such that, for every
\[
0<q<C_3\frac{\mu\kappa^2r}{n},
\]
the spectral estimate
\[
\bX_0=\Hcal_r\!\left(q^{-1}\Pcal_\Lambda(\bX_\star)\right), \qquad \Lambda\sim\operatorname{Bernoulli}(q),
\]
satisfies
\[
\Pr\left\{ \sharpnorm{\bX_0-\bX_\star} \ge\frac12\sigma_r(\bX_\star) \right\} \ge\frac12.
\]
\end{theorem}

\begin{proof}
The proof is deferred to \Cref{sec:proof-necessary-scales}.
\end{proof}

The theorem identifies the $\mu\kappa^2r/n$ sampling scale for the ordinary spectral estimator over this family. In \Cref{alg:init}, reconstruction is applied to successive residuals whose sharp error decreases geometrically.

\section{Proof Framework: Local Convergence and Initialization}\label{sec:interfaces}

The global results are obtained by combining local convergence of the fixed-sample iterations with the multiscale initialization. We first state the local results for RGD and RGN and then show that \Cref{alg:init} reaches their respective hypotheses.

\subsection{Local convergence}
\label{sec:local-condition-geometry}

The two local results use different error controls:
\begin{subequations}\label{eq:local-scales}
\begin{align}
\norm{\bX-\bX_\star}_{\F} &\le\frac1{128}\sqrt{p}\,\sigma_r(\bX_\star), &&\text{for RGD}, \label{eq:rgd-local-scale}\\
\sharpnorm{\bX-\bX_\star} &\le\frac{\sigma_r(\bX_\star)}{1000\mu r}, &&\text{for RGN}. \label{eq:rgn-local-scale}
\end{align}
\end{subequations}
Thus RGD is controlled in a Frobenius neighborhood whose radius depends on the sampling probability, whereas RGN requires the stronger sharp-norm control.

A common geometric identity underlying both analyses is the exactness of the graph retraction for the population tangent correction. If $\sharpnorm{\bX-\bX_\star}\le\frac14\sigma_r(\bX_\star)$, then a later established \Cref{lem:population-exact} shows that the graph core is invertible and
\begin{equation}\label{eq:population-exact-main}
\operatorname{Retr}_{\bX}\!\left(\Pcal_{T_{\bX}}(\bX_\star-\bX)\right)=\bX_\star.
\end{equation}
Hence the local analysis reduces to controlling the error introduced by sampling. For a nonterminal RGD step, let $\bm\xi=\Pcal_{T_{\bX}}p^{-1}\Pcal_\Omega(\bX_\star-\bX)$ and let $\alpha$ be the stepsize in \eqref{eq:rgd-exact-line-search}. Then
\[
\alpha\bm\xi-\Pcal_{T_{\bX}}(\bX_\star-\bX) = (\alpha-1)\Pcal_{T_{\bX}}(\bX_\star-\bX) +\alpha\Pcal_{T_{\bX}}(p^{-1}\Pcal_\Omega-\Ical)(\bX_\star-\bX).
\]
The tangent sampling isometry controls both terms, which are first order in the error. For RGN, if $\bm\xi$ denotes the tangent correction, then
\begin{equation}\label{eq:rgn-correction-main}
\Pcal_{T_{\bX}}q^{-1}\Pcal_{\widehat{\Omega}}\Pcal_{T_{\bX}} \left(\bm\xi-\Pcal_{T_{\bX}}(\bX_\star-\bX)\right) = \Pcal_{T_{\bX}}q^{-1}\Pcal_{\widehat{\Omega}}\Pcal_{T_{\bX}^\perp} (\bX_\star-\bX).
\end{equation}
The right-hand side is driven by the normal component, which is quadratic in the error. Indeed, under the local sharp-norm condition, a later established \Cref{lem:normal} gives
\begin{equation}\label{eq:normal-quadratic-main}
\sharpnorm{\Pcal_{T_{\bX}^\perp}(\bX_\star-\bX)} \le\frac{16}{3} \frac{\sharpnorm{\bX-\bX_\star}^{2}}{\sigma_r(\bX_\star)}.
\end{equation}
The two mechanisms are summarized in \Cref{fig:local-error-equations}.

\begin{figure}[H]
\centering
\resizebox{0.98\linewidth}{!}{%
\begin{tikzpicture}[node distance=5mm and 7mm]
\node[diagramnode, text width=4.1cm] (rgd1) {True-tangent\\sampling isometry\\
\Cref{lem:true-tangent-rip}};
\node[diagramnode, text width=4.1cm, right=of rgd1] (rgd2) {Uniform tangent transfer\\and exact line search\\
\footnotesize{$\alpha\bm\xi-\Pcal_{T_{\bX}}(\bX_\star-\bX)$}\\
\Cref{lem:det-transfer}};
\node[diagramnode, text width=4.1cm, right=of rgd2] (rgd3) {Linear RGD\\ contraction};
\node[diagramnode, text width=4.1cm, below=8mm of rgd1] (rgn1) {Quadratic normal component\\
\footnotesize{$\sharpnorm{\Pcal_{T_{\bX}^\perp}(\bX_\star-\bX)}
\le(16/3)\sharpnorm{\bX-\bX_\star}^2/\sigma_r(\bX_\star)$}\\
\Cref{lem:normal}};
\node[diagramnode, text width=4.1cm, right=of rgn1] (rgn2) {Sampled tangent correction\\and local conditioning\\
\footnotesize{$\Pcal_{T_{\bX}}q^{-1}\Pcal_{\widehat{\Omega}}
\Pcal_{T_{\bX}^\perp}(\bX_\star-\bX)$}\\
\Cref{lem:local-conditioning}};
\node[diagramnode, text width=4.1cm, right=of rgn2] (rgn3) {Quadratic RGN convergence};
\draw[diagramarrow] (rgd1)--(rgd2);
\draw[diagramarrow] (rgd2)--(rgd3);
\draw[diagramarrow] (rgn1)--(rgn2);
\draw[diagramarrow] (rgn2)--(rgn3);
\end{tikzpicture}}
\caption{Local mechanisms for RGD and RGN. For RGD, the tangent sampling isometry controls the stepsize and the first-order tangent perturbation. For RGN, the sampled correction is driven by the quadratic normal component. Finite leave-one-out control brings the iterates into a Frobenius neighborhood where quadratic convergence follows deterministically.}
\label{fig:local-error-equations}
\end{figure}
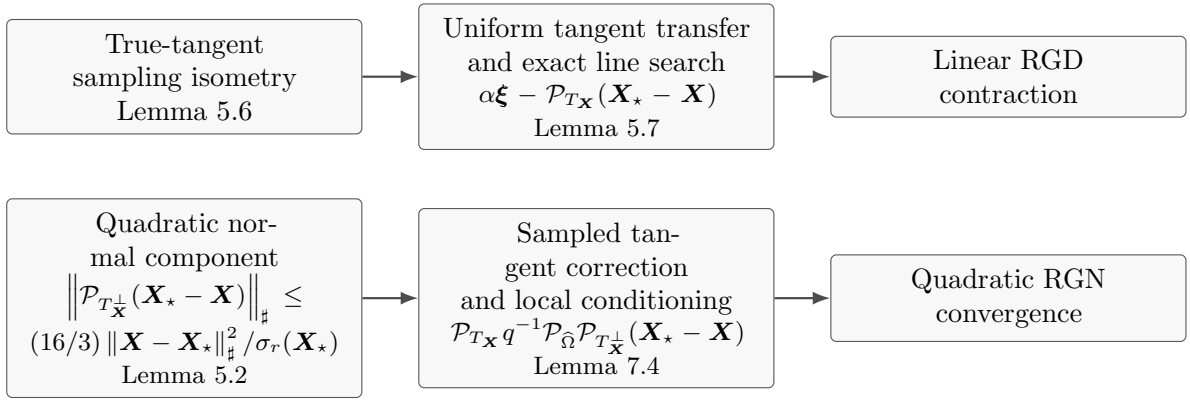

For RGD, the high-probability event is uniform over the entire Frobenius neighborhood in \eqref{eq:rgd-local-scale}. Consequently, the initial point need not be independent of the observation set in the analysis.

\begin{theorem}[Local convergence of RGD]\label[theorem]{thm:fixed-rgd-local}
Suppose that \Cref{ass:incoherence} holds and $\Omega\sim\operatorname{Bernoulli}(p)$. If $p\ge C_4\frac{\mu r\log n}{n}$, where $C_4>0$ is a sufficiently large absolute constant, then, with probability at least $1-n^{-10}/48$, the following holds simultaneously for every $\bX_0\in\Mcal_r$ satisfying
\[
\norm{\bX_0-\bX_\star}_{\F} \le\frac1{128}\sqrt{p}\,\sigma_r(\bX_\star):
\]
\Cref{alg:fixed-rgd} on $\Omega$ is well defined, and its iterates satisfy
\[
\norm{\bX_k-\bX_\star}_{\F} \le \left(\frac38\right)^k \norm{\bX_0-\bX_\star}_{\F}, \qquad k\ge0.
\]
\end{theorem}

\begin{proof}
The proof is deferred to \Cref{sec:proof-fixed-rgd}.
\end{proof}

For RGN, the initial point need be independent of $\widehat{\Omega}$ in the analysis. This independence permits the leave-one-out argument used to control the initial RGN iterates, after which the quadratic Frobenius recursion applies.

\begin{theorem}[Local convergence of RGN]\label[theorem]{thm:fixed}
Suppose that \Cref{ass:incoherence} holds and $\widehat{\Omega}\sim\operatorname{Bernoulli}(q)$. Let $\bX_0\in\Mcal_r$ be independent of $\widehat{\Omega}$ and satisfy $\sharpnorm{\bX_0-\bX_\star}
\le\frac{\sigma_r(\bX_\star)}{1000\mu r}$. Set $\bar K:=\ceil{\log_2\log_2(4n)}$. If $q\ge C_4\frac{\mu r\log n}{n}$, then, conditional on $\bX_0$, with probability at least $1-n^{-10}/12$ over $\widehat{\Omega}$, the iterates of \Cref{alg:fixed} are well defined and satisfy, simultaneously for all $0\le k\le\bar K$,
\[
\sharpnorm{\bX_k-\bX_\star} \le\frac{\sigma_r(\bX_\star)}{280\mu r} \left(\frac7{25}\right)^{2^k}.
\]
Moreover, for every $k\ge\bar K$,
\[
\norm{\bX_{k+1}-\bX_\star}_{\F} \le 4\frac{\norm{\bX_k-\bX_\star}_{\F}^2} {\sqrt q\,\sigma_r(\bX_\star)}.
\]
\end{theorem}

\begin{proof}
The proof is deferred to \Cref{sec:proof-local-rgn}.
\end{proof}

\subsection{Multiscale initialization}\label{sec:thm-ini}

We now show that \Cref{alg:init} reaches the two local conditions above. For a fixed approximation $\bZ$ and an independent observation subset $\Lambda$, the residual correction satisfies
\begin{equation}\label{eq:spectral-initial-definitions}
\bZ+q^{-1}\Pcal_\Lambda(\bX_\star-\bZ) =\bX_\star+(q^{-1}\Pcal_\Lambda-\Ical)(\bX_\star-\bZ).
\end{equation}
Thus its sampling perturbation is determined by the current error, rather than by the largest singular value of $\bX_\star$. The first result bounds one spectral reconstruction.

\begin{theorem}\label[theorem]{thm:sharp-spectral}
Suppose that \Cref{ass:incoherence} holds. Let $\bZ$ be fixed with $\rank(\bZ)\le r$ and $\sharpnorm{\bZ-\bX_\star}\le\tau$, where $\tau>0$. Let $\Lambda\sim\operatorname{Bernoulli}(q)$ be independent of the Gaussian matrix used in $\Tcal_\tau$. If
\[
q\ge C_5\frac{\mu r\log n}{n},
\]
where $C_5>0$ is a sufficiently large absolute constant, then
\[
\bZ^+=\Tcal_\tau\!\left(\bZ+q^{-1}\Pcal_\Lambda(\bX_\star-\bZ)\right)
\]
satisfies, with probability at least $1-n^{-12}/24$,
\begin{equation}\label{eq:sharp-spectral-total}
\rank(\bZ^+)\le r, \qquad \sharpnorm{\bZ^+-\bX_\star}\le\frac14\tau.
\end{equation}
\end{theorem}

\begin{proof}
The proof is deferred to \Cref{sec:proof-sharp-spectral}.
\end{proof}

The proof controls the sampling perturbation and its products with the true singular spaces before estimating the reconstructed matrix. The finite computation in \eqref{eq:threshold-iteration} attains the required accuracy without a gap between adjacent singular values. The supporting estimates are proved in \Cref{app:fresh,app:spectral}. 

\begin{theorem}[Multiscale initialization]\label[theorem]{thm:init}
Suppose that \Cref{ass:incoherence} holds and $1\le K\le n$. Let $\{\bZ_\ell\}_{\ell=0}^{K}$ denote the complete sequence in \eqref{eq:multiscale-update}, and let $\bZ_{\widehat K}$ be the output of \Cref{alg:init}. If $q\ge C_5\mu r\log n/n$, then, with probability at least $1-n^{-10}/3$, \Cref{alg:init} is well defined, $\widehat K\le K$, and
\begin{equation}\label{eq:init-invariant}
\norm{\bX_\star}_{\F}\le\tau_0\le3\norm{\bX_\star}_{\F}, \qquad \rank(\bZ_\ell)\le r, \qquad \sharpnorm{\bZ_\ell-\bX_\star}\le4^{-\ell}\tau_0, \quad 0\le\ell\le K.
\end{equation}
In particular, $K=\ceil{5+\log_4(\kappa\sqrt{nr})}$ gives a rank-$r$ output satisfying
\begin{equation}\label{eq:init-rgd-entrance}
\norm{\bZ_{\widehat K}-\bX_\star}_{\F}\le\frac1{128}\sqrt q\,\sigma_r(\bX_\star),
\end{equation}
whereas $K=\ceil{6+\log_4(\mu r^{3/2}\kappa)}$ gives a rank-$r$ output satisfying
\begin{equation}\label{eq:init-rgn-entrance}
\sharpnorm{\bZ_{\widehat K}-\bX_\star}\le\frac{\sigma_r(\bX_\star)}{1000\mu r}.
\end{equation}
\end{theorem}

\begin{proof}
The proof is deferred to \Cref{sec:proof-finite-rgd}.
\end{proof}

The two choices of $K$ allow at most $O(\log(n\kappa))$ and $O(\log(2\mu r\kappa))$ reconstructions, respectively. For RGD, $q\le p$ and \eqref{eq:init-rgd-entrance} imply the hypothesis of \Cref{thm:fixed-rgd-local}. For RGN, \eqref{eq:init-rgn-entrance} gives the sharp-norm hypothesis of \Cref{thm:fixed}, while the initialization, including the stopping decision, is independent of the reserved set $\widehat\Omega$. The following proposition justifies the residual test, which shows a residual increase cannot cause a return before the corresponding local convergence condition is satisfied.

\begin{proposition}\label[proposition]{prop:no-premature-rebound}
 Under the assumptions of \Cref{thm:init}, with probability at least $1-n^{-10}/6$, if \Cref{alg:init} returns at the residual test, then its output satisfies
\begin{equation}\label{eq:residual-stop-accuracy}
\norm{\bZ_{\widehat K}-\bX_\star}_{\F} \le\frac{\sigma_r(\bX_\star)}{1024n}.
\end{equation}   
\end{proposition}
\begin{proof}
The proof is given in \Cref{sec:proof-residual-stopping}.
\end{proof}

\section{Local Convergence of RGD}
\label{sec:local-rgd}

We first establish the geometric estimates used in the convergence and initialization arguments. We then prove a sampling estimate that holds uniformly over a Frobenius neighborhood of $\bX_\star$ and use it to prove \Cref{thm:fixed-rgd-local}. The initialization results are proved in \Cref{sec:initialization}.

\subsection{Deterministic geometry}

We begin with bounds for the singular subspaces and the normal component of the error, followed by perturbation bounds for the graph retraction. The proofs are given in \Cref{app:geometry}.

For $\bX=\bU\bm\Sigma\bV^T\in\Mcal_r$, define
\begin{equation}\label{eq:target-graph-factors}
\bG_{\bX}:=\bU^T\bX_\star\bV, \qquad \bL_{\bX}:=(\bI-\bU\bU^T)\bX_\star\bV, \qquad \bR_{\bX}:=(\bI-\bV\bV^T)\bX_\star^T\bU.
\end{equation}
We also write
\begin{equation}\label{eq:tangent-normal-components}
\bm t_{\bX}:=\Pcal_{T_{\bX}}(\bX_\star-\bX), \qquad \bN_{\bX}:=\Pcal_{T_{\bX}^\perp}(\bX_\star-\bX)
\end{equation}
for the tangent and normal components of the error. Here $\bG_{\bX}$ is the representation of $\bX_\star$ in the current singular bases, while $\bL_{\bX}$ and $\bR_{\bX}$ describe its components outside the current singular subspaces. The following lemma bounds the incoherence of these bases and the variation of the associated projectors. Its operator-norm estimates follow the argument in \cite[Lemma~4.1]{WeiCaiChanLeung2020}, while the rowwise estimates use the sharp norm.

\begin{lemma}
\label[lemma]{lem:current-incoherence}
Suppose that \Cref{ass:incoherence} holds. Let $\bX=\bU\bS\bV^T\in\Mcal_r$ be a compact orthonormal factorization, and set $e_\sharp:=\sharpnorm{\bX-\bX_\star}$. If $e_\sharp\le\sigma_r(\bX_\star)/8$, then
\[
\max\left\{\norm{\bU}_{2,\infty},\norm{\bV}_{2,\infty}\right\} \le2\sqrt{\frac{\mu r}{n}}.
\]
The singular-space projectors satisfy the rowwise bounds
\[
\max\left\{ \norm{\bU\bU^T-\bU_\star\bU_\star^T}_{2,\infty}, \norm{\bV\bV^T-\bV_\star\bV_\star^T}_{2,\infty} \right\} \le\frac{32}{7}\sqrt{\frac{\mu r}{n}}\frac{e_\sharp}{\sigma_r(\bX_\star)}.
\]
Moreover,
\[
\norm{\Pcal_{T_{\bX}}-\Pcal_{T_{\bX_\star}}}_{\F\to\F} \le\frac{16}{7}\frac{e_\sharp}{\sigma_r(\bX_\star)}.
\]
\end{lemma}

\begin{proof}
The proof is deferred to \Cref{sec:proof-current-incoherence}.
\end{proof}

The next lemma gives a sharp-norm counterpart of the quadratic normal-component estimate in \cite[Lemma~4.1]{WeiCaiChanLeung2020}, together with the graph-factor representation used below.

\begin{lemma}\label[lemma]{lem:normal}
Suppose that \Cref{ass:incoherence} holds. Let $\bX=\bU\bS\bV^T\in\Mcal_r$ be a compact orthonormal factorization, and set $e_\sharp:=\sharpnorm{\bX-\bX_\star}$. If $e_\sharp\le\sigma_r(\bX_\star)/8$, then $\bG_{\bX}$ is invertible, and the graph factors in \eqref{eq:target-graph-factors} satisfy
\begin{equation}\label{eq:normal-factorization-main}
\bN_{\bX}=\bL_{\bX}\bG_{\bX}^{-1}\bR_{\bX}^T,
\end{equation}
and
\begin{equation}\label{eq:normal-factor-bounds-main}
\begin{aligned}
\norm{\bG_{\bX}^{-1}}_{\op} \le\frac4{3\sigma_r(\bX_\star)},\quad \max\left\{\norm{\bL_{\bX}}_{\op},\norm{\bR_{\bX}}_{\op}\right\} &\le e_\sharp,\\
\max\left\{\norm{\bL_{\bX}}_{2,\infty},\norm{\bR_{\bX}}_{2,\infty}\right\} &\le4\sqrt{\frac{\mu r}{n}}\,e_\sharp.&&
\end{aligned}
\end{equation}
Moreover, the normal component satisfies
\[
\sharpnorm{\bN_{\bX}} \le\frac{16}{3}\frac{e_\sharp^2}{\sigma_r(\bX_\star)}.
\]
Its row, column, and entrywise norms satisfy
\[
\max\left\{ \norm{\bN_{\bX}}_{2,\infty}, \norm{\bN_{\bX}^T}_{2,\infty} \right\} \le\frac{16}{3}\sqrt{\frac{\mu r}{n}}\frac{e_\sharp^2}{\sigma_r(\bX_\star)}, \qquad \norm{\bN_{\bX}}_\infty \le\frac{64\mu r}{3n}\frac{e_\sharp^2}{\sigma_r(\bX_\star)}.
\]
\end{lemma}

\begin{proof}
The proof is deferred to \Cref{sec:proof-normal}.
\end{proof}

The following identity shows that the population tangent correction recovers $\bX_\star$ exactly. It is the local inverse formula for the orthographic retraction \cite[Sec.~3.2.3]{AbsilOseledets2015}; its proof verifies the required core invertibility under the stated error bound.
\begin{lemma}\label[lemma]{lem:population-exact}
Let $\bX\in\Mcal_r$ satisfy $\sharpnorm{\bX-\bX_\star}\le\sigma_r(\bX_\star)/4$. Then $\operatorname{Retr}_{\bX}(\bm t_{\bX})=\bX_\star$.
\end{lemma}

\begin{proof}
The proof is deferred to \Cref{sec:proof-population-exact}.
\end{proof}

We next bound the effect of perturbing the population tangent correction.

\begin{lemma}
\label[lemma]{lem:graph-correction}
Suppose that \Cref{ass:incoherence} holds. Let $\bX\in\Mcal_r$, set $e_\sharp:=\sharpnorm{\bX-\bX_\star}$, and let $\bm\eta\in T_{\bX}\Mcal_r$. If $e_\sharp\le\sigma_r(\bX_\star)/8$ and $\sharpnorm{\bm\eta}\le\sigma_r(\bX_\star)/16$, then $\operatorname{Retr}_{\bX}(\bm t_{\bX}+\bm\eta)$ is well defined and
\[
\sharpnorm{\operatorname{Retr}_{\bX}(\bm t_{\bX}+\bm\eta) -\bX_\star-\bm\eta} \le13\frac{e_\sharp\sharpnorm{\bm\eta}}{\sigma_r(\bX_\star)} +6\frac{\sharpnorm{\bm\eta}^2}{\sigma_r(\bX_\star)}.
\]
\end{lemma}

\begin{proof}
The proof is deferred to \Cref{sec:proof-graph-correction}.
\end{proof}

The following Frobenius estimates will be used for local RGD convergence and for the eventual quadratic convergence of RGN. Part~(i) is a variant of \cite[Lemma~4.1]{WeiCaiChanLeung2020}.

\begin{lemma}
\label[lemma]{lem:normal-F}
Let $\bX\in\Mcal_r$.  Then the following statements hold.

\emph{(i)} Suppose that $\norm{\bX-\bX_\star}_{\F}\le\sigma_r(\bX_\star)/4$.  Then
\[
\norm{\Pcal_{T_{\bX}^\perp}(\bX_\star-\bX)}_{\F} \le2\frac{\norm{\bX-\bX_\star}_{\F}^2}{\sigma_r(\bX_\star)}, \qquad \norm{\Pcal_{T_{\bX}}-\Pcal_{T_{\bX_\star}}}_{\F\to\F} \le3\frac{\norm{\bX-\bX_\star}_{\F}}{\sigma_r(\bX_\star)}.
\]

\emph{(ii)} Let $\bm\eta\in T_{\bX}\Mcal_r$.  Suppose that $\norm{\bX-\bX_\star}_{\F}\le\sigma_r(\bX_\star)/40$ and $\norm{\bm\eta}_{\F}\le\sigma_r(\bX_\star)/320$.  Then the graph retraction is well defined and its nonlinear remainder satisfies
\[
\norm{\operatorname{Retr}_{\bX}(\bm t_{\bX}+\bm\eta) -\bX_\star-\bm\eta}_{\F} \le\frac{11}{5} \frac{\norm{\bX-\bX_\star}_{\F}\norm{\bm\eta}_{\F}} {\sigma_r(\bX_\star)} +\frac{11}{10}\frac{\norm{\bm\eta}_{\F}^2}{\sigma_r(\bX_\star)}.
\]
\end{lemma}

\begin{proof}
The proof is deferred to \Cref{sec:proof-normal-F}.
\end{proof}

We next establish the standard sampling isometry at the tangent space of $\bX_\star$~\cite{Chen2015,WeiCaiChanLeung2020} and transfer it uniformly to nearby tangent spaces.

\subsection{Sampling isometry and uniform transfer}

\begin{lemma}
\label[lemma]{lem:true-tangent-rip}
Suppose that \Cref{ass:incoherence} holds, and let $\Lambda\sim\operatorname{Bernoulli}(p_\Lambda)$. If $p_\Lambda\ge c_1\mu r\log n/n$, where $c_1>0$ is a sufficiently large absolute constant, then, with probability at least $1-n^{-10}/48$,
\[
\norm{\Pcal_{T_{\bX_\star}}p_\Lambda^{-1}\Pcal_\Lambda\Pcal_{T_{\bX_\star}} -\Pcal_{T_{\bX_\star}}}_{\F\to\F} \le\frac1{16}.
\]
\end{lemma}

\begin{proof}
The proof is deferred to \Cref{sec:proof-true-tangent-rip}.
\end{proof}

We next transfer the preceding isometry to nearby tangent spaces. A related local tangent-space estimate appears in \cite[Lemma~4.2]{WeiCaiChanLeung2020}. The following deterministic lemma holds simultaneously throughout the stated Frobenius neighborhood, so the matrix $\bX$ may depend on $\Lambda$.

\begin{lemma}\label[lemma]{lem:det-transfer}
Let $\Lambda$ be an index set and let $0<p_\Lambda\le1$. Suppose that
\begin{equation}\label{eq:det-transfer-hypothesis}
\norm{\Pcal_{T_{\bX_\star}}p_\Lambda^{-1}\Pcal_\Lambda \Pcal_{T_{\bX_\star}} -\Pcal_{T_{\bX_\star}}}_{\F\to\F} \le\frac1{16}.
\end{equation}
Then, simultaneously for all $\bX\in\Mcal_r$ satisfying
\[
\norm{\bX-\bX_\star}_{\F} \le\frac1{40}\sqrt{p_\Lambda}\,\sigma_r(\bX_\star),
\]
and all $\bm\zeta\in T_{\bX}\Mcal_r$, we have
\begin{equation}\label{eq:fixed-rgd-transfer}
\frac23\norm{\bm\zeta}_{\F}^2 \le \inner{\bm\zeta}{p_\Lambda^{-1}\Pcal_\Lambda\bm\zeta} \le \frac54\norm{\bm\zeta}_{\F}^2.
\end{equation}
If, in addition, $\norm{\bX-\bX_\star}_{\F}\le\sqrt{p_\Lambda}\,\sigma_r(\bX_\star)/128$, then
\begin{equation}\label{eq:fixed-rgd-transfer-small}
\frac78\norm{\bm\zeta}_{\F}^2 \le\inner{\bm\zeta}{p_\Lambda^{-1}\Pcal_\Lambda\bm\zeta} \le\frac98\norm{\bm\zeta}_{\F}^2.
\end{equation}
\end{lemma}

\begin{proof}
Fix any $\bX\in\Mcal_r$ satisfying $\norm{\bX-\bX_\star}_{\F}\le \sqrt{p_\Lambda}\sigma_r(\bX_\star)/40$ and any $\bm\zeta\in T_{\bX}\Mcal_r$.  By homogeneity, it suffices to consider $\norm{\bm\zeta}_{\F}=1$. The assumed neighborhood is contained in $\norm{\bX-\bX_\star}_{\F}\le\sigma_r(\bX_\star)/40$, and hence \Cref{lem:normal-F} applies. Since $\Pcal_{T_{\bX}}\bm\zeta=\bm\zeta$, we obtain
\begin{align}
\norm{\bm\zeta-\Pcal_{T_{\bX_\star}}\bm\zeta}_{\F} &= \norm{ \bigl(\Pcal_{T_{\bX}}-\Pcal_{T_{\bX_\star}}\bigr)\bm\zeta }_{\F} \notag\\
&\le \norm{\Pcal_{T_{\bX}}-\Pcal_{T_{\bX_\star}}}_{\F\to\F} \notag\\
&\le 3\frac{\norm{\bX-\bX_\star}_{\F}} {\sigma_r(\bX_\star)} \le \frac3{40}\sqrt{p_\Lambda},\notag\\
\norm{\Pcal_{T_{\bX_\star}}\bm\zeta}_{\F} &\ge 1-\norm{\bm\zeta-\Pcal_{T_{\bX_\star}}\bm\zeta}_{\F} \ge 1-\frac3{40}\sqrt{p_\Lambda}. \label{eq:det-transfer-tangent-motion}
\end{align}

Applying \eqref{eq:det-transfer-hypothesis} to $\Pcal_{T_{\bX_\star}}\bm\zeta\in T_{\bX_\star}\Mcal_r$ gives
\[
\sqrt{\frac{15}{16}} \norm{\Pcal_{T_{\bX_\star}}\bm\zeta}_{\F} \le p_\Lambda^{-1/2} \norm{\Pcal_\Lambda\Pcal_{T_{\bX_\star}}\bm\zeta}_{\F} \le \sqrt{\frac{17}{16}} \norm{\Pcal_{T_{\bX_\star}}\bm\zeta}_{\F}.
\]

By \eqref{eq:det-transfer-hypothesis}, \eqref{eq:det-transfer-tangent-motion}, the triangle inequality, and $\norm{\Pcal_\Lambda}_{\F\to\F}\le1$, we have
\begin{align*}
p_\Lambda^{-1/2}\norm{\Pcal_\Lambda\bm\zeta}_{\F} &\ge p_\Lambda^{-1/2} \norm{\Pcal_\Lambda\Pcal_{T_{\bX_\star}}\bm\zeta}_{\F} - p_\Lambda^{-1/2} \norm{ \Pcal_\Lambda \bigl(\bm\zeta-\Pcal_{T_{\bX_\star}}\bm\zeta\bigr) }_{\F}\\
&\ge \sqrt{\frac{15}{16}} \norm{\Pcal_{T_{\bX_\star}}\bm\zeta}_{\F} - p_\Lambda^{-1/2} \norm{\bm\zeta-\Pcal_{T_{\bX_\star}}\bm\zeta}_{\F}\\
&\ge\sqrt{\frac23}.
\end{align*}

By \eqref{eq:det-transfer-hypothesis}, \eqref{eq:det-transfer-tangent-motion}, $\norm{\Pcal_{T_{\bX_\star}}\bm\zeta}_{\F}\le1$, and the triangle inequality, we also have
\begin{align*}
p_\Lambda^{-1/2}\norm{\Pcal_\Lambda\bm\zeta}_{\F} &\le p_\Lambda^{-1/2} \norm{\Pcal_\Lambda\Pcal_{T_{\bX_\star}}\bm\zeta}_{\F} + p_\Lambda^{-1/2} \norm{ \Pcal_\Lambda \bigl(\bm\zeta-\Pcal_{T_{\bX_\star}}\bm\zeta\bigr) }_{\F}\\
&\le \sqrt{\frac{17}{16}} \norm{\Pcal_{T_{\bX_\star}}\bm\zeta}_{\F} + p_\Lambda^{-1/2} \norm{\bm\zeta-\Pcal_{T_{\bX_\star}}\bm\zeta}_{\F}\\
&\le\sqrt{\frac54}.
\end{align*}
Finally, since $\Pcal_\Lambda$ is an orthogonal projector, the above inequalities  give
\[
\frac23\le \inner{\bm\zeta}{p_\Lambda^{-1}\Pcal_\Lambda\bm\zeta} =p_\Lambda^{-1}\norm{\Pcal_\Lambda\bm\zeta}_{\F}^2 \le\frac54.
\]
Rescaling the $\bm\zeta$ proves \eqref{eq:fixed-rgd-transfer}. Since $\bX$ was arbitrary in the Frobenius neighborhood specified in \Cref{lem:det-transfer}, \eqref{eq:fixed-rgd-transfer} holds simultaneously throughout that neighborhood.

To prove \eqref{eq:fixed-rgd-transfer-small}, suppose that $\norm{\bX-\bX_\star}_{\F}\le\sqrt{p_\Lambda}\,\sigma_r(\bX_\star)/128$ and $\norm{\bm\zeta}_{\F}=1$. By \Cref{lem:normal-F} and orthogonality, we have
\begin{align*}
\norm{\bm\zeta-\Pcal_{T_{\bX_\star}}\bm\zeta}_{\F} &\le\frac3{128}\sqrt{p_\Lambda},\\
\norm{\Pcal_{T_{\bX_\star}}\bm\zeta}_{\F}^2 &=1-\norm{\bm\zeta-\Pcal_{T_{\bX_\star}}\bm\zeta}_{\F}^2 \ge1-\frac9{128^2}.
\end{align*}
The preceding comparison with the true tangent space now gives
\begin{align*}
\sqrt{\frac78}\le p_\Lambda^{-1/2}\norm{\Pcal_\Lambda\bm\zeta}_{\F} \le\sqrt{\frac98}.
\end{align*}
Squaring and rescaling proves \eqref{eq:fixed-rgd-transfer-small}.
\end{proof}
\subsection{Proof of Theorem~\ref{thm:fixed-rgd-local}}
\label{sec:proof-fixed-rgd}

\begin{proof}
For $C_4\ge c_1$, \Cref{lem:true-tangent-rip} implies that, with probability at least $1-n^{-10}/48$,
\[
\norm{\Pcal_{T_{\bX_\star}}p^{-1}\Pcal_\Omega\Pcal_{T_{\bX_\star}}-\Pcal_{T_{\bX_\star}}}_{\F\to\F}\le\frac1{16}.
\]
On the event supplied by \Cref{lem:true-tangent-rip}, fix any $\bX\in\Mcal_r$ satisfying
\[
\norm{\bX-\bX_\star}_{\F}\le\frac1{128}\sqrt{p}\,\sigma_r(\bX_\star),
\]
and set $\bE:=\bX_\star-\bX$ and $\bm\xi:=\Pcal_{T_{\bX}}p^{-1}\Pcal_\Omega\bE$. Applying \eqref{eq:fixed-rgd-transfer-small} and using the self-adjointness of $\Pcal_{T_{\bX}}p^{-1}\Pcal_\Omega\Pcal_{T_{\bX}}$, we obtain
\begin{equation}\label{eq:local-rgd-operator-bounds}
\norm{\Pcal_{T_{\bX}}p^{-1}\Pcal_\Omega\Pcal_{T_{\bX}}-\Pcal_{T_{\bX}}}_{\F\to\F}\le\frac18, \qquad p^{-1/2}\norm{\Pcal_\Omega\Pcal_{T_{\bX}}}_{\F\to\F}\le\frac3{2\sqrt2}.
\end{equation}
If $\bm\xi\ne\bm0$, then \eqref{eq:fixed-rgd-transfer-small} gives
\[
p^{-1}\norm{\Pcal_\Omega\bm\xi}_{\F}^2 =\inner{\bm\xi}{p^{-1}\Pcal_\Omega\bm\xi} \ge\frac78\norm{\bm\xi}_{\F}^2>0.
\]
Thus the stepsize in \eqref{eq:rgd-exact-line-search} is well defined and satisfies
\begin{equation}\label{eq:local-rgd-stepsize-bounds}
\frac89\le\alpha\le\frac87, \qquad |\alpha-1|\le\frac17.
\end{equation}
For $\bm\xi=\bm0$, set $\alpha=1$ in the following estimates, so that \eqref{eq:local-rgd-stepsize-bounds} still holds. Set $\bm\eta:=\alpha\bm\xi-\bm t_{\bX}$. Decomposing $\bE$ into its tangent and normal components, we have
\begin{equation}\label{eq:local-rgd-eta-decomposition}
\begin{aligned}
\bm\eta =(\alpha-1)\bm t_{\bX} +\alpha\bigl(\Pcal_{T_{\bX}}p^{-1}\Pcal_\Omega\Pcal_{T_{\bX}}-\Pcal_{T_{\bX}}\bigr)\bm t_{\bX}+\alpha\Pcal_{T_{\bX}}p^{-1}\Pcal_\Omega\Pcal_{T_{\bX}^\perp}\bE.
\end{aligned}
\end{equation}
The assumed neighborhood is contained in $\norm{\bE}_{\F}<\sigma_r(\bX_\star)/4$. Hence part~\textup{(i)} of \Cref{lem:normal-F} applies, and \eqref{eq:local-rgd-operator-bounds}--\eqref{eq:local-rgd-eta-decomposition} give
\begin{align}
\norm{\bm\eta}_{\F} &\le\left(|\alpha-1|+\frac\alpha8\right)\norm{\bE}_{\F} +\frac{3\alpha}{2\sqrt{2p}}\norm{\Pcal_{T_{\bX}^\perp}\bE}_{\F}\notag\\
&\le\frac27\norm{\bE}_{\F} +\frac{24}{7\sqrt2}\frac{\norm{\bE}_{\F}^2}{\sqrt p\,\sigma_r(\bX_\star)} \le\frac13\norm{\bE}_{\F}. \label{eq:local-rgd-eta-bound}
\end{align}
In particular, $\norm{\bE}_{\F}\le\sigma_r(\bX_\star)/128<\sigma_r(\bX_\star)/40$ and $\norm{\bm\eta}_{\F}\le\sigma_r(\bX_\star)/384<\sigma_r(\bX_\star)/320$. Consequently, part~\textup{(ii)} of \Cref{lem:normal-F} applies. Since $\alpha\bm\xi=\bm t_{\bX}+\bm\eta$, we obtain
\begin{equation}\label{eq:local-rgd-one-step}
\begin{aligned}
\norm{\operatorname{Retr}_{\bX}(\alpha\bm\xi)-\bX_\star}_{\F} &\le\norm{\bm\eta}_{\F}\left(1+\frac{11}{5}\frac{\norm{\bE}_{\F}}{\sigma_r(\bX_\star)}+\frac{11}{10}\frac{\norm{\bm\eta}_{\F}}{\sigma_r(\bX_\star)}\right)\\
&\le\frac13\left(1+\frac{11}{640}+\frac{11}{3840}\right)\norm{\bE}_{\F} \le\frac38\norm{\bE}_{\F}.
\end{aligned}
\end{equation}
If $\bm\xi=\bm0$, then $\operatorname{Retr}_{\bX}(\alpha\bm\xi)=\bX$, so \eqref{eq:local-rgd-one-step} implies $\bX=\bX_\star$. Thus every terminal iterate in the stated neighborhood equals $\bX_\star$.

Since $\bX$ was arbitrary in the Frobenius ball of \Cref{thm:fixed-rgd-local}, \eqref{eq:local-rgd-one-step} holds simultaneously throughout that ball. In particular, the neighborhood is invariant under the RGD update. Therefore, if $\bX_0$ satisfies the hypothesis of the theorem, induction gives
\[
\norm{\bX_k-\bX_\star}_{\F}\le\left(\frac38\right)^k\norm{\bX_0-\bX_\star}_{\F},\qquad k\ge0.
\]
The same induction verifies the hypotheses of \Cref{lem:normal-F} and the positivity of every nonterminal line-search denominator, and therefore guarantees that all subsequent iterates are well defined. This proves the theorem.
\end{proof}

\section{Multiscale Initialization and Global Convergence of RGD}
\label{sec:initialization}

We now prove \Cref{thm:sharp-spectral,thm:init}. We first bound the approximate reconstruction error in the sharp norm, then justify the residual stopping rule and verify the two local convergence conditions. Combining these estimates with \Cref{thm:fixed-rgd-local} also proves the global RGD result.

\subsection{Sharp spectral reconstruction}
\label{sec:proof-spectral}

For a fixed error matrix $\bE$ and an independent subset $\Lambda\sim\operatorname{Bernoulli}(q)$, the sampling perturbation is $\bW=(q^{-1}\Pcal_\Lambda-\Ical)\bE$. To control both row and column errors, we use the symmetric dilation
\begin{equation}\label{eq:spectral-dilation}
\mathscr W:=
\begin{bmatrix}
\bm0&\bW\\
\bW^T&\bm0
\end{bmatrix}
, \qquad \mathscr Q_\star:=
\begin{bmatrix}
\bU_\star&\bm0\\
\bm0&\bV_\star
\end{bmatrix}
.
\end{equation}
The columns of $\mathscr Q_\star$ are orthonormal, and \Cref{ass:incoherence} gives $\norm{\mathscr Q_\star}_{2,\infty}\le\sqrt{\mu r/n}$. The following lemma controls powers of the sampling perturbation on these columns.

\begin{lemma}\label[lemma]{lem:spectral-concentration}
Suppose that \Cref{ass:incoherence} holds. Let $\bE$ be fixed with $\sharpnorm{\bE}\le\tau$, where $\tau>0$. If $q\ge c_2\mu r\log n/n$, where $c_2>0$ is a sufficiently large absolute constant, then, with probability at least $1-n^{-12}/48$ over $\Lambda$,
\begin{equation}\label{eq:spectral-power-bounds}
\norm{\mathscr W}_{\op}\le\frac{\tau}{512}, \qquad \norm{\mathscr W^j\mathscr Q_\star}_{2,\infty} \le2\sqrt{\frac{\mu r}{n}}\left(\frac{\tau}{64}\right)^j, \qquad j\ge0.
\end{equation}
\end{lemma}

\begin{proof}
The proof is deferred to \Cref{app:fresh}.
\end{proof}

We next give a deterministic reconstruction estimate. Its hypotheses concern approximate singular factors, rather than an exact truncated SVD. This distinction allows the computation in \eqref{eq:threshold-iteration} to stop after a prescribed number of steps.

\begin{lemma}\label[lemma]{lem:spectral-reconstruction}
Suppose that \Cref{ass:incoherence} holds. Let $\bY=\bX_\star+\bW$, and suppose that \eqref{eq:spectral-power-bounds} holds for some $\tau>0$. Let $\bZ^+$ have rank at most $r$. If $\bZ^+\ne\bm0$, let $\bZ^+=\widehat{\bU}\widehat{\bm\Sigma}\widehat{\bV}^T$ be a compact singular value decomposition and suppose that
\begin{subequations}\label{eq:spectral-approximation-conditions}
\begin{align}
\sigma_{\min}(\widehat{\bm\Sigma})&\ge\frac{\tau}{8}, &\bY^T\widehat{\bU}&=\widehat{\bV}\widehat{\bm\Sigma},\label{eq:spectral-one-sided-pair}\\
\norm{\bY\widehat{\bV}-\widehat{\bU}\widehat{\bm\Sigma}}_{\op} &\le\frac{\tau}{64}\sqrt{\frac{\mu r}{n}}, &\norm{\bY-\bZ^+}_{\op}&\le\frac{7\tau}{32}.\label{eq:spectral-pair-residual}
\end{align}
\end{subequations}
If $\bZ^+=\bm0$, suppose instead that $\norm{\bY}_{\op}\le7\tau/32$. Then $\sharpnorm{\bZ^+-\bX_\star}\le\tau/4$.
\end{lemma}

\begin{proof}
The proof is deferred to \Cref{sec:proof-spectral-reconstruction}.
\end{proof}

The iteration in \eqref{eq:threshold-iteration} is a randomized subspace iteration; see, e.g., \cite{HalkoMartinssonTropp2011}. The next lemma verifies the approximation conditions for $\Tcal_\tau$. Only the singular components above the current error scale must be accurately represented. No gap between adjacent singular values is assumed.

\begin{lemma}\label[lemma]{lem:spectral-approximation}
Let $\bY\in\R^{n\times n}$ satisfy $\sigma_{r+1}(\bY)\le\tau/512$, where $\tau>0$. Then $\bZ^+=\Tcal_\tau(\bY)$ has rank at most $r$ and, with probability at least $1-n^{-12}/48$ over its Gaussian initial matrix, satisfies \eqref{eq:spectral-approximation-conditions}. If $\bZ^+=\bm0$, the conclusion is $\norm{\bY}_{\op}\le7\tau/32$.

\end{lemma}

\begin{proof}
The proof is deferred to \Cref{sec:proof-spectral-approximation}.
\end{proof}

\subsection{Proof of Theorem~\ref{thm:sharp-spectral}}
\label{sec:proof-sharp-spectral}

\begin{proof}
Set
\[
\bW=(q^{-1}\Pcal_\Lambda-\Ical)(\bX_\star-\bZ), \qquad \bY=\bX_\star+\bW.
\]
Choose $C_5\ge c_2$. By \Cref{lem:spectral-concentration}, except with probability $n^{-12}/48$,
\[
\eqref{eq:spectral-power-bounds}\quad\text{holds}, \qquad \sigma_{r+1}(\bY)\le\norm{\bW}_{\op}\le\frac{\tau}{512}.
\]
Conditional on such a realization of $\bY$, \Cref{lem:spectral-approximation} gives, except with probability $n^{-12}/48$,
\[
\rank(\Tcal_\tau(\bY))\le r, \qquad \eqref{eq:spectral-approximation-conditions}\quad\text{holds}.
\]
Hence \Cref{lem:spectral-reconstruction} yields
\[
\sharpnorm{\Tcal_\tau(\bY)-\bX_\star}\le\frac{\tau}{4}.
\]
The result follows by a union bound.
\end{proof}

\subsection{Proof of Proposition~\ref{prop:no-premature-rebound}}
\label{sec:proof-residual-stopping}

\begin{proof}
We continue \eqref{eq:multiscale-update} through step $K$, irrespective of the stopping test. Scalar Bernstein gives $\norm{\bX_\star}_{\F}\le\tau_0\le3\norm{\bX_\star}_{\F}$ except with probability $n^{-12}/48$. At each reconstruction, conditional on the preceding observations and Gaussian matrices, \Cref{lem:spectral-concentration} gives \eqref{eq:spectral-power-bounds} except with probability $n^{-12}/48$. Conditional on the corrected matrix, \eqref{eq:threshold-gaussian-bound} holds with the same failure bound and implies \eqref{eq:spectral-approximation-conditions} by the proof of \Cref{lem:spectral-approximation}. Induction using \Cref{lem:spectral-reconstruction} therefore gives an event, with failure probability at most $(2K+1)n^{-12}/48$, on which \eqref{eq:init-invariant}, the sampling estimates \eqref{eq:spectral-power-bounds}, and the Gaussian estimate \eqref{eq:threshold-gaussian-bound} hold throughout the complete sequence.

\emph{Observed residuals.} For every nonzero iterate on this event, \eqref{eq:spectral-offspace-entry} and the projection onto the true singular spaces give
\[
\max\{\norm{\bU_\ell}_{2,\infty},\norm{\bV_\ell}_{2,\infty}\} \le2\sqrt{\frac{\mu r}{n}},
\]
where $\bZ_\ell=\bU_\ell\bm\Sigma_\ell\bV_\ell^T$ is a compact singular value decomposition. Fix an index $\ell<K$ for which these preceding reconstruction estimates hold and $\rank(\bZ_\ell)=r$, and condition on the observations and Gaussian matrices used to construct $\bZ_\ell$. Write $\bE_\ell=\bZ_\ell-\bX_\star$. The identity
\[
\bE_\ell=\bU_\star\bU_\star^T\bE_\ell +(\bI-\bU_\star\bU_\star^T)\bE_\ell\bV_\ell\bV_\ell^T
\]
implies
\begin{equation}\label{eq:residual-relative-entry}
\norm{\bE_\ell}_{\infty} \le3\sqrt{\frac{\mu r}{n}}\norm{\bE_\ell}_{\F}.
\end{equation}
The next observation set $\Omega^{(\ell+1)}$ is independent of $\bE_\ell$. For the centered sum defining $R_\ell^2-\norm{\bE_\ell}_{\F}^2$, the variance and summand bounds in \Cref{lem:bernstein} are at most
\[
\frac{9\mu r}{nq}\norm{\bE_\ell}_{\F}^4 \quad\text{and}\quad \frac{9\mu r}{nq}\norm{\bE_\ell}_{\F}^2,
\]
respectively. Set $\bW=(q^{-1}\Pcal_{\Omega^{(\ell+1)}}-\Ical)(-\bE_\ell)$ and $T_\star=T_{\bX_\star}\Mcal_r$. Since
\[
\norm{\Pcal_{T_\star}(\bm e_i\bm e_j^T)}_{\F}^2\le\frac{2\mu r}{n},
\]
the corresponding bounds for the vectorization of $\Pcal_{T_\star}\bW$ are at most $2\mu r\norm{\bE_\ell}_{\F}^2/(nq)$ and $3\sqrt2\mu r\norm{\bE_\ell}_{\F}/(nq)$. Applying \Cref{lem:bernstein} with $t=24\log n$ gives, after increasing $C_5$,
\begin{equation}\label{eq:residual-fresh-estimates}
\frac34\norm{\bE_\ell}_{\F}^2\le R_\ell^2\le\frac54\norm{\bE_\ell}_{\F}^2, \qquad \norm{\Pcal_{T_\star}\bW}_{\F}\le\frac1{16}\norm{\bE_\ell}_{\F},
\end{equation}
except with conditional probability $n^{-12}/48$. These estimates also hold when $\bE_\ell=\bm0$. A conditional union bound over $\ell<K$, together with the preceding reconstruction events, gives total failure probability at most
\[
\frac{3K+1}{48}n^{-12}<\frac{n^{-10}}6.
\]
We work on this joint event for the remainder of the proof.

\emph{Reconstruction after rank $r$ is attained.} Fix a rank-$r$ iterate $\bZ_\ell$ with $\ell<K$, and write $\sigma=\sigma_r(\bX_\star)$ and $\tau=\tau_\ell$. The preceding reconstruction retains $r$ singular values above $\tau_{\ell-1}/8=\tau/2$. By Weyl's inequality and \eqref{eq:spectral-power-bounds},
\[
\frac{\tau}{2}\le\sigma+\frac{\tau}{128}, \qquad\text{hence}\qquad \tau\le3\sigma.
\]
For the next corrected matrix $\bY=\bX_\star+\bW$, it follows that
\begin{equation}\label{eq:residual-full-rank-scale}
\norm{\bW}_{\op}\le\frac{\tau}{512}\le\frac{\sigma}{16}, \qquad \sigma_r(\bY)\ge\frac{15\sigma}{16}.
\end{equation}

Let $\widetilde{\bX}=\Hcal_r(\bY)$ and $\widetilde{\bE}=\widetilde{\bX}-\bX_\star$. Then $\norm{\widetilde{\bE}}_{\op}\le2\norm{\bW}_{\op}$. In the left and right singular coordinates of $\bX_\star$, write $\widetilde{\bE}=\left[\begin{smallmatrix}\bA&\bB\\\bC&\bm D\end{smallmatrix}\right]$. The matrix $\bm\Sigma_\star+\bA$ is invertible, and $\rank(\widetilde{\bX})=r$ gives $\bm D=\bC(\bm\Sigma_\star+\bA)^{-1}\bB$. Consequently,
\[
\left|\inner{\Pcal_{T_\star^\perp}\bW}{\widetilde{\bE}}\right| \le\frac{\norm{\bW}_{\op}}{\sigma-2\norm{\bW}_{\op}} \norm{\bB}_{\F}\norm{\bC}_{\F} \le\frac{\norm{\bW}_{\op}}{2(\sigma-2\norm{\bW}_{\op})} \norm{\widetilde{\bE}}_{\F}^2.
\]
The best-approximation property of $\widetilde{\bX}$ implies
\[
\norm{\widetilde{\bE}}_{\F}^2 \le2\inner{\bW}{\widetilde{\bE}} \le2\norm{\Pcal_{T_\star}\bW}_{\F}\norm{\widetilde{\bE}}_{\F} +\frac{\norm{\bW}_{\op}}{\sigma-2\norm{\bW}_{\op}}\norm{\widetilde{\bE}}_{\F}^2.
\]
Using \eqref{eq:residual-full-rank-scale} and \eqref{eq:residual-fresh-estimates}, we obtain
\begin{equation}\label{eq:residual-exact-reconstruction}
\norm{\widetilde{\bX}-\bX_\star}_{\F} \le4\norm{\Pcal_{T_\star}\bW}_{\F} \le\frac14\norm{\bE_\ell}_{\F}.
\end{equation}

It remains to account for the finite subspace iteration in $\Tcal_\tau$. By \eqref{eq:residual-full-rank-scale} and $\tau\le3\sigma$, the space $\range(\bU_1)$ in the proof of \Cref{lem:spectral-approximation} is precisely the leading $r$-dimensional left singular space of $\bY$. Write $\bm\Sigma_1$ for its singular values and $\bm P=\bQ\bQ^T$ for the final subspace projector. Equation~\eqref{eq:threshold-weighted-subspace}, with $t=\ceil{12\log n}$ and $n^{32}30^{-t}\le n^{-7}$, yields
\[
\norm{(\bI-\bm P)\bU_1}_{\op}\le n^{-7}, \qquad \norm{(\bI-\bm P)\bU_1\bm\Sigma_1}_{\op} \le\frac{\tau}{8\sqrt2}n^{-7}.
\]
In particular,
\[
\sigma_r(\bQ^T\bY) \ge\sqrt{1-n^{-14}}\,\sigma_r(\bY) >\frac{\tau}{8}.
\]
Thus all $r$ singular values are retained and $\bZ_{\ell+1}=\bm P\bY$ has rank $r$. The equality of the largest principal angles between two $r$-dimensional spaces also gives
\[
\norm{\bm P(\bI-\bU_1\bU_1^T)}_{\op} =\norm{(\bI-\bm P)\bU_1}_{\op}\le n^{-7}.
\]
Since $\norm{\bY-\widetilde{\bX}}_{\op}\le\tau/512$, we have
\begin{align*}
\norm{\bZ_{\ell+1}-\widetilde{\bX}}_{\F} &\le\norm{(\bI-\bm P)\widetilde{\bX}}_{\F} +\norm{\bm P(\bY-\widetilde{\bX})}_{\F}\\
&\le\sqrt r\,n^{-7}\tau\left(\frac1{8\sqrt2}+\frac1{512}\right) \le n^{-6}\tau.
\end{align*}
Together with \eqref{eq:residual-exact-reconstruction}, this proves
\begin{equation}\label{eq:residual-relative-reconstruction}
\norm{\bE_{\ell+1}}_{\F} \le\frac14\norm{\bE_\ell}_{\F}+n^{-6}\tau_\ell.
\end{equation}
The same argument shows that every iterate after the first rank-$r$ iterate also has rank $r$.

\emph{The stopping test.} Therefore, whenever the residual test returns a candidate on the joint event, the compared rank-$r$ iterates are consecutive. Write them as $\bZ_\ell$ and $\bZ_{\ell+1}$, where $\ell+1<K$ and $R_{\ell+1}>R_\ell$. By \eqref{eq:residual-fresh-estimates} and \eqref{eq:residual-relative-reconstruction},
\[
\sqrt{\frac34}\norm{\bE_\ell}_{\F} <\sqrt{\frac54}\left(\frac14\norm{\bE_\ell}_{\F}+n^{-6}\tau_\ell\right).
\]
It follows that
\[
\norm{\bE_\ell}_{\F} <2n^{-6}\tau_\ell \le6n^{-6}\sigma_r(\bX_\star) \le\frac{\sigma_r(\bX_\star)}{1024n},
\]
where the sampling condition, with $C_5$ sufficiently large, implies $n\ge8$. Since the algorithm returns $\bZ_\ell$, this proves \eqref{eq:residual-stop-accuracy}. Finally, \eqref{eq:sharp-main} gives $\sharpnorm{\bE_\ell}\le n\norm{\bE_\ell}_{\F}/(\mu r)$, and the sampling condition gives $q\ge r/n$. Thus the returned matrix satisfies both local entrance conditions. This proves the proposition.
\end{proof}

\subsection{Proof of Theorem~\ref{thm:init}}
\label{sec:proof-finite-rgd}

\begin{proof}[Proof of \Cref{thm:init}]
We first consider the complete sequence in \eqref{eq:multiscale-update}.

\emph{Initial scale.} By incoherence,
\[
\frac{|(\bX_\star)_{ij}|^2}{\norm{\bX_\star}_{\F}^2} \le\frac{\mu^2r^2}{n^2}, \qquad \sum_{i,j}\frac{|(\bX_\star)_{ij}|^2}{\norm{\bX_\star}_{\F}^2}=1.
\]
Scalar Bernstein applied to the observed squared entries gives
\[
\Pr\!\left\{\left|\frac{\norm{\Pcal_{\Omega^{(0)}}(\bX_\star)}_{\F}^2}{q\norm{\bX_\star}_{\F}^2}-1\right|>\frac12\right\} \le2\exp\!\left(-\frac{3n^2q}{28\mu^2r^2}\right) \le\frac{n^{-12}}{48},
\]
after increasing $C_5$, since $\mu r\le n$. Thus, outside this event,
\begin{equation}\label{eq:init-scale-bound}
\norm{\bX_\star}_{\F}\le\tau_0\le3\norm{\bX_\star}_{\F}.
\end{equation}
In particular, $\tau_0>0$ and $\sharpnorm{\bX_\star}=\sigma_1(\bX_\star)\le\tau_0$.

\emph{Induction step.} Let $\mathcal F_\ell$ contain the observations in $\Omega^{(0)},\ldots,\Omega^{(\ell)}$ and the Gaussian matrices used in the first $\ell$ reconstructions. Conditional on $\mathcal F_\ell$, the matrix $\bZ_\ell$ and the scale $\tau_\ell$ are fixed, whereas $\Omega^{(\ell+1)}$ and the next Gaussian matrix are independent. Let $\mathcal G_0$ be the event in \eqref{eq:init-scale-bound}, and define
\[
\mathcal G_{\ell+1} :=\mathcal G_\ell\cap \left\{\rank(\bZ_{\ell+1})\le r, \sharpnorm{\bZ_{\ell+1}-\bX_\star}\le4^{-(\ell+1)}\tau_0\right\}.
\]
On $\mathcal G_\ell$, the induction hypothesis and \Cref{thm:sharp-spectral} imply
\[
\Pr(\mathcal G_\ell\setminus\mathcal G_{\ell+1}\mid\mathcal F_\ell) \le\frac{n^{-12}}{24}.
\]
Taking expectations and proceeding by induction proves \eqref{eq:init-invariant} through step $K$ on $\mathcal G_K$.

\emph{Probability estimate.} Since $K\le n$ and $\mathcal G_0\supseteq\cdots\supseteq\mathcal G_K$,
\begin{equation}\label{eq:init-failure}
\Pr(\mathcal G_K^c) \le\frac{n^{-12}}{48}+K\frac{n^{-12}}{24} =\frac{2K+1}{48}n^{-12} <\frac{n^{-10}}6.
\end{equation}

\emph{Verification of the local hypotheses at step $K$.} For $K=\ceil{5+\log_4(\kappa\sqrt{nr})}$, the difference $\bZ_K-\bX_\star$ has rank at most $2r$. Using \eqref{eq:init-invariant}, $\tau_0\le3\sqrt r\,\kappa\sigma_r(\bX_\star)$, and $q\ge r/n$, we obtain
\[
\norm{\bZ_K-\bX_\star}_{\F} \le\sqrt{2r}\,4^{-K}\tau_0 \le\frac{3\sqrt2}{1024}\sqrt{\frac rn}\,\sigma_r(\bX_\star) \le\frac{\sqrt q}{128}\sigma_r(\bX_\star).
\]
For $K=\ceil{6+\log_4(\mu r^{3/2}\kappa)}$, the same invariant gives
\[
\sharpnorm{\bZ_K-\bX_\star} \le4^{-K}\tau_0 \le\frac{3}{4096\mu r}\sigma_r(\bX_\star) \le\frac{\sigma_r(\bX_\star)}{1000\mu r}.
\]
In both cases,
\[
\norm{\bZ_K-\bX_\star}_{\op}<\sigma_r(\bX_\star), \qquad \rank(\bZ_K)\le r,
\]
so Weyl's inequality gives $\rank(\bZ_K)=r$.

\emph{Residual stopping.} Intersect $\mathcal G_K$ with the event in \Cref{prop:no-premature-rebound}. The total failure probability is at most $n^{-10}/3$. If no residual increase is detected, then $\widehat K=K$ and the preceding bounds apply. Otherwise, \eqref{eq:residual-stop-accuracy}, \eqref{eq:sharp-main}, and $q\ge r/n$ give
\[
\norm{\bZ_{\widehat K}-\bX_\star}_{\F}\le\frac{\sqrt q}{128}\sigma_r(\bX_\star), \qquad \sharpnorm{\bZ_{\widehat K}-\bX_\star}\le\frac{\sigma_r(\bX_\star)}{1000\mu r}.
\]
The residual test returns only an iterate of rank $r$. This proves \eqref{eq:init-rgd-entrance}--\eqref{eq:init-rgn-entrance} in both cases.
\end{proof}

We can now apply the local RGD theorem to the initialization output. The corresponding RGN argument uses the independence of $\widehat\Omega$ in \Cref{sec:proof-global-rgn}.

\subsection{Proof of Theorem~\ref{thm:rgd-global}}
\label{sec:proof-global-rgd}

\begin{proof}
Choose $C_1\ge24(C_4+C_5+1)$. Since
\[
K+1\le12\log(n\kappa), \qquad p=1-(1-q)^{K+1}\le(K+1)q,
\]
the sampling condition gives $q\ge \frac{C_1\mu r\log n}{12n}$. It also implies $K+1\le n$. Hence \Cref{thm:init} gives, with probability at least $1-n^{-10}/3$,
\[
\norm{\bX_0-\bX_\star}_{\F} \le\frac{\sqrt q}{128}\sigma_r(\bX_\star) \le\frac{\sqrt p}{128}\sigma_r(\bX_\star).
\]
The event in \Cref{thm:fixed-rgd-local} holds simultaneously throughout this neighborhood, so no independence between $\bX_0$ and $\Omega$ is required. A union bound proves the convergence statement. 
\end{proof}

\section{Convergence of RGN}
\label{sec:local-rgn}

We prove \Cref{thm:fixed} by controlling the initial RGN iterates in the sharp norm and then establishing a quadratic Frobenius recurrence. For the initial iterates, we compare the original sequence with auxiliary sequences obtained by completing one row or column. Once the Frobenius error is sufficiently small, \Cref{lem:det-transfer} gives the required estimate uniformly, without further leave-one-out comparisons.

Throughout this section, $\widehat{\Omega}\sim\operatorname{Bernoulli}(q)$ and $\Rcal^0:=q^{-1}\Pcal_{\widehat{\Omega}}$. The initial point $\bX_0$ is independent of $\widehat{\Omega}$. Inverses of tangent normal operators are taken on their corresponding tangent spaces. Whenever \Cref{lem:degree} is used below, it is applied with $(\Lambda,p_\Lambda)=(\widehat{\Omega},q)$.

\subsection{Dependence and the leave-one-out construction}
\label{sec:proof-fixed-rgn}

The difficulty is that $\bX_k=\bX_k(\widehat{\Omega})$, so concentration cannot be applied to $T_{\bX_k}\Mcal_r$ as though it were independent of the observation set $\widehat{\Omega}$.  Following the row- and column-deletion construction in \cite[Sec.~7.2 and Algorithm~5]{MaWangChiChen2020} and \cite{DingChen2020}, define, for each row $i$ and column $j$,
\[
\Rcal^{\mathrm r,i}(\bZ) :=\bm e_i\bm e_i^T\bZ +q^{-1}\Pcal_{\widehat{\Omega}}\bigl((\bI-\bm e_i\bm e_i^T)\bZ\bigr),
\]
\[
\Rcal^{\mathrm c,j}(\bZ) :=\bZ\bm e_j\bm e_j^T +q^{-1}\Pcal_{\widehat{\Omega}}\bigl(\bZ(\bI-\bm e_j\bm e_j^T)\bigr),
\]
and set
\[
\mathfrak A:=\{0\}\cup\{(\mathrm r,i):1\le i\le n\} \cup\{(\mathrm c,j):1\le j\le n\}.
\]
Each completed operator replaces one row or column by its population counterpart. The corresponding auxiliary sequence is independent of the Bernoulli variables in that row or column. These sequences are used only for $0\le k\le\bar K$.

The following lemma gives a sampling isometry that holds simultaneously for the original operator and all completed operators.

\begin{lemma}\label[lemma]{lem:all-true-rip}
Suppose that \Cref{ass:incoherence} holds. If $q\ge c_1\mu r\log n/n$, with $c_1$ as in \Cref{lem:true-tangent-rip}, then, with probability at least $1-n^{-10}/24$,
\[
\max_{\alpha\in\mathfrak A} \norm{\Pcal_{T_{\bX_\star}}\Rcal^\alpha\Pcal_{T_{\bX_\star}} -\Pcal_{T_{\bX_\star}}}_{\F\to\F} \le\frac1{16}.
\]
\end{lemma}

\begin{proof}
The proof is deferred to \Cref{sec:proof-all-true-rip}.
\end{proof}

For every $\alpha\in\mathfrak A$, set $\bX_0^\alpha=\bX_0$. We use the deterministic error bound
\begin{equation}\label{eq:rgn-prefix-envelope}
\rho_k:=\frac{\sigma_r(\bX_\star)}{280\mu r} \left(\frac7{25}\right)^{2^k}, \qquad \rho_{k+1}=280\mu r\frac{\rho_k^2}{\sigma_r(\bX_\star)}.
\end{equation}

To define each auxiliary sequence on every outcome, we use the following stopping rule for $0\le k<\bar K$. Given $\bX_k^\alpha$, let $\bm t_{\bX_k^\alpha}$ be its tangent error component from \eqref{eq:tangent-normal-components}. If the tangent least-squares problem
\[
\min_{\bm\xi\in T_{\bX_k^\alpha}\Mcal_r} \frac12\inner{\bX_k^\alpha+\bm\xi-\bX_\star} {\Rcal^\alpha(\bX_k^\alpha+\bm\xi-\bX_\star)}
\]
has a unique minimizer, let $\bm\xi_k^\alpha$ be that minimizer and set $\bm\eta_k^\alpha:=\bm\xi_k^\alpha-\bm t_{\bX_k^\alpha}$.  Otherwise, set
\[
\bm\xi_k^\alpha:=\bm t_{\bX_k^\alpha}, \qquad \bm\eta_k^\alpha:=\bm 0.
\]

When the minimizer is unique and the graph retraction is defined, denote the candidate next state by
\[
\widehat{\bX}_{k+1}^\alpha :=\operatorname{Retr}_{\bX_k^\alpha}(\bm\xi_k^\alpha).
\]

We accept the update only if the minimizer is unique, the graph retraction is defined, and
\[
\sharpnorm{\bX_k^\alpha-\bX_\star}\le\rho_k, \qquad \sharpnorm{\widehat{\bX}_{k+1}^\alpha-\bX_\star}\le\rho_{k+1},
\]
hold. In this case, set $\bX_{k+1}^\alpha:=\widehat{\bX}_{k+1}^\alpha$. Otherwise, leave $\bX_k^\alpha$ unchanged and set the subsequent states equal to $\bX_\star$: for every $k+1\le\ell\le\bar K$,
\[
\bX_\ell^\alpha:=\bX_\star, \qquad \bm\xi_\ell^\alpha=\bm\eta_\ell^\alpha:=\bm 0.
\]
After stopping, we use the fixed compact SVD $\bX_\star=\bU_\star\bm\Sigma_\star\bV_\star^T$. Thus the auxiliary sequence is defined on every outcome and remains independent of the variables in its completed row or column.

For these stopped processes, define
\[
d_k^{\rm loo}:= \max_{\alpha\in\mathfrak A}\norm{\bX_k^\alpha-\bX_k^0}_{\F}, \qquad h_k^{\rm corr}:= \max_{\alpha\in\mathfrak A}\norm{\bm\eta_k^\alpha-\bm\eta_k^0}_{\F}.
\]
For $0\le k\le\bar K$, write $\bX_k:=\bX_k^0$. If the updates at indices $0,\ldots,k-1$ are accepted, this sequence agrees with \Cref{alg:fixed} through index $k$. We will prove that all these updates are accepted on the event used in the convergence proof.

\subsection{Simultaneous sampling events}

We first bound the number of observations in each row and column. These bounds will be used to control the sampled tangent operators.

\begin{lemma}\label[lemma]{lem:degree}
Let $\Lambda\sim\operatorname{Bernoulli}(p_\Lambda)$, and set $\delta_{ij}:=\mathbf 1_{\{(i,j)\in\Lambda\}}$. Then, with probability at least $1-2ne^{-p_\Lambda n/3}$,
\[
\max_i\sum_j\delta_{ij}\le2p_\Lambda n, \qquad \max_j\sum_i\delta_{ij}\le2p_\Lambda n.
\]
\end{lemma}

\begin{proof}
For every row or column degree $d\sim\operatorname{Binomial}(n,p_\Lambda)$, the multiplicative Chernoff bound \cite[Corollary~5.2]{Tropp2012}, with deviation parameter one, gives
\[
\Pr\{d>2p_\Lambda n\}\le e^{-p_\Lambda n/3}.
\]
Therefore, a union bound over the $2n$ row and column degrees proves the result.
\end{proof}

The next lemma gives concentration bounds that hold simultaneously for every completed row, every completed column, and every index $0\le k<\bar K$.

\begin{lemma}\label[lemma]{lem:loo-probability}
Let $\bX_0$ be the initial point and write $\delta_{ij}:=\mathbf 1_{\{(i,j)\in\widehat{\Omega}\}}$. For the stopped row-$i$ and column-$j$ trajectories, define
\[
\bm z_k^{\mathrm r,i}:=\left[\Pcal_{T_{\bX_k^{\mathrm r,i}}^\perp}(\bX_\star-\bX_k^{\mathrm r,i})-\bm\eta_k^{\mathrm r,i}\right]^T\bm e_i, \qquad \bm z_k^{\mathrm c,j}:=\left[\Pcal_{T_{\bX_k^{\mathrm c,j}}^\perp}(\bX_\star-\bX_k^{\mathrm c,j})-\bm\eta_k^{\mathrm c,j}\right]\bm e_j,
\]
and set
\[
(\bm w_k^{\mathrm r,i})_a:=\left(\frac{\delta_{ia}}q-1\right)(\bm z_k^{\mathrm r,i})_a, \qquad (\bm w_k^{\mathrm c,j})_a:=\left(\frac{\delta_{aj}}q-1\right)(\bm z_k^{\mathrm c,j})_a.
\]
Let $\bV_k^{\mathrm r,i}$ and $\bU_k^{\mathrm c,j}$ be the right and left singular factors selected by the fixed compact-SVD convention. Suppose that \Cref{ass:incoherence} holds and $\bX_0$ is independent of $\widehat{\Omega}$ and satisfies
\[
\sharpnorm{\bX_0-\bX_\star}\le\rho_0=\frac{\sigma_r(\bX_\star)}{1000\mu r}.
\]
Then, conditional on $\bX_0$, with probability at least $1-n^{-10}/48$, the following two inequalities hold simultaneously for all $1\le i,j\le n$ and $0\le k<\bar K$:
\begin{align}
2\sqrt{\frac{\mu r}{n}}\norm{\bm w_k^{\mathrm r,i}}_2+\norm{(\bV_k^{\mathrm r,i})^T\bm w_k^{\mathrm r,i}}_2 &\le8\sqrt{\frac{3\log n}{q}}\left(4\sqrt{\frac{\mu r}{n}}\norm{\bm z_k^{\mathrm r,i}}_2+\norm{\bm z_k^{\mathrm r,i}}_\infty\right)\notag\\
&\quad+\frac{64\sqrt{\mu r/n}\,\log n}{q}\norm{\bm z_k^{\mathrm r,i}}_\infty, \label{eq:simultaneous-row-event}\\
2\sqrt{\frac{\mu r}{n}}\norm{\bm w_k^{\mathrm c,j}}_2+\norm{(\bU_k^{\mathrm c,j})^T\bm w_k^{\mathrm c,j}}_2 &\le8\sqrt{\frac{3\log n}{q}}\left(4\sqrt{\frac{\mu r}{n}}\norm{\bm z_k^{\mathrm c,j}}_2+\norm{\bm z_k^{\mathrm c,j}}_\infty\right)\notag\\
&\quad+\frac{64\sqrt{\mu r/n}\,\log n}{q}\norm{\bm z_k^{\mathrm c,j}}_\infty. \label{eq:simultaneous-column-event}
\end{align}
\end{lemma}

\begin{proof}
The proof is deferred to \Cref{sec:proof-loo-probability}.
\end{proof}

\subsection{Error bounds for the initial RGN iterates}

The next lemma bounds the sampled tangent normal operators uniformly over the sharp-norm neighborhood. It also bounds the difference between the sampled and population tangent corrections.

\begin{lemma}\label[lemma]{lem:local-conditioning}
Let $\alpha\in\mathfrak A$ and $\bX\in\Mcal_r$, set $e_\sharp:=\sharpnorm{\bX-\bX_\star}$, and let $\bm t_{\bX}$ be defined by \eqref{eq:tangent-normal-components}. Suppose that \Cref{ass:incoherence} and the conclusions of \Cref{lem:all-true-rip,lem:degree} hold, and that $e_\sharp\le\frac{\sigma_r(\bX_\star)}{1000\mu r}$. Then
\begin{equation}\label{eq:local-conditioning-isometry}
\frac9{10}\norm{\bm\zeta}_{\F}^2\le\inner{\bm\zeta}{\Rcal^\alpha\bm\zeta}\le\frac{11}{10}\norm{\bm\zeta}_{\F}^2, \qquad \bm\zeta\in T_{\bX}\Mcal_r.
\end{equation}
Let $\bm\xi_{\bX}^\alpha$ be the resulting unique sampled tangent minimizer and set $\bm\eta_{\bX}^\alpha:=\bm\xi_{\bX}^\alpha-\bm t_{\bX}$. Then
\begin{equation}\label{eq:finite-correction-F-main}
\norm{\bm\eta_{\bX}^\alpha}_{\F}\le9\mu r\frac{e_\sharp^2}{\sigma_r(\bX_\star)}.
\end{equation}
\end{lemma}

\begin{proof}
The proof is deferred to \Cref{sec:proof-local-conditioning}.
\end{proof}

To propagate the sharp-norm error bound, we also need coordinatewise control of the corrections. The following lemma relates this control to the difference between the original and completed corrections and then bounds that difference.

\begin{lemma}
\label[lemma]{lem:correction-proximity}
Let $0\le k<\bar K$. Suppose that \Cref{ass:incoherence} and the conclusions of \Cref{lem:all-true-rip,lem:degree} hold, and that
\[
\max_{\alpha\in\mathfrak A}\sharpnorm{\bX_k^\alpha-\bX_\star}\le\rho_k, \qquad d_k^{\rm loo}\le2\sqrt{\frac{\mu r}{n}}\rho_k.
\]
Then every correction satisfies
\begin{equation}\label{eq:finite-coordinate-main}
\sharpnorm{\bm\eta_k^\alpha} \le2\sqrt{\frac{n}{\mu r}}\,h_k^{\rm corr} +19\mu r\frac{\rho_k^2}{\sigma_r(\bX_\star)}, \qquad \alpha\in\mathfrak A.
\end{equation}
If, in addition, $q\ge c_4\mu r\log n/n$ for a sufficiently large absolute constant $c_4>0$ and the conclusion of \Cref{lem:loo-probability} holds, then
\begin{equation}\label{eq:finite-proximity-main}
h_k^{\rm corr} \le128\sqrt{\frac{\mu r}{n}}\,\mu r \frac{\rho_k^2}{\sigma_r(\bX_\star)}.
\end{equation}
\end{lemma}

\begin{proof}
The proof is deferred to \Cref{sec:proof-correction-proximity}.
\end{proof}

We next pass from the tangent corrections to the retracted iterates. By \Cref{lem:graph-correction}, if $\sharpnorm{\bX-\bX_\star}\le\rho_k$ and $\sharpnorm{\bm\eta}\le\sigma_r(\bX_\star)/16$, then
\begin{align}
\sharpnorm{\operatorname{Retr}_{\bX} (\Pcal_{T_{\bX}}(\bX_\star-\bX)+\bm\eta)-\bX_\star} &\le \sharpnorm{\bm\eta} +13\frac{\rho_k\sharpnorm{\bm\eta}}{\sigma_r(\bX_\star)} +6\frac{\sharpnorm{\bm\eta}^2}{\sigma_r(\bX_\star)}. \label{eq:finite-one-base-main}
\end{align}
The following lemma bounds the difference between retracted iterates at two nearby matrices. For $\bX,\bY\in\Mcal_r$, write $d_{\bX,\bY}:=\norm{\bX-\bY}_{\F}$.

\begin{lemma}\label[lemma]{lem:graph-two-base}
Let $\bX,\bY\in\Mcal_r$, let $\bm\eta_{\bX}\in T_{\bX}\Mcal_r$ and $\bm\eta_{\bY}\in T_{\bY}\Mcal_r$, and use $\bm t_{\bX},\bm t_{\bY}$ from \eqref{eq:tangent-normal-components}.  Set $d_\eta:=\norm{\bm\eta_{\bX}-\bm\eta_{\bY}}_{\F}$.
Suppose that
\[
\max\left\{\sharpnorm{\bX-\bX_\star}, \sharpnorm{\bY-\bX_\star}\right\} \le \rho\le\frac{\sigma_r(\bX_\star)}{1000\mu r}
\]
and
\[
\max\left\{\norm{\bm\eta_{\bX}}_{\F}, \norm{\bm\eta_{\bY}}_{\F}\right\} \le9\mu r\frac{\rho^2}{\sigma_r(\bX_\star)}.
\]
Then $\operatorname{Retr}_{\bX}(\bm t_{\bX}+\bm\eta_{\bX})$ and $\operatorname{Retr}_{\bY}(\bm t_{\bY}+\bm\eta_{\bY})$ are well defined, and
\begin{align}
\norm{\operatorname{Retr}_{\bX}(\bm t_{\bX}+\bm\eta_{\bX}) -\operatorname{Retr}_{\bY}(\bm t_{\bY}+\bm\eta_{\bY})}_{\F} \le\frac54d_\eta +19\frac{\mu r \rho^2}{\sigma_r(\bX_\star)^2}d_{\bX,\bY}. \label{eq:finite-two-base-main}
\end{align}
\end{lemma}

\begin{proof}
The proof is deferred to \Cref{sec:proof-graph-two-base}.
\end{proof}

\subsection{Proof of Theorem~\ref{thm:fixed}}
\label{sec:proof-local-rgn}

\begin{proof}
We prove by induction that
\begin{equation}\label{eq:finite-prefix-induction}
\max_{\alpha\in\mathfrak A} \sharpnorm{\bX_k^\alpha-\bX_\star}\le\rho_k, \qquad d_k^{\rm loo}\le2\sqrt{\frac{\mu r}{n}}\,\rho_k.
\end{equation}
Whenever the candidate retraction is defined, the update rule gives
\begin{equation}\label{eq:local-rgn-one-base-decomposition}
\widehat{\bX}_{k+1}^\alpha-\bX_\star =\bm\eta_k^\alpha+ \left[ \operatorname{Retr}_{\bX_k^\alpha} (\bm t_{\bX_k^\alpha}+\bm\eta_k^\alpha) -\bX_\star-\bm\eta_k^\alpha \right].
\end{equation}
For the actual and an auxiliary candidate, subtracting the two updates gives
\begin{align*}
\widehat{\bX}_{k+1}^0-\widehat{\bX}_{k+1}^\alpha =\bm\eta_k^0-\bm\eta_k^\alpha +\left[\widehat{\bX}_{k+1}^0-\bX_\star-\bm\eta_k^0 -\bigl(\widehat{\bX}_{k+1}^\alpha-\bX_\star -\bm\eta_k^\alpha\bigr)\right].
\end{align*}
\emph{Step 1: bounds for the initial RGN iterates.} Condition on an arbitrary realization of $\bX_0$ satisfying $\sharpnorm{\bX_0-\bX_\star}\le \sigma_r(\bX_\star)/(1000\mu r)$.  Since $\bX_0$ is independent of $\widehat{\Omega}$, this conditioning leaves the law of $\widehat{\Omega}$ unchanged.  Fix $C_4\ge c_1+c_4+128$. By the sampling assumption,
\[
q\ge C_4\frac{\mu r\log n}{n} \ge(c_1+c_4)\frac{\mu r\log n}{n}.
\]
Thus the sampling-rate hypotheses of \Cref{lem:all-true-rip,lem:correction-proximity} are satisfied.

Suppose that the conclusions of \Cref{lem:all-true-rip,lem:degree,lem:loo-probability} hold simultaneously. At $k=0$, all auxiliary trajectories equal $\bX_0$ and $d_0^{\rm loo}=0$, so \eqref{eq:finite-prefix-induction} holds at $k=0$. Suppose inductively that transitions $0,\ldots,k-1$ have been accepted for every auxiliary process and that \eqref{eq:finite-prefix-induction} holds at some $k<\bar K$. Since $\rho_k\le\sigma_r(\bX_\star)/(1000\mu r)$, \Cref{lem:local-conditioning} implies that every sampled tangent normal operator at iteration $k$ is invertible.  Hence every tangent minimizer is unique, and its correction satisfies
\[
\Pcal_{T_{\bX_k^\alpha}}\Rcal^\alpha \Pcal_{T_{\bX_k^\alpha}}\bm\eta_k^\alpha =\Pcal_{T_{\bX_k^\alpha}}\Rcal^\alpha \bN_{\bX_k^\alpha}, \qquad \alpha\in\mathfrak A.
\]

Combining \eqref{eq:finite-coordinate-main} and \eqref{eq:finite-proximity-main}, and using $\rho_k\le\sigma_r(\bX_\star)/(1000\mu r)$, gives
\begin{equation}\label{eq:finite-prefix-correction-sharp}
\sharpnorm{\bm\eta_k^\alpha} \le275\mu r\frac{\rho_k^2}{\sigma_r(\bX_\star)} \le\frac{11}{40}\rho_k <\frac1{16}\sigma_r(\bX_\star), \qquad \alpha\in\mathfrak A.
\end{equation}

The bound \eqref{eq:finite-prefix-induction} at index $k$ and \eqref{eq:finite-prefix-correction-sharp} verify the hypotheses of \Cref{lem:graph-correction} with $\bX=\bX_k^\alpha$ and $\bm\eta=\bm\eta_k^\alpha$.  Hence every candidate graph core is invertible and every candidate retraction is well defined. Applying \eqref{eq:finite-one-base-main} gives
\begin{align*}
\sharpnorm{\widehat{\bX}_{k+1}^\alpha-\bX_\star} &<280\mu r\frac{\rho_k^2}{\sigma_r(\bX_\star)} =\rho_{k+1}.
\end{align*}
Thus transition $k$ is accepted for every $\alpha\in\mathfrak A$.

For the two-base update, \Cref{lem:local-conditioning} and the induction hypothesis give $\max\{\norm{\bm\eta_k^0}_{\F},\norm{\bm\eta_k^\alpha}_{\F}\} \le9\mu r\rho_k^2/\sigma_r(\bX_\star)$, which verifies the correction-size hypothesis of \Cref{lem:graph-two-base}. Using \eqref{eq:finite-two-base-main}, \eqref{eq:finite-proximity-main}, and the induction hypothesis, we obtain
\begin{align*}
\norm{\bX_{k+1}^0-\bX_{k+1}^\alpha}_{\F} \le\frac54h_k^{\rm corr} +19\frac{\mu r\rho_k^2}{\sigma_r(\bX_\star)^2} d_k^{\rm loo} <2\sqrt{\frac{\mu r}{n}}\,\rho_{k+1}.
\end{align*}
Taking the maximum over $\alpha$ proves \eqref{eq:finite-prefix-induction} at index $k+1$. Therefore, all updates before $\bar K$ are accepted, and
\[
\max_{\alpha\in\mathfrak A} \sharpnorm{\bX_k^\alpha-\bX_\star}\le\rho_k, \qquad d_k^{\rm loo}\le2\sqrt{\frac{\mu r}{n}}\,\rho_k, \qquad 0\le k\le\bar K.
\]
In particular, $\sharpnorm{\bX_k-\bX_\star}\le\rho_k$ for $0\le k\le\bar K$.

\emph{Step 2: the Frobenius error at $\bar K$.} Since the difference of two rank-$r$ matrices has rank at most $2r$, $2^{-2^{\bar K}}\le(4n)^{-1}$, $\sqrt{2r}\le2\mu r$, and $q\ge C_4\mu r\log n/n\ge n^{-2}$, we have
\begin{align*}
\norm{\bX_{\bar K}-\bX_\star}_{\F} &\le\sqrt{2r}\,\sharpnorm{\bX_{\bar K}-\bX_\star} \le\sqrt{2r}\,\rho_{\bar K}\\
&\le\frac{\sqrt{2r}}{280\mu r}\,2^{-2^{\bar K}} \sigma_r(\bX_\star) \le\frac1{40n}\sigma_r(\bX_\star) \le\frac1{40}\sqrt q\,\sigma_r(\bX_\star).
\end{align*}
Since all updates before $\bar K$ are accepted, the stopped sequence agrees with \Cref{alg:fixed} through index $\bar K$. From this point onward, $\bX_k$ denotes the iterates of that algorithm without the auxiliary stopping rule.

\emph{Step 3: quadratic convergence in the Frobenius norm.} Since $0\in\mathfrak A$, \Cref{lem:all-true-rip} gives
\[
\norm{\Pcal_{T_{\bX_\star}}q^{-1}\Pcal_{\widehat\Omega} \Pcal_{T_{\bX_\star}}-\Pcal_{T_{\bX_\star}}}_{\F\to\F} \le\frac1{16}.
\]
Fix $k\ge\bar K$ and suppose that
\[
e_k:=\norm{\bX_k-\bX_\star}_{\F} \le\frac1{40}\sqrt q\,\sigma_r(\bX_\star).
\]
Whenever the tangent minimizer is unique, set $\bm\eta_k:=\bm\xi_k-\bm t_{\bX_k}$.  The update rule gives
\begin{equation}\label{eq:local-rgn-tail-decomposition}
\bX_{k+1}-\bX_\star =\bm\eta_k+ \left[ \operatorname{Retr}_{\bX_k}(\bm t_{\bX_k}+\bm\eta_k) -\bX_\star-\bm\eta_k \right].
\end{equation}
By \Cref{lem:det-transfer} with $(\Lambda,p_\Lambda)=(\widehat{\Omega},q)$, the sampled tangent normal operator on $T_{\bX_k}\Mcal_r$ is invertible and
\begin{equation}\label{eq:local-rgn-tail-operator-bounds}
\norm{\left( \Pcal_{T_{\bX_k}}q^{-1}\Pcal_{\widehat{\Omega}} \Pcal_{T_{\bX_k}} \right)^{-1}}_{\F\to\F}\le\frac32, \qquad q^{-1/2}\norm{\Pcal_{\widehat{\Omega}} \Pcal_{T_{\bX_k}}}_{\F\to\F}\le\sqrt{\frac54}.
\end{equation}
Hence the tangent minimizer $\bm\xi_k$ is unique, and
\[
\Pcal_{T_{\bX_k}}q^{-1}\Pcal_{\widehat{\Omega}} \Pcal_{T_{\bX_k}}\bm\eta_k =\Pcal_{T_{\bX_k}}q^{-1}\Pcal_{\widehat{\Omega}}\bN_{\bX_k}.
\]
The self-adjointness of $\Pcal_{\widehat{\Omega}}$, \eqref{eq:local-rgn-tail-operator-bounds}, part~\textup{(i)} of \Cref{lem:normal-F}, and $e_k\le\sqrt q\,\sigma_r(\bX_\star)/40$ give
\begin{align}
\norm{\bm\eta_k}_{\F} &\le\frac32\norm{\Pcal_{T_{\bX_k}}q^{-1} \Pcal_{\widehat{\Omega}}\bN_{\bX_k}}_{\F}\notag\\
&\le\frac32\sqrt{\frac{5}{4q}}\norm{\bN_{\bX_k}}_{\F}\notag\\
&\le\frac{3\sqrt5}{2} \frac{e_k^2}{\sqrt q\,\sigma_r(\bX_\star)} <\frac1{320}\sigma_r(\bX_\star). \label{eq:local-rgn-tail-correction}
\end{align}
The same neighborhood gives $e_k\le\sigma_r(\bX_\star)/40$. Combining \eqref{eq:local-rgn-tail-decomposition}, \eqref{eq:local-rgn-tail-correction}, and part~\textup{(ii)} of \Cref{lem:normal-F}, and using $e_k\le\sqrt q\,\sigma_r(\bX_\star)/40$, shows that the retraction is well defined and yields
\begin{align*}
\norm{\bX_{k+1}-\bX_\star}_{\F} &\le\norm{\bm\eta_k}_{\F} \left(1+\frac{11}{5}\frac{e_k}{\sigma_r(\bX_\star)} +\frac{11}{10}\frac{\norm{\bm\eta_k}_{\F}}{\sigma_r(\bX_\star)}\right)\\
&\le4\frac{e_k^2}{\sqrt q\,\sigma_r(\bX_\star)} \le\frac1{10}e_k \le\frac1{40}\sqrt q\,\sigma_r(\bX_\star).
\end{align*}
Starting from $k=\bar K$, induction proves the Q-quadratic recurrence for every $k\ge\bar K$.  It follows that $\bX_k\to\bX_\star$.

Conditional on the chosen realization of $\bX_0$, a union bound for the events in \Cref{lem:all-true-rip,lem:degree,lem:loo-probability}, together with $q\ge C_4\mu r\log n/n$ with $C_4\ge128$ and $n\ge2$, gives total failure probability at most
\begin{equation}\label{eq:fixed-rgn-failure}
\frac{n^{-10}}{24}+2ne^{-qn/3}+\frac{n^{-10}}{48} <\frac{n^{-10}}{12}.
\end{equation}
The argument applies to every realization of $\bX_0$ satisfying $\sharpnorm{\bX_0-\bX_\star}\le\sigma_r(\bX_\star)/(1000\mu r)$. Therefore, \eqref{eq:fixed-rgn-failure} gives the conditional probability asserted in \Cref{thm:fixed} and completes the proof.
\end{proof}

\subsection{Proof of Theorem~\ref{thm:end}}
\label{sec:proof-global-rgn}

\begin{proof}
Choose $C_2\ge32(C_4+C_5+1)$. Since
\[
K+2\le16\log(2\mu r\kappa), \qquad p=1-(1-q)^{K+2}\le(K+2)q,
\]
the sampling condition gives
\[
q\ge \frac{C_2 \mu r\log n}{16n}.
\]
It also implies $K+2\le n$. Thus \Cref{thm:init} gives, with failure probability at most $n^{-10}/3$,
\[
\sharpnorm{\bX_0-\bX_\star} \le\frac{\sigma_r(\bX_\star)}{1000\mu r}.
\]
The stopping index and the returned initializer depend only on $\Omega^{(0)},\ldots,\Omega^{(K)}$ and the independent Gaussian matrices, and are therefore independent of $\widehat\Omega$. Conditional on a successful initialization, \Cref{thm:fixed} fails with probability at most $n^{-10}/12$. Hence the union bound gives the convergence statement in \Cref{thm:end} fails  with probability at most $\frac{n^{-10}}3+\frac{n^{-10}}{12}<n^{-10}$.
\end{proof}

\subsection{Proof of \Cref{cor:regularized-rgn}}
\label{sec:proof-reg-rgn}
\begin{proof}
The proof is based on the event in the proof of \Cref{thm:end}. By Step~2 in the proof of \Cref{thm:fixed}, we have $\norm{\bX_{\bar K}-\bX_\star}_{\F}
\le\frac{\sigma_r(\bX_\star)}{40n}$. For $C_2$ sufficiently large, the sampling condition also gives $\norm{\bX_{\bar K}-\bX_\star}_{\F} \le\frac{\sqrt q}{128}\sigma_r(\bX_\star)$. Fix $k\ge\bar K$ and suppose that
\[
e_k:=\norm{\bX_k-\bX_\star}_{\F} \le\frac{\sqrt q}{128}\sigma_r(\bX_\star).
\]
Let $\widehat{\bm\xi}_k$ denote the ordinary RGN direction at $\bX_k$ and set
\[
\widehat{\bm\eta}_k := \widehat{\bm\xi}_k-\bm t_{\bX_k}.
\]
The argument in Step~3 of the proof of \Cref{thm:fixed}, together with \eqref{eq:fixed-rgd-transfer-small}, gives
\[
\norm{\widehat{\bm\eta}_k}_{\F} \le \frac{3\sqrt5}{2} \frac{e_k^2}{\sqrt q\,\sigma_r(\bX_\star)}, \qquad \norm{\widehat{\bm\xi}_k}_{\F} \le\frac{17}{16}e_k.
\]
Moreover, by the definition of $\lambda_k$, the ordinary normal equation, \eqref{eq:fixed-rgd-transfer-small}, and Weyl's inequality,
\[
\frac{\lambda_k}{q} = \frac{ \norm{ \Pcal_{T_{\bX_k}}q^{-1}\Pcal_{\widehat{\Omega}} \Pcal_{T_{\bX_k}}\widehat{\bm\xi}_k }_{\F}} {\sigma_r(\bX_k)} \le \frac54\frac{e_k}{\sigma_r(\bX_\star)}.
\]
Subtracting the ordinary and regularized normal equations gives
\[
\left( \Pcal_{T_{\bX_k}}q^{-1}\Pcal_{\widehat{\Omega}} \Pcal_{T_{\bX_k}} +\frac{\lambda_k}{q}\Ical \right) (\bm\xi_k-\widehat{\bm\xi}_k) = -\frac{\lambda_k}{q}\widehat{\bm\xi}_k.
\]
Again by \eqref{eq:fixed-rgd-transfer-small}, we have
\[
\norm{\bm\xi_k-\widehat{\bm\xi}_k}_{\F} \le \frac{85}{56} \frac{e_k^2}{\sigma_r(\bX_\star)}.
\]
Hence, with $\bm\eta_k:=\bm\xi_k-\bm t_{\bX_k}$, we have
\[
\norm{\bm\eta_k}_{\F} \le \norm{\widehat{\bm\eta}_k}_{\F} +\norm{\bm\xi_k-\widehat{\bm\xi}_k}_{\F} < 5\frac{e_k^2}{\sqrt q\,\sigma_r(\bX_\star)} < \frac{\sigma_r(\bX_\star)}{320}.
\]
Part~\textup{(ii)} of \Cref{lem:normal-F} therefore applies and yields
\[
\begin{aligned}
\norm{\bX_{k+1}-\bX_\star}_{\F} \le \norm{\bm\eta_k}_{\F} \left( 1+\frac{11}{5}\frac{e_k}{\sigma_r(\bX_\star)} +\frac{11}{10}\frac{\norm{\bm\eta_k}_{\F}} {\sigma_r(\bX_\star)} \right) \le \bar C_2 \frac{e_k^2}{\sqrt q\,\sigma_r(\bX_\star)}
\end{aligned}
\]
for an absolute constant $\bar C_2>0$. The right-hand side is at most $\sqrt q\,\sigma_r(\bX_\star)/128$ for $C_2$ sufficiently large. Thus the stated neighborhood is invariant, and induction proves the result for every $k\ge\bar K$.
\end{proof}

\section{Numerical Experiments}
\label{sec:numerics}

In this section, we compare factorized gradient descent (FGD), ScaledGD~\cite{TongMaChi2021}, RGD, and RGN on randomly generated matrix completion problems. All computations were performed in MATLAB R$2024$b on $64$-bit Windows,  with an Intel Core Ultra~$5$ $125$H CPU and $16$ GB memory. We measure the reconstruction error by
\[
\operatorname{err}_k:=\frac{\norm{\bX_k-\bX_\star}_{\F}}{\norm{\bX_\star}_{\F}}.
\]

\paragraph{Convergence and running time.} We first compare the convergence and running time under different condition numbers. We take $n=1000$, $r=10$, and $p=0.2$. Let $\bU_\star$ contain the left singular vectors of an $n\times r$ matrix with independent random signs. For $\kappa\in\{1,10,100\}$, we set
\[
\bX_\star=\bU_\star\operatorname{diag}(\sigma_1,\ldots,\sigma_r)\bU_\star^T,
\]
where the nonzero singular values are linearly spaced from $1$ to $1/\kappa$. Each entry is observed independently with probability $p$ without noise. We use the same $\bU_\star$ and $\Omega$ for all three values of $\kappa$; for each $\kappa$, the four methods receive the same observations $\Pcal_\Omega(\bX_\star)$. FGD and ScaledGD use the same spectral initialization and the respective updates in \cite{TongMaChi2021}, with stepsize $0.5$ for both methods. RGD and RGN share the initialization in \Cref{alg:init}, with $q=p$, $\Omega^{(\ell)}=\Omega$, and $K=8$, returning the rank-$r$ candidate with the smallest observed residual. In \eqref{eq:threshold-iteration}, we additionally stop when
\[
\left(1-r^{-1}\norm{\bQ_{\ell,t-1}^T\bQ_{\ell,t}}_{\F}^2\right)^{1/2} \le10^{-3}
\]
holds at two consecutive iterations with $t\ge3$; the iteration limit remains $\ceil{12\log n}$. RGD and RGN follow \Cref{alg:fixed-rgd,alg:regularized-rgn}, respectively, with $\widehat\Omega=\Omega$. For RGN, CG is stopped at relative residual $0.05$ or after $200$ iterations. All four methods terminate when
\[
\frac{\norm{\Pcal_\Omega(\bX_k-\bX_\star)}_{\F}} {\norm{\Pcal_\Omega(\bX_\star)}_{\F}}\le10^{-14},
\]
in at most $1000$ iterations. \Cref{fig:convergence-time} plots $\operatorname{err}_k$ against the iteration count and elapsed time for one realization at each $\kappa$.  Elapsed time includes initialization and the subsequent iterations. 

\begin{figure}[!t]
\centering
\includegraphics[width=0.45\textwidth]{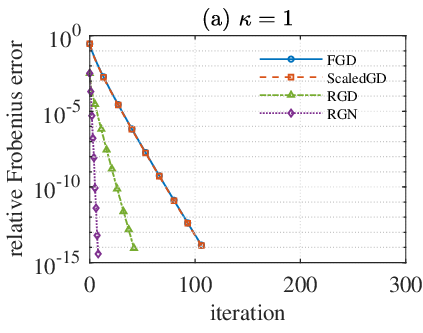}
\hfill
\includegraphics[width=0.45\textwidth]{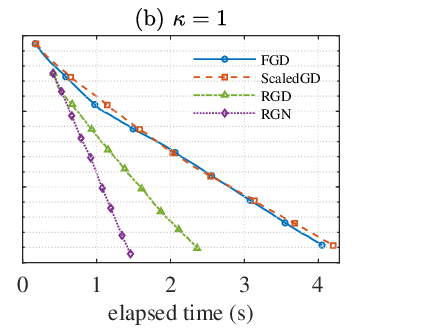}
\includegraphics[width=0.45\textwidth]{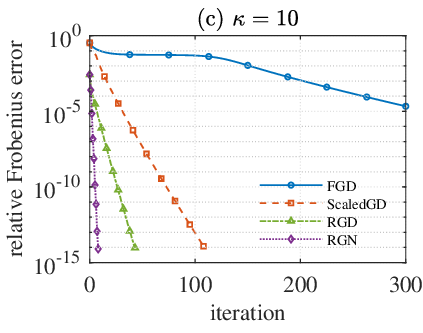}
\hfill
\includegraphics[width=0.45\textwidth]{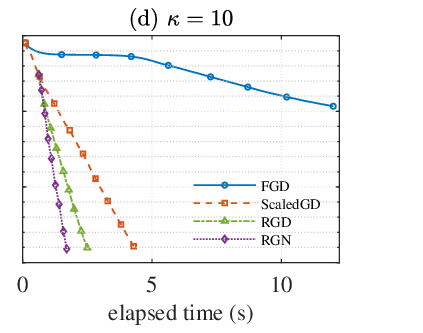}
\includegraphics[width=0.45\textwidth]{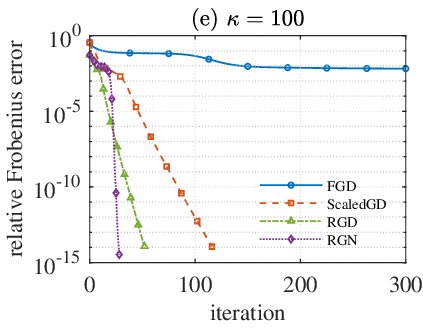}
\hfill
\includegraphics[width=0.45\textwidth]{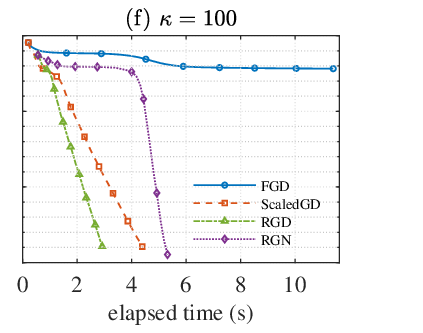}
\caption{Relative reconstruction errors of FGD, ScaledGD, RGD, and RGN
versus iteration count (left) and elapsed time in seconds (right), with
$n=1000$, $r=10$, and $p=0.2$. From top to bottom, $\kappa=1,10,100$.}
\label{fig:convergence-time}
\end{figure}

\paragraph{Empirical recovery rates.} We next examine how the number of observations needed for recovery varies with the rank. We set $n=500$ and conduct $30$ random trials for each pair $r\in\{2,4,\ldots,30\}$, $m/n\in\{4,8,\ldots,84\}$. In each trial, let $\bU_{\star,r}$ and $\bV_{\star,r}$ consist of the first $r$ columns of two independent $n\times30$ Haar orthonormal frames, and form $\bX_\star=\bU_{\star,r}\bV_{\star,r}^T$. We observe $m$ entries uniformly without replacement. For this experiment, the search directions are those in \Cref{alg:fixed-rgd,alg:fixed}, with $\widehat\Omega=\Omega$.   The relative observed residual tolerances are $10^{-12}$ for RGD and $10^{-8}$ for RGN. Recovery is declared successful if $\operatorname{err}_k\le10^{-6}$. \Cref{fig:phase-transition} reports the fraction of successful trials at each $(r,m)$.

\begin{figure}[!t]
\centering
\includegraphics[height=240bp]{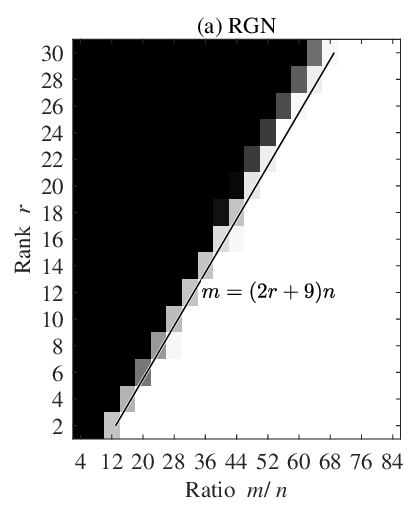}
\hfill
\includegraphics[height=240bp]{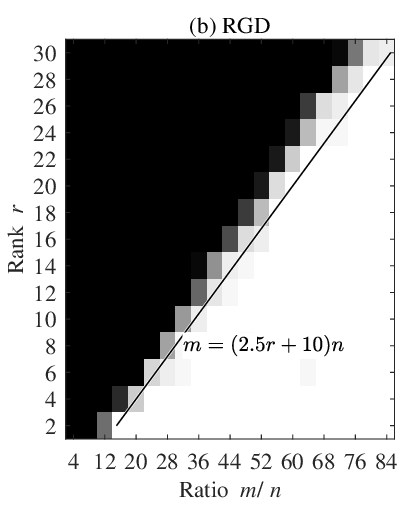}
\hfill
\includegraphics[height=240bp]{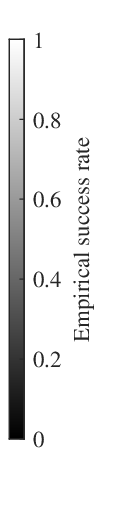}
\caption{Empirical recovery rates of RGN (left) and RGD (middle) for $n=500$
and $\kappa=1$, over $30$ trials at each $(r,m)$. The horizontal axis is
$m/n$ and the vertical axis is $r$. White denotes success in all trials,
and black denotes failure in all trials. The solid reference lines are
$m=(2r+9)n$ for RGN and $m=(2.5r+10)n$ for RGD.}
\label{fig:phase-transition}
\end{figure}
Both panels exhibit a transition from failure to successful recovery as $m$ increases.  These results show an approximately linear dependence on $nr$ over the tested ranks. The solid lines in \Cref{fig:phase-transition} provide simple reference relations for this observed transition.

\section{Concluding Remarks}\label{sec:conclusion}

We have established global recovery guarantees for RGD and RGN under standard incoherence and Bernoulli sampling. With high probability, $O(\mu nr\log n\log(n\kappa))$ and $O(\mu nr\log n\log(2\mu r\kappa))$ observations suffice for the two methods, respectively. The sample requirements are linear in $n$ and $r$, up to logarithmic factors. RGD converges linearly, while RGN satisfies a doubly exponential sharp-norm error bound during its initial iterations and a quadratic Frobenius recurrence thereafter. Both methods use the multiscale residual initialization with residual stopping and different upper bounds on the number of reconstructions for their respective local conditions. The observation set is fixed at the beginning, and the sampling requirement does not depend on the target accuracy.

Nevertheless, several extensions merit further study. One is to establish the same guarantees when every initialization step reuses the entire observation set. Another is to analyze RGN with an inexact tangent solve and a computable stopping criterion, so that the convergence rate can be related to the total number of CG iterations.

\section*{Declaration on the use of artificial intelligence}
Generative AI tools were used during revision to assist with language editing,
structural reorganization, bibliographic cross-checking, and consistency checks
of notation, formulas, and proofs. 
\appendix

\section{Lower bound for ordinary spectral initialization}
\label{app:optimality}
\label{sec:proof-necessary-scales}

We prove \Cref{thm:necessary-scales} by constructing a block-diagonal family of incoherent matrices. 

\begin{proof}[Proof of \Cref{thm:necessary-scales}]
Take $C_3=2^{-12}$. Choose pairwise disjoint sets $I_1,\ldots,I_r\subseteq\{1,\ldots,n\}$ with $|I_a|=s$, and set
\[
\bm u_a=\bm v_a=s^{-1/2}\bm1_{I_a}, \qquad \bX_\star=\kappa\bm u_1\bm v_1^T+\sum_{a=2}^r\bm u_a\bm v_a^T.
\]
Then
\[
\sigma_r(\bX_\star)=1, \qquad \kappa(\bX_\star)=\kappa, \qquad \max_i\norm{\bm e_i^T\bU_\star}_2^2 =\max_j\norm{\bm e_j^T\bV_\star}_2^2 =\frac1s=\frac{\mu r}{n}.
\]
Thus \Cref{ass:incoherence} holds. Set $d=qs$ and $\bY=q^{-1}\Pcal_\Lambda(\bX_\star)$. On the active coordinates,
\[
\bY=\operatorname{diag}(\bY_1,\ldots,\bY_r), \qquad \bY_1=\frac{\kappa}{d}\bm\Delta_1, \qquad \bY_a=\frac1d\bm\Delta_a,\quad 2\le a\le r,
\]
where the $\bm\Delta_a\in\{0,1\}^{s\times s}$ have independent Bernoulli$(q)$ entries. Since $\mu=n/(rs)$ and $\kappa^2\le s$,
\[
d=qs<C_3\kappa^2, \qquad q<C_3\frac{\kappa^2}{s}\le C_3<\frac12.
\]
Suppose first that $d\le\frac14\log(rs)$. For each active row,
\[
\pi:=\Pr\{\text{the row is empty on its support}\} =(1-q)^s\ge e^{-2qs}=e^{-2d}\ge(rs)^{-1/2}.
\]
Hence
\[
\Pr\{\text{an active row is empty}\} =1-(1-\pi)^{rs} \ge1-e^{-rs\pi} \ge1-e^{-\sqrt{rs}}>\frac12.
\]
If $i\in I_a$ is such a row, then
\[
\bm e_i^T\bY=\bm0 \quad\Longrightarrow\quad \bm e_i^T\Hcal_r(\bY)=\bm0, \qquad \norm{\bm e_i^T\bX_\star}_2\ge s^{-1/2},
\]
and therefore
\[
\sharpnorm{\Hcal_r(\bY)-\bX_\star} \ge\frac12\sqrt{s}\,\norm{\bm e_i^T(\Hcal_r(\bY)-\bX_\star)}_2 \ge\frac12.
\]
Now suppose that $d>\frac14\log(rs)$. Let $\mathcal D$ be the event that every row and column degree of every $\bm\Delta_a$ is at most $16d$. The multiplicative Chernoff bound gives
\[
\Pr(\mathcal D^c) \le2rs\exp\bigl(-(16\log16-15)d\bigr) <2(rs)^{1-(16\log16-15)/4} \le\frac1{30}.
\]
With $M=\norm{\bm\Delta_1}_{\F}^2$, we have
\[
\E M=sd, \qquad \E M^2\le s^2d^2+sd, \qquad sd>\frac{d^2}{C_3}>\frac{\log^2(rs)}{16C_3}.
\]
Hence the Paley--Zygmund inequality yields
\[
\Pr\{\mathcal D\cap\{M\ge sd/4\}\} \ge\frac9{16}\frac{sd}{sd+1}-\Pr(\mathcal D^c) >\frac12.
\]
On $\mathcal D$, the maximum row sum and the maximum column sum of
$\bm\Delta_a$ are both at most $16d$. Hence
\[
\norm{\bm\Delta_a}_{\op} \le \sqrt{ \left(\max_i\sum_j|(\bm\Delta_a)_{ij}|\right) \left(\max_j\sum_i|(\bm\Delta_a)_{ij}|\right)} \le16d, \qquad 1\le a\le r,
\]
and therefore
\begin{align*}
\norm{\bY_1}_{\F}^2 &\ge\frac{\kappa^2s}{4d}, &\norm{\bY_1}_{\op}&\le16\kappa,\\
\max_{2\le a\le r}\norm{\bY_a}_{\op} &\le16, &\sum_{j\ge2}\sigma_j(\bY_1)^2 &\ge\frac{\kappa^2s}{4d}-256\kappa^2>\left(\frac1{4C_3}-256\right)s>256(s-1).
\end{align*}
Thus
\[
\sigma_2(\bY_1)>16\ge\max_{2\le a\le r}\norm{\bY_a}_{\op}.
\]
If $\bY_1$ has at least $r$ singular values larger than $16$, then every best rank-at-most-$r$ approximation of $\bY$ is supported on $I_1\times I_1$, and hence
\[
\sharpnorm{\Hcal_r(\bY)-\bX_\star} \ge\norm{\Pcal_{(I_2\cup\cdots\cup I_r)\times(I_2\cup\cdots\cup I_r)}\bX_\star}_{\op}=1.
\]
Otherwise, the Eckart--Young--Mirsky theorem and $\sigma_2(\bY_1)>16$ imply, with $I=I_2\cup\cdots\cup I_r$,
\[
\rank\!\left(\Pcal_{I\times I}\Hcal_r(\bY)\right)\le r-2.
\]
Since $\Pcal_{I\times I}\bX_\star$ has $r-1$ singular values equal to one,
\[
\sharpnorm{\Hcal_r(\bY)-\bX_\star} \ge\norm{\Pcal_{I\times I}(\Hcal_r(\bY)-\bX_\star)}_{\op} \ge1.
\]
In both cases,
\[
\Pr\!\left\{\sharpnorm{\Hcal_r(\bY)-\bX_\star}\ge\frac12\right\}>\frac12.
\]
This proves the theorem.
\end{proof}

\section{Probabilistic Lemmas}\label{app:tools}

We collect the concentration tools used throughout the sampling arguments. We use the following standard forms of the rectangular and self-adjoint matrix Bernstein inequalities; see \cite[Theorems~1.6 and~6.1]{Tropp2012}.

\begin{lemma}\label[lemma]{lem:bernstein}
Let $\bZ_1,\ldots,\bZ_N\in\R^{d_1\times d_2}$ be independent mean-zero random matrices. Suppose that $\norm{\bZ_a}_{\op}\le L$ almost surely, and define
\[
v:=\max\left\{ \norm{\sum_a\E(\bZ_a\bZ_a^T)}_{\op}, \norm{\sum_a\E(\bZ_a^T\bZ_a)}_{\op} \right\}.
\]
Then, for every $t>0$,
\[
\Pr\left\{ \norm{\sum_a\bZ_a}_{\op} > \sqrt{2vt}+\frac23Lt \right\} \le(d_1+d_2)e^{-t}.
\]
If the $\bZ_a$ are self-adjoint operators on a $d$-dimensional Hilbert space, then, for every $u>0$,
\[
\Pr\left\{ \norm{\sum_a\bZ_a}_{\op}\ge u \right\} \le 2d\exp\left(-\frac{u^2}{2(v+Lu/3)}\right).
\]
\end{lemma}

\subsection[Proof of the true-tangent estimate]{Proof of \Cref{lem:true-tangent-rip}}
\label{sec:proof-true-tangent-rip}

We next prove the true-tangent sampling isometry used in the local RGD analysis.

\begin{proof}[Proof of \Cref{lem:true-tangent-rip}]
Fix $c_1=2^{15}$. Let $\delta_{ij}:=\mathbf 1_{\{(i,j)\in\Lambda\}}$ and $\bm a_{ij}=\Pcal_{T_{\bX_\star}}(\bm e_i\bm e_j^T)$, and define $(\bm a\otimes\bm a)(\bZ):=\inner{\bm a}{\bZ}_{\F}\bm a$. Since the matrices $\bm e_i\bm e_j^T$ form an orthonormal basis of $\R^{n\times n}$, we have
\[
\sum_{i,j}\bm a_{ij}\otimes\bm a_{ij} =\Ical_{T_{\bX_\star}}.
\]
Consequently, we have
\[
\Pcal_{T_{\bX_\star}}p_\Lambda^{-1}\Pcal_\Lambda \Pcal_{T_{\bX_\star}}-\Pcal_{T_{\bX_\star}} =\sum_{i,j}\left(\frac{\delta_{ij}}{p_\Lambda}-1\right) (\bm a_{ij}\otimes\bm a_{ij}).
\]
By \eqref{eq:tangent-projector} and the orthogonality of its two summands, we have
\begin{align*}
\bm a_{ij}&=\bU_\star\bU_\star^T\bm e_i\bm e_j^T +(\bI-\bU_\star\bU_\star^T)\bm e_i\bm e_j^T \bV_\star\bV_\star^T,\\
\norm{\bm a_{ij}}_{\F}^2 &=\norm{\bU_\star\bU_\star^T\bm e_i}_2^2 +\norm{(\bI-\bU_\star\bU_\star^T)\bm e_i}_2^2 \norm{\bV_\star\bV_\star^T\bm e_j}_2^2 \le\frac{2\mu r}{n}.
\end{align*}
Each summand has norm at most $2\mu r/(np_\Lambda)$.  Moreover, $\E(\delta_{ij}/p_\Lambda-1)^2=(1-p_\Lambda)/p_\Lambda$ and $(\bm a\otimes\bm a)^2 =\norm{\bm a}_{\F}^2(\bm a\otimes\bm a)$, so the variance operator satisfies
\[
\frac{1-p_\Lambda}{p_\Lambda}\sum_{i,j} \norm{\bm a_{ij}}_{\F}^2(\bm a_{ij}\otimes\bm a_{ij}) \preceq\frac{2\mu r}{np_\Lambda}\Ical_{T_{\bX_\star}}.
\]
Since $\dim(T_{\bX_\star})=r(2n-r)\le2nr$, $p_\Lambda\ge c_1\mu r\log n/n$, $r\le n$, and $n\ge2$, \Cref{lem:bernstein} at threshold $1/16$ gives
\begin{align*}
\Pr\left\{ \norm{\Pcal_{T_{\bX_\star}}p_\Lambda^{-1}\Pcal_\Lambda \Pcal_{T_{\bX_\star}}-\Pcal_{T_{\bX_\star}}}_{\F\to\F} >\frac1{16}\right\} &\le\frac{n^{-10}}{48}.
\end{align*}
This proves the lemma.
\end{proof}

\subsection[Proof of the residual sampling estimate]{Proof of \Cref{lem:spectral-concentration}}\label{app:fresh}
\label{sec:proof-spectral-concentration}

We prove the sampling estimates for a fixed error matrix. Independence from the current observation component is supplied by conditioning in \Cref{sec:proof-finite-rgd}.

\begin{proof}[Proof of \Cref{lem:spectral-concentration}]
By \eqref{eq:sharp-main},
\[
\norm{\bE}_{\infty}\le\frac{4\mu r}{n}\tau, \qquad \max\{\norm{\bE}_{2,\infty},\norm{\bE^T}_{2,\infty}\} \le2\sqrt{\frac{\mu r}{n}}\tau.
\]
Hence, for $\bW=(q^{-1}\Pcal_\Lambda-\Ical)\bE$,
\begin{equation}\label{eq:residual-variance}
|W_{ij}|\le\frac{4\mu r\tau}{nq}, \qquad \max\left\{\max_i\sum_j\E W_{ij}^2,\max_j\sum_i\E W_{ij}^2\right\} \le\frac{4\mu r\tau^2}{nq}.
\end{equation}
By \Cref{lem:bernstein},
\begin{equation}\label{eq:residual-op-probability}
\Pr\!\left\{\norm{\mathscr W}_{\op}>\frac{\tau}{512}\right\} \le4n\exp\!\left(-c\frac{nq}{\mu r}\right).
\end{equation}

For $|z|=\tau/64$, set
\[
\bm F(z)=(z\bI-\mathscr W)^{-1}\mathscr Q_\star.
\]
For $i\in[2n]$, let $\mathscr W^{(i)}$ be the principal minor obtained by deleting row and column $i$, let $\bm w_i$ be the deleted off-diagonal column, and define
\[
\bm F^{(i)}(z)=
\begin{cases}
(z\bI-\mathscr W^{(i)})^{-1}(\mathscr Q_\star)_{-i,:}, &\norm{\mathscr W^{(i)}}_{\op}\le\tau/512,\\
\bm0,&\text{otherwise}.
\end{cases}
\]
Conditional on $\mathscr W^{(i)}$, \eqref{eq:residual-variance} and \Cref{lem:bernstein}, applied to the real and imaginary parts, give
\begin{equation}\label{eq:residual-row-probability}
\Pr\!\left\{ \norm{\bm w_i^T\bm F^{(i)}(z)}_2> \frac{\tau}{1024}\norm{\bm F^{(i)}(z)}_{2,\infty} \,\middle|\,\mathscr W^{(i)}\right\} \le(4r+1)\exp\!\left(-c\frac{nq}{\mu r}\right).
\end{equation}

Let $\Gamma$ be a uniform grid on $|z|=\tau/64$ with
\[
|\Gamma|\le52\sqrt n, \qquad \min_{z_0\in\Gamma}|z-z_0| \le\frac{\tau}{1024}\sqrt{\frac{\mu r}{n}}.
\]
A union bound in \eqref{eq:residual-op-probability}--\eqref{eq:residual-row-probability} yields, except with probability $Cn^3\exp(-cnq/(\mu r))$,
\begin{equation}\label{eq:residual-grid-event}
\norm{\mathscr W}_{\op}\le\frac{\tau}{512}, \qquad \norm{\bm w_i^T\bm F^{(i)}(z)}_2 \le\frac{\tau}{1024}\norm{\bm F^{(i)}(z)}_{2,\infty}
\end{equation}
simultaneously for $i\in[2n]$ and $z\in\Gamma$. Increasing $c_2$ makes this failure probability at most $n^{-12}/48$. Fix a realization in \eqref{eq:residual-grid-event}. Then
\[
\norm{(z\bI-\mathscr W^{(i)})^{-1}}_{\op}\le\frac{512}{7\tau}, \qquad \norm{\bm w_i}_2\le\frac{\tau}{512},
\]
and
\begin{align*}
\bm F_{-i,:}(z) &=\bm F^{(i)}(z)+(z\bI-\mathscr W^{(i)})^{-1}\bm w_i\bm F_{i,:}(z),\\
\norm{\bm F^{(i)}(z)}_{2,\infty} &\le\frac87\norm{\bm F(z)}_{2,\infty},\\
\bm F_{i,:}(z) &=\frac{(\mathscr Q_\star)_{i,:}+\bm w_i^T\bm F^{(i)}(z)} {z-\bm w_i^T(z\bI-\mathscr W^{(i)})^{-1}\bm w_i}.
\end{align*}
Moreover,
\[
\left|z-\bm w_i^T(z\bI-\mathscr W^{(i)})^{-1}\bm w_i\right| \ge\frac{55\tau}{3584}>\frac{\tau}{66}.
\]
Using \Cref{ass:incoherence} and \eqref{eq:residual-grid-event},
\[
\norm{\bm F(z)}_{2,\infty} \le\frac{66}{\tau}\left( \sqrt{\frac{\mu r}{n}}+ \frac{\tau}{896}\norm{\bm F(z)}_{2,\infty}\right) \le\frac{72}{\tau}\sqrt{\frac{\mu r}{n}}, \qquad z\in\Gamma.
\]
For arbitrary $|z|=\tau/64$, choose $z_0\in\Gamma$ as above. The resolvent identity gives
\[
\norm{\bm F(z)-\bm F(z_0)}_{\op} \le\left(\frac{512}{7\tau}\right)^2|z-z_0| \le\frac6\tau\sqrt{\frac{\mu r}{n}},
\]
and hence
\begin{equation}\label{eq:residual-resolvent-bound}
\sup_{|z|=\tau/64} \norm{(z\bI-\mathscr W)^{-1}\mathscr Q_\star}_{2,\infty} \le\frac{128}{\tau}\sqrt{\frac{\mu r}{n}}.
\end{equation}
Finally,
\[
\mathscr W^j\mathscr Q_\star =\frac{1}{2\pi\mathrm i}\oint_{|z|=\tau/64} z^j(z\bI-\mathscr W)^{-1}\mathscr Q_\star\,dz, \qquad j\ge0.
\]
Combining this identity with \eqref{eq:residual-resolvent-bound} proves \eqref{eq:spectral-power-bounds}.
\end{proof}

\section{Deterministic geometry of fixed-rank matrices}\label{app:geometry}

We prove the geometric estimates used in \Cref{sec:local-rgd,sec:initialization,sec:local-rgn}. We first derive the graph-retraction identity and the sharp-norm bounds at one matrix. We then prove the Frobenius estimates and the comparison bounds for two nearby matrices.

\subsection{Tangent coordinates and the graph-retraction identity}

Let $\bX=\bU\bm\Sigma\bV^T\in\Mcal_r$ and $\bm\xi\in T_{\bX}\Mcal_r$. Define the orthogonal tangent coordinates
\[
\bM:=\bU^T\bm\xi\bV,\qquad \bB:=(\bI-\bU\bU^T)\bm\xi\bV,\qquad \bC:=(\bI-\bV\bV^T)\bm\xi^T\bU.
\]
Then $\bU^T\bB=0$, $\bV^T\bC=0$, and
\[
\bm\xi=\bU\bM\bV^T+\bB\bV^T+\bU\bC^T.
\]
Moreover, $\bU^T(\bX+\bm\xi)\bV=\bm\Sigma+\bM$. Hence, whenever $\bm\Sigma+\bM$ is nonsingular, \eqref{eq:graph-retraction-opt} gives
\[
\operatorname{Retr}_{\bX}(\bm\xi) = \bigl(\bU+\bB(\bm\Sigma+\bM)^{-1}\bigr) (\bm\Sigma+\bM) \bigl(\bV+\bC(\bm\Sigma+\bM)^{-T}\bigr)^T.
\]
Expanding the product yields
\begin{equation}\label{eq:graph-retraction-identity}
\operatorname{Retr}_{\bX}(\bm\xi) = \bX+\bm\xi+\bB(\bm\Sigma+\bM)^{-1}\bC^T.
\end{equation}

For later two-base comparisons, we also record the invariance under orthogonal changes of basis. For a compact orthonormal factorization $\bX=\bU\bS\bV^T$, orthogonal changes of basis $\bU\mapsto\bU\bm O_U$ and $\bV\mapsto\bV\bm O_V$ leave $\bX$, $\Pcal_{T_{\bX}}$, and the graph retraction unchanged when the core and tangent coordinates are transformed accordingly. Thus, when comparing two matrices, we align their left and right bases separately by orthogonal Procrustes transformations and transform their cores at the same time. The estimates below are invariant under these choices; in particular, $\sigma_{\min}(\bS)=\sigma_r(\bX)$.

\subsection{Sharp-norm estimates}

We begin by controlling the current singular spaces and their projectors in the sharp-norm neighborhood.

\subsubsection{Proof of Lemma~\ref{lem:current-incoherence}}
\label{sec:proof-current-incoherence}

\begin{proof}
Let $\bE:=\bX_\star-\bX$. Weyl's inequality \cite[Supplement, Theorem~2]{CaiWuXia2025} gives $\sigma_r(\bX)\ge\sigma_r(\bX_\star)-e_\sharp$.  Since $(\bI-\bU_\star\bU_\star^T)\bX_\star=0$ and $(\bI-\bV_\star\bV_\star^T)\bX_\star^T=0$, we have
\begin{subequations}\label{eq:current-subspace-identities}
\begin{align}
(\bI-\bU_\star\bU_\star^T)\bU & =-(\bI-\bU_\star\bU_\star^T)\bE\bV\bS^{-1}, \label{eq:current-left-subspace-identity}\\
(\bI-\bV_\star\bV_\star^T)\bV & =-(\bI-\bV_\star\bV_\star^T)\bE^T\bU\bS^{-T}. \label{eq:current-right-subspace-identity}
\end{align}
\end{subequations}
For every $i$ and $j$, \Cref{ass:incoherence} gives
\[
\max\left\{ \norm{\bm e_i^T\bU_\star\bU_\star^T\bU}_2, \norm{\bm e_j^T\bV_\star\bV_\star^T\bV}_2 \right\} \le\sqrt{\frac{\mu r}{n}}.
\]
Equations~\eqref{eq:current-left-subspace-identity} and \eqref{eq:current-right-subspace-identity}, together with $e_\sharp\le\sigma_r(\bX_\star)/8$, give
\begin{align*}
\norm{\bm e_i^T(\bI-\bU_\star\bU_\star^T)\bU}_2 &\le\frac{\norm{\bm e_i^T\bE}_2 +\norm{\bU_\star^T\bm e_i}_2\norm{\bU_\star^T\bE}_{\op}} {\sigma_r(\bX_\star)-e_\sharp} \le\frac{3\sqrt{\mu r/n}\,e_\sharp} {\sigma_r(\bX_\star)-e_\sharp},\\
\norm{\bm e_i^T\bU}_2 &\le\sqrt{\frac{\mu r}{n}} \left(1+\frac{3e_\sharp}{\sigma_r(\bX_\star)-e_\sharp}\right) \le\frac{10}{7}\sqrt{\frac{\mu r}{n}},\\
\norm{\bm e_j^T(\bI-\bV_\star\bV_\star^T)\bV}_2 &\le\frac{\norm{\bm e_j^T\bE^T}_2 +\norm{\bV_\star^T\bm e_j}_2\norm{\bV_\star^T\bE^T}_{\op}} {\sigma_r(\bX_\star)-e_\sharp} \le\frac{3\sqrt{\mu r/n}\,e_\sharp} {\sigma_r(\bX_\star)-e_\sharp},\\
\norm{\bm e_j^T\bV}_2 &\le\sqrt{\frac{\mu r}{n}} \left(1+\frac{3e_\sharp}{\sigma_r(\bX_\star)-e_\sharp}\right) \le\frac{10}{7}\sqrt{\frac{\mu r}{n}},\\
\max\{\norm{\bU}_{2,\infty},\norm{\bV}_{2,\infty}\} &\le\frac{10}{7}\sqrt{\frac{\mu r}{n}} <2\sqrt{\frac{\mu r}{n}}.
\end{align*}

For two orthogonal projectors of the same rank, the operator norm of their difference equals the largest sine of the principal angles.  Hence \eqref{eq:current-left-subspace-identity}, \eqref{eq:current-right-subspace-identity}, and $e_\sharp\le\sigma_r(\bX_\star)/8$ give
\begin{equation}\label{eq:current-projector-op-left}
\begin{aligned}
\norm{\bU\bU^T-\bU_\star\bU_\star^T}_{\op} =\norm{(\bI-\bU_\star\bU_\star^T)\bU}_{\op} &\le\frac{e_\sharp}{\sigma_r(\bX_\star)-e_\sharp},\\
\norm{\bV\bV^T-\bV_\star\bV_\star^T}_{\op} =\norm{(\bI-\bV_\star\bV_\star^T)\bV}_{\op} &\le\frac{e_\sharp}{\sigma_r(\bX_\star)-e_\sharp},\\
\max\left\{ \norm{\bU\bU^T-\bU_\star\bU_\star^T}_{\op}, \norm{\bV\bV^T-\bV_\star\bV_\star^T}_{\op} \right\} &\le\frac{8}{7}\frac{e_\sharp}{\sigma_r(\bX_\star)}.
\end{aligned}
\end{equation}

For the rowwise projector bounds, the two projector identities, \Cref{ass:incoherence}, $e_\sharp\le\sigma_r(\bX_\star)/8$, and $\max\{\norm{\bU}_{2,\infty},\norm{\bV}_{2,\infty}\} \le(10/7)\sqrt{\mu r/n}$ give
\begin{align*}
\bU\bU^T-\bU_\star\bU_\star^T &=(\bI-\bU_\star\bU_\star^T)\bU\bU^T -\bU_\star\bU_\star^T(\bI-\bU\bU^T),\\
\norm{\bU\bU^T-\bU_\star\bU_\star^T}_{2,\infty} &\le\frac{4\sqrt{\mu r/n}\,e_\sharp} {\sigma_r(\bX_\star)-e_\sharp} \le\frac{32}{7}\sqrt{\frac{\mu r}{n}} \frac{e_\sharp}{\sigma_r(\bX_\star)},\\
\bV\bV^T-\bV_\star\bV_\star^T &=(\bI-\bV_\star\bV_\star^T)\bV\bV^T -\bV_\star\bV_\star^T(\bI-\bV\bV^T),\\
\norm{\bV\bV^T-\bV_\star\bV_\star^T}_{2,\infty} &\le\frac{4\sqrt{\mu r/n}\,e_\sharp} {\sigma_r(\bX_\star)-e_\sharp} \le\frac{32}{7}\sqrt{\frac{\mu r}{n}} \frac{e_\sharp}{\sigma_r(\bX_\star)},\\
\max\Big\{ \norm{\bU\bU^T-\bU_\star\bU_\star^T}_{2,\infty}, & \norm{\bV\bV^T-\bV_\star\bV_\star^T}_{2,\infty} \Big\} \le\frac{32}{7}\sqrt{\frac{\mu r}{n}} \frac{e_\sharp}{\sigma_r(\bX_\star)}.
\end{align*}
Finally, the tangent-projector identity and \eqref{eq:current-projector-op-left} give
\begin{align*}
(\Pcal_{T_{\bX}}-\Pcal_{T_{\bX_\star}})\bZ &=(\bU\bU^T-\bU_\star\bU_\star^T)\bZ(\bI-\bV\bV^T) +(\bI-\bU_\star\bU_\star^T)\bZ(\bV\bV^T-\bV_\star\bV_\star^T),\\
\norm{\Pcal_{T_{\bX}}-\Pcal_{T_{\bX_\star}}}_{\F\to\F} &\le\frac{16}{7}\frac{e_\sharp}{\sigma_r(\bX_\star)}.
\end{align*}
This proves the lemma.
\end{proof}

We next control the normal component and the associated graph factors.

\subsubsection{Proof of Lemma~\ref{lem:normal}}
\label{sec:proof-normal}

\begin{proof}
Let $\bE:=\bX_\star-\bX$. Weyl's inequality and $\norm{\bG_{\bX}-\bS}_{\op}\le e_\sharp$ give
\[
\sigma_{\min}(\bG_{\bX}) \ge \sigma_r(\bX_\star)-2e_\sharp \ge\frac34\sigma_r(\bX_\star), \qquad \norm{\bG_{\bX}^{-1}}_{\op} \le\frac4{3\sigma_r(\bX_\star)}.
\]
Relative to the orthogonal decompositions generated by $\bU$ and $\bV$, the upper-left, lower-left, and upper-right blocks of $\bX_\star$ are $\bG_{\bX}$, $\bL_{\bX}$, and $\bR_{\bX}^T$, respectively.  Since $\rank(\bX_\star)=r$ and $\bG_{\bX}$ is invertible, the Schur complement of $\bG_{\bX}$ vanishes.  Therefore, the normal component is exactly
\begin{equation}\label{eq:normal-graph-factorization}
\bN_{\bX}=\bL_{\bX}\bG_{\bX}^{-1}\bR_{\bX}^T.
\end{equation}
Moreover, we have
\begin{align*}
\max\{\norm{\bL_{\bX}}_{\op},\norm{\bR_{\bX}}_{\op}\} \le e_\sharp,\qquad \norm{\bN_{\bX}}_{\op} \le\frac43\frac{e_\sharp^2}{\sigma_r(\bX_\star)}.
\end{align*}
For the rowwise graph factors, the definition of the sharp norm and the incoherence bound in \Cref{lem:current-incoherence} give, for every $i$, we have
\begin{equation}\label{eq:normal-graph-row-bounds}
\begin{aligned}
\norm{\bm e_i^T\bL_{\bX}}_2 \le \norm{\bm e_i^T\bE}_2 +\norm{\bm e_i^T\bU}_2\norm{\bU^T\bE}_{\op} &\le 4\sqrt{\frac{\mu r}{n}}\,e_\sharp,\\
\norm{\bm e_i^T\bR_{\bX}}_2 \le \norm{\bm e_i^T\bE^T}_2 +\norm{\bm e_i^T\bV}_2\norm{\bV^T\bE^T}_{\op} &\le 4\sqrt{\frac{\mu r}{n}}\,e_\sharp,\\
\max\{\norm{\bL_{\bX}}_{2,\infty}, \norm{\bR_{\bX}}_{2,\infty}\} &\le4\sqrt{\frac{\mu r}{n}}\,e_\sharp.
\end{aligned}
\end{equation}

Applying \eqref{eq:normal-graph-row-bounds} to the left and right factors in \eqref{eq:normal-graph-factorization} yields
\[
\max\left\{\norm{\bN_{\bX}}_{2,\infty}, \norm{\bN_{\bX}^T}_{2,\infty}\right\} \le\frac{16}{3}\sqrt{\frac{\mu r}{n}} \frac{e_\sharp^2}{\sigma_r(\bX_\star)}.
\]
For every $i,j$, using the rowwise bounds on both graph factors gives
\begin{align*}
\abs{\bm e_i^T\bN_{\bX}\bm e_j} &\le \norm{\bm e_i^T\bL_{\bX}}_2 \norm{\bG_{\bX}^{-1}}_{\op} \norm{\bm e_j^T\bR_{\bX}}_2 \le\frac{64\mu r}{3n} \frac{e_\sharp^2}{\sigma_r(\bX_\star)},\\
\norm{\bN_{\bX}}_{\infty} &\le\frac{64\mu r}{3n} \frac{e_\sharp^2}{\sigma_r(\bX_\star)}.
\end{align*}
Applying \eqref{eq:sharp-main} to the bounds for $\norm{\bN_{\bX}}_{\op}$, $\norm{\bN_{\bX}}_{2,\infty}$, $\norm{\bN_{\bX}^T}_{2,\infty}$, and $\norm{\bN_{\bX}}_\infty$ gives
\[
\sharpnorm{\bN_{\bX}} \le\frac{16}{3}\frac{e_\sharp^2}{\sigma_r(\bX_\star)},
\]
which proves the sharp-norm bound in \Cref{lem:normal}.
\end{proof}

We next use the graph factorization to verify exactness of the population tangent correction.

\subsubsection{Proof of Lemma~\ref{lem:population-exact}}
\label{sec:proof-population-exact}

\begin{proof}
Equation~\eqref{eq:tangent-projector} gives the tangent coordinates
\[
\bM=\bG_{\bX}-\bm\Sigma, \qquad \bB=\bL_{\bX}, \qquad \bC=\bR_{\bX}.
\]
Moreover, Weyl's inequality gives
\[
\sigma_{\min}(\bG_{\bX}) \ge \sigma_r(\bX_\star)-2\sharpnorm{\bX-\bX_\star} \ge \frac12\sigma_r(\bX_\star)>0,
\]
so the updated core is invertible. Since $\rank(\bX_\star)=r$, the Schur complement of $\bG_{\bX}$ in the block representation of $\bX_\star$ vanishes.  Using $\bm t_{\bX}=\Pcal_{T_{\bX}}(\bX_\star-\bX)$ in \eqref{eq:graph-retraction-identity} gives
\begin{align*}
\operatorname{Retr}_{\bX}(\bm t_{\bX}) &=\bX+\bm t_{\bX} +\bL_{\bX}\bG_{\bX}^{-1}\bR_{\bX}^T,\\
\Pcal_{T_{\bX}^{\perp}}(\bX_\star-\bX) &=\bL_{\bX}\bG_{\bX}^{-1}\bR_{\bX}^T,\\
\operatorname{Retr}_{\bX}(\bm t_{\bX}) &=\bX+\Pcal_{T_{\bX}}(\bX_\star-\bX) +\Pcal_{T_{\bX}^{\perp}}(\bX_\star-\bX) =\bX_\star.
\end{align*}
This proves the population identity.
\end{proof}

For the proofs of \Cref{lem:graph-correction,lem:normal-F}, let $\bX=\bU\bm\Sigma\bV^T\in\Mcal_r$ and $\bm\eta\in T_{\bX}\Mcal_r$, and define
\begin{equation}\label{eq:tangent-correction-coordinates}
\bM_\eta:=\bU^T\bm\eta\bV, \qquad \bB_\eta:=(\bI-\bU\bU^T)\bm\eta\bV, \qquad \bC_\eta:=(\bI-\bV\bV^T)\bm\eta^T\bU, \qquad \bS_\eta:=\bG_{\bX}+\bM_\eta.
\end{equation}
These coordinates describe the perturbation of the population tangent correction.

\subsubsection{Proof of Lemma~\ref{lem:graph-correction}}

\label{sec:proof-graph-correction}

\begin{proof}
By the definition of the sharp norm and \eqref{eq:tangent-correction-coordinates}, we have
\[
\begin{aligned}
\max\left\{ \norm{\bM_\eta}_{\op},\norm{\bB_\eta}_{\op},\norm{\bC_\eta}_{\op} \right\} &\le\sharpnorm{\bm\eta},\\
\max\left\{\norm{\bB_\eta}_{2,\infty}, \norm{\bC_\eta}_{2,\infty}\right\} &\le4\sqrt{\frac{\mu r}{n}}\,\sharpnorm{\bm\eta}.
\end{aligned}
\]
By \Cref{lem:normal}, we have $\norm{\bG_{\bX}^{-1}}_{\op}\le4/(3\sigma_r(\bX_\star))$. Since $\sigma_{\min}(\bG_{\bX})\ge3\sigma_r(\bX_\star)/4$ and $\sharpnorm{\bm\eta}\le\sigma_r(\bX_\star)/16$, we have
\[
\norm{(\bG_{\bX}+\bM_\eta)^{-1}}_{\op} \le\frac{16}{11\sigma_r(\bX_\star)}.
\]
Since $e_\sharp\le\sigma_r(\bX_\star)/8<\sigma_r(\bX_\star)/4$, \Cref{lem:population-exact} applies.  It follows from \Cref{lem:population-exact} and \eqref{eq:graph-retraction-identity} that
\begin{equation}\label{eq:graph-correction-remainder}
\begin{aligned}
&~\operatorname{Retr}_{\bX}(\bm t_{\bX}+\bm\eta)-\bX_\star-\bm\eta\cr =&~\bB_\eta\bS_\eta^{-1}\bR_{\bX}^T +\bL_{\bX}\bS_\eta^{-1}\bC_\eta^T +\bB_\eta\bS_\eta^{-1}\bC_\eta^T -\bL_{\bX}\bG_{\bX}^{-1}\bM_\eta\bS_\eta^{-1}\bR_{\bX}^T.
\end{aligned}
\end{equation}
For $\bm F,\bm H\in\R^{n\times r}$ and $\bm K\in\R^{r\times r}$, we have
\begin{equation}\label{eq:sharp-product-bound}
\begin{aligned}
\sharpnorm{\bm F\bm K\bm H^T} \le\norm{\bm K}_{\op}\max\Bigl\{& \norm{\bm F}_{\op}\norm{\bm H}_{\op}, \frac12\sqrt{\frac{n}{\mu r}}\norm{\bm F}_{2,\infty}\norm{\bm H}_{\op},\\
&\frac12\sqrt{\frac{n}{\mu r}}\norm{\bm F}_{\op}\norm{\bm H}_{2,\infty}, \frac{n}{4\mu r}\norm{\bm F}_{2,\infty}\norm{\bm H}_{2,\infty} \Bigr\}.
\end{aligned}
\end{equation}
Applying \eqref{eq:sharp-product-bound} to the four terms in \eqref{eq:graph-correction-remainder}, and using \eqref{eq:normal-factor-bounds-main}, $e_\sharp\le\sigma_r(\bX_\star)/8$, and $\sharpnorm{\bm\eta}\le\sigma_r(\bX_\star)/16$, gives
\[
\sharpnorm{\operatorname{Retr}_{\bX}(\bm t_{\bX}+\bm\eta) -\bX_\star-\bm\eta} \le13\frac{e_\sharp\sharpnorm{\bm\eta}}{\sigma_r(\bX_\star)} +6\frac{\sharpnorm{\bm\eta}^2}{\sigma_r(\bX_\star)}.
\]
This proves the nonlinear remainder bound.
\end{proof}

\subsection[Proof of the Frobenius geometry estimate]{Proof of \Cref{lem:normal-F}}
\label{sec:proof-normal-F}

We next prove the Frobenius estimates used in the local RGD argument and in the quadratic RGN continuation.

\begin{proof}[Proof of \Cref{lem:normal-F}]
We first prove part~\textup{(i)}. Let $\bX=\bU\bm\Sigma\bV^T$ be a compact singular value decomposition, and use the graph factors in \eqref{eq:target-graph-factors}. Weyl's inequality and $\norm{\bG_{\bX}-\bm\Sigma}_{\op}\le\norm{\bX-\bX_\star}_{\F}$ give
\[
\sigma_{\min}(\bG_{\bX})\ge\sigma_r(\bX_\star)-2\norm{\bX-\bX_\star}_{\F}\ge\frac12\sigma_r(\bX_\star), \qquad \norm{\bG_{\bX}^{-1}}_{\op}\le\frac2{\sigma_r(\bX_\star)}.
\]
Since $\rank(\bX_\star)=r$ and $\bG_{\bX}$ is invertible, the Schur complement of $\bG_{\bX}$ in the block representation of $\bX_\star$ vanishes. Moreover, $\norm{\bL_{\bX}}_{\F}\le\norm{\bX-\bX_\star}_{\F}$ and $\norm{\bR_{\bX}}_{\op}\le\norm{\bX-\bX_\star}_{\F}$. Therefore, we have
\[
\Pcal_{T_{\bX}^\perp}(\bX_\star-\bX)=\bL_{\bX}\bG_{\bX}^{-1}\bR_{\bX}^T, \qquad \norm{\Pcal_{T_{\bX}^\perp}(\bX_\star-\bX)}_{\F}\le2\frac{\norm{\bX-\bX_\star}_{\F}^2}{\sigma_r(\bX_\star)}.
\]
For the tangent-space motion, Weyl's inequality gives $\sigma_r(\bX)\ge\sigma_r(\bX_\star)-\norm{\bX-\bX_\star}_{\F}$. Hence the equal-rank projector identity gives
\begin{align*}
\norm{\bU\bU^T-\bU_\star\bU_\star^T}_{\op} &\le\frac{\norm{\bX-\bX_\star}_{\F}}{\sigma_r(\bX_\star)-\norm{\bX-\bX_\star}_{\F}} \le\frac{4\norm{\bX-\bX_\star}_{\F}}{3\sigma_r(\bX_\star)},\\
\norm{\bV\bV^T-\bV_\star\bV_\star^T}_{\op} &\le\frac{\norm{\bX-\bX_\star}_{\F}}{\sigma_r(\bX_\star)-\norm{\bX-\bX_\star}_{\F}} \le\frac{4\norm{\bX-\bX_\star}_{\F}}{3\sigma_r(\bX_\star)}.
\end{align*}
Using the tangent-projector identity, we consequently obtain
\[
\norm{\Pcal_{T_{\bX}}-\Pcal_{T_{\bX_\star}}}_{\F\to\F} \le\frac{8\norm{\bX-\bX_\star}_{\F}}{3\sigma_r(\bX_\star)} \le3\frac{\norm{\bX-\bX_\star}_{\F}}{\sigma_r(\bX_\star)}.
\]
This proves part~\textup{(i)}.

We next prove part~\textup{(ii)}. Let $\bX=\bU\bm\Sigma\bV^T$ be a compact singular value decomposition, and use the graph factors in \eqref{eq:target-graph-factors}. Weyl's inequality and the tangent-coordinate bound $\norm{\bM_\eta}_{\op}\le\norm{\bm\eta}_{\F}$ give
\[
\norm{\bG_{\bX}^{-1}}_{\op}\le\frac1{\sigma_r(\bX_\star)-2\norm{\bX-\bX_\star}_{\F}}, \qquad \norm{\bS_\eta^{-1}}_{\op}\le\frac1{\sigma_r(\bX_\star)-2\norm{\bX-\bX_\star}_{\F}-\norm{\bm\eta}_{\F}}.
\]
In particular, the two cores are invertible under the hypotheses. Since $\rank(\bX_\star)=r$, the Schur complement identity gives $\Pcal_{T_{\bX}^\perp}(\bX_\star-\bX)=\bL_{\bX}\bG_{\bX}^{-1}\bR_{\bX}^T$. Applying \eqref{eq:graph-retraction-identity} to $\bm t_{\bX}+\bm\eta$ and using $\bS_\eta^{-1}-\bG_{\bX}^{-1}=-\bG_{\bX}^{-1}\bM_\eta\bS_\eta^{-1}$, we have
\begin{equation}\label{eq:normal-F-retraction-decomposition}
\begin{aligned}
&~\operatorname{Retr}_{\bX}(\bm t_{\bX}+\bm\eta)-\bX_\star-\bm\eta\cr =&~\bB_\eta\bS_\eta^{-1}\bR_{\bX}^T+\bL_{\bX}\bS_\eta^{-1}\bC_\eta^T+\bB_\eta\bS_\eta^{-1}\bC_\eta^T -\bL_{\bX}\bG_{\bX}^{-1}\bM_\eta\bS_\eta^{-1}\bR_{\bX}^T.
\end{aligned}
\end{equation}
Moreover, $\max\{\norm{\bL_{\bX}}_{\F},\norm{\bR_{\bX}}_{\F}\}\le\norm{\bX-\bX_\star}_{\F}$ and $\max\{\norm{\bM_\eta}_{\F},\norm{\bB_\eta}_{\F},\norm{\bC_\eta}_{\F}\}\le\norm{\bm\eta}_{\F}$. Combining these bounds with \eqref{eq:normal-F-retraction-decomposition} gives
\begin{align*}
\norm{\operatorname{Retr}_{\bX}(\bm t_{\bX}+\bm\eta)-\bX_\star-\bm\eta}_{\F} &\le\frac{2\norm{\bX-\bX_\star}_{\F}\norm{\bm\eta}_{\F}+\norm{\bm\eta}_{\F}^2}{\sigma_r(\bX_\star)-2\norm{\bX-\bX_\star}_{\F}-\norm{\bm\eta}_{\F}}\\
&\quad+\frac{\norm{\bX-\bX_\star}_{\F}^2\norm{\bm\eta}_{\F}}{\bigl(\sigma_r(\bX_\star)-2\norm{\bX-\bX_\star}_{\F}\bigr)\bigl(\sigma_r(\bX_\star)-2\norm{\bX-\bX_\star}_{\F}-\norm{\bm\eta}_{\F}\bigr)}\\
&\le\frac{11}{5}\frac{\norm{\bX-\bX_\star}_{\F}\norm{\bm\eta}_{\F}}{\sigma_r(\bX_\star)}+\frac{11}{10}\frac{\norm{\bm\eta}_{\F}^2}{\sigma_r(\bX_\star)}.
\end{align*}
This proves part~\textup{(ii)}.
\end{proof}

\subsection{Comparison of two nearby matrices}

We first compare aligned singular factors and graph cores at two nearby matrices.

\subsubsection{Variation of the aligned factors}

The next lemma compares the singular factors and graph cores of two nearby matrices. We use the projector and Procrustes estimates in \cite[Lemmas~4.1 and~4.5]{WeiCaiChanLeung2020}, aligning the left and right bases separately.
\begin{lemma}\label[lemma]{lem:factor-motion}
Let $\bX,\bY\in\Mcal_r$. For $\bZ\in\{\bX,\bY\}$, represent
\[
\bZ=\bU_{\bZ}\bS_{\bZ}\bV_{\bZ}^T
\]
in independently aligned Procrustes gauges so that $\bU_{\bX}^T\bU_{\bY}\succeq0$ and $\bV_{\bX}^T\bV_{\bY}\succeq0$, and use the graph factors from \eqref{eq:target-graph-factors} for each $\bZ\in\{\bX,\bY\}$. Suppose that
\[
\max\left\{ \sharpnorm{\bX-\bX_\star}, \sharpnorm{\bY-\bX_\star} \right\} \le \rho\le\frac{\sigma_r(\bX_\star)}{1000\mu r}.
\]
Then
\[
\max\left\{ \norm{\bU_{\bX}-\bU_{\bY}}_{\F}, \norm{\bV_{\bX}-\bV_{\bY}}_{\F}, \norm{\bU_{\bX}\bU_{\bX}^T-\bU_{\bY}\bU_{\bY}^T}_{\F}, \norm{\bV_{\bX}\bV_{\bX}^T-\bV_{\bY}\bV_{\bY}^T}_{\F} \right\} \le\frac32\frac{d_{\bX,\bY}}{\sigma_r(\bX_\star)}.
\]
Moreover,
\[
\max\left\{ \norm{\bL_{\bX}-\bL_{\bY}}_{\F}, \norm{\bR_{\bX}-\bR_{\bY}}_{\F} \right\} \le\left(1+3\frac{\rho}{\sigma_r(\bX_\star)}\right)d_{\bX,\bY},
\]
and
\[
\norm{\bG_{\bX}-\bG_{\bY}}_{\F}\le5d_{\bX,\bY}, \qquad \max_{\bZ\in\{\bX,\bY\}}\norm{\bG_{\bZ}^{-1}}_{\op} \le\frac1{\sigma_r(\bX_\star)-2\rho}.
\]
Finally,
\[
\norm{\bG_{\bX}^{-1}-\bG_{\bY}^{-1}}_{\F} \le\frac{5d_{\bX,\bY}}{\bigl(\sigma_r(\bX_\star)-2\rho\bigr)^2}.
\]
\end{lemma}

\begin{proof}
By Weyl's inequality, $\min\{\sigma_r(\bX),\sigma_r(\bY)\}>0$. Since $(\bI-\bU_{\bX}\bU_{\bX}^T)\bX=\bm 0$, $\bY\bV_{\bY}\bS_{\bY}^{-1}=\bU_{\bY}$, and the transposed identities hold on the right, the Procrustes estimates in \cite[Lemmas~4.1 and~4.5]{WeiCaiChanLeung2020} give
\begin{align*}
\norm{\bU_{\bX}-\bU_{\bY}}_{\F} &\le\frac32\frac{d_{\bX,\bY}}{\sigma_r(\bX_\star)}, & \norm{\bV_{\bX}-\bV_{\bY}}_{\F} &\le\frac32\frac{d_{\bX,\bY}}{\sigma_r(\bX_\star)},\\
\norm{\bU_{\bX}\bU_{\bX}^T-\bU_{\bY}\bU_{\bY}^T}_{\F} &\le\frac32\frac{d_{\bX,\bY}}{\sigma_r(\bX_\star)}, & \norm{\bV_{\bX}\bV_{\bX}^T-\bV_{\bY}\bV_{\bY}^T}_{\F} &\le\frac32\frac{d_{\bX,\bY}}{\sigma_r(\bX_\star)}.
\end{align*}
This proves the factor and projector estimates.

In the Procrustes gauges of the statement, set
\[
\bC_U:=\bU_{\bX}^T\bU_{\bY}\succeq0, \qquad \bC_V:=\bV_{\bX}^T\bV_{\bY}\succeq0.
\]
Since $\rank(\bX-\bY)\le2r$ and $\max\{\sharpnorm{\bX-\bX_\star},\sharpnorm{\bY-\bX_\star}\}\le\rho$, we have $d_{\bX,\bY}\le2\sqrt{2r}\,\rho$. Combining this with $\rho\le\sigma_r(\bX_\star)/(1000\mu r)$ and the projector estimates obtained from \cite[Lemmas~4.1 and~4.5]{WeiCaiChanLeung2020} shows that $\bC_U$ and $\bC_V$ are invertible. Their definitions give
\[
\bI-\bC_U^2=\bU_{\bY}^T(\bI-\bU_{\bX}\bU_{\bX}^T)\bU_{\bY}, \qquad \bI-\bC_V^2=\bV_{\bY}^T(\bI-\bV_{\bX}\bV_{\bX}^T)\bV_{\bY}.
\]
Consequently, we have
\begin{align*}
(\bI-\bC_U)\bS_{\bY} &= (\bI+\bC_U)^{-1}\bU_{\bY}^T(\bI-\bU_{\bX}\bU_{\bX}^T)(\bY-\bX)\bV_{\bY},\\
\bS_{\bY}(\bI-\bC_V) &=\bU_{\bY}^T(\bY-\bX)(\bI-\bV_{\bX}\bV_{\bX}^T)\bV_{\bY}(\bI+\bC_V)^{-1}.
\end{align*}
Since $\bS_{\bX}=\bU_{\bX}^T\bX\bV_{\bX}$ and $\bS_{\bY}=\bU_{\bY}^T\bY\bV_{\bY}$, substituting these identities gives
\begin{align*}
\bS_{\bX}-\bS_{\bY} &=\bU_{\bX}^T(\bX-\bY)\bV_{\bX}+(\bC_U-\bI)\bS_{\bY}\bC_V+\bS_{\bY}(\bC_V-\bI),\\
\norm{\bS_{\bX}-\bS_{\bY}}_{\F} &\le3d_{\bX,\bY}.
\end{align*}

Using the definitions of the graph factors together with the factor and projector bounds, we have
\begin{align*}
\norm{\bL_{\bX}-\bL_{\bY}}_{\F} &\le d_{\bX,\bY}+3\frac{\rho\,d_{\bX,\bY}}{\sigma_r(\bX_\star)},\\
\norm{\bR_{\bX}-\bR_{\bY}}_{\F} &\le d_{\bX,\bY}+3\frac{\rho\,d_{\bX,\bY}}{\sigma_r(\bX_\star)}.
\end{align*}
Moreover, we have
\begin{align*}
\norm{\bU_{\bX}^T(\bX_\star-\bX)\bV_{\bX}-\bU_{\bY}^T(\bX_\star-\bY)\bV_{\bY}}_{\F}\le d_{\bX,\bY}+3\frac{\rho\,d_{\bX,\bY}}{\sigma_r(\bX_\star)}.
\end{align*}
Combining the preceding estimate with $\norm{\bS_{\bX}-\bS_{\bY}}_{\F}\le3d_{\bX,\bY}$ yields
\[
\norm{\bG_{\bX}-\bG_{\bY}}_{\F}\le5d_{\bX,\bY}.
\]
For $\bZ\in\{\bX,\bY\}$, the definition of $\bG_{\bZ}$ and Weyl's inequality give
\[
\sigma_{\min}(\bG_{\bZ})\ge\sigma_r(\bX_\star)-2\rho, \qquad \norm{\bG_{\bZ}^{-1}}_{\op}\le\frac1{\sigma_r(\bX_\star)-2\rho}.
\]
The resolvent identity then gives
\[
\norm{\bG_{\bX}^{-1}-\bG_{\bY}^{-1}}_{\F} \le\frac{5d_{\bX,\bY}}{\bigl(\sigma_r(\bX_\star)-2\rho\bigr)^2}.
\]
This completes the proof.
\end{proof}

We now apply the factor estimates to the graph-retraction identity. The proof also verifies that both updated cores are invertible, so both retractions are well defined.

\subsubsection{Proof of Lemma~\ref{lem:graph-two-base}}
\label{sec:proof-graph-two-base}

\begin{proof}
We align the singular factors as in \Cref{lem:factor-motion}. For $\bZ\in\{\bX,\bY\}$, write
\[
\bM_{\bZ}:=\bU_{\bZ}^T\bm\eta_{\bZ}\bV_{\bZ},\qquad \bB_{\bZ}:=(\bI-\bU_{\bZ}\bU_{\bZ}^T)\bm\eta_{\bZ}\bV_{\bZ},
\]
and
\[
\bC_{\bZ}:=(\bI-\bV_{\bZ}\bV_{\bZ}^T)\bm\eta_{\bZ}^T\bU_{\bZ},\qquad \bS_{\bm\eta,\bZ}:=\bG_{\bZ}+\bM_{\bZ}.
\]
Set
\[
s_\eta:=\max\{\norm{\bm\eta_{\bX}}_{\F},\norm{\bm\eta_{\bY}}_{\F}\},\qquad \Delta:=d_\eta+3\frac{s_\eta d_{\bX,\bY}}{\sigma_r(\bX_\star)}.
\]
Using the aligned bases, we have
\begin{align*}
\bM_{\bX}-\bM_{\bY} &=(\bU_{\bX}-\bU_{\bY})^T\bm\eta_{\bX}\bV_{\bX}+\bU_{\bY}^T(\bm\eta_{\bX}-\bm\eta_{\bY})\bV_{\bX}+\bU_{\bY}^T\bm\eta_{\bY}(\bV_{\bX}-\bV_{\bY}),\\
\bB_{\bX}-\bB_{\bY} &=(\bU_{\bY}\bU_{\bY}^T-\bU_{\bX}\bU_{\bX}^T)\bm\eta_{\bX}\bV_{\bX}+(\bI-\bU_{\bY}\bU_{\bY}^T)(\bm\eta_{\bX}-\bm\eta_{\bY})\bV_{\bX}\\
&\quad +(\bI-\bU_{\bY}\bU_{\bY}^T)\bm\eta_{\bY}(\bV_{\bX}-\bV_{\bY}),\\
\bC_{\bX}-\bC_{\bY} &=(\bV_{\bY}\bV_{\bY}^T-\bV_{\bX}\bV_{\bX}^T)\bm\eta_{\bX}^T\bU_{\bX}+(\bI-\bV_{\bY}\bV_{\bY}^T)(\bm\eta_{\bX}-\bm\eta_{\bY})^T\bU_{\bX}\\
&\quad +(\bI-\bV_{\bY}\bV_{\bY}^T)\bm\eta_{\bY}^T(\bU_{\bX}-\bU_{\bY}).
\end{align*}
Therefore, \Cref{lem:factor-motion} gives
\begin{equation}\label{eq:two-base-coordinate-differences}
\max\left\{\norm{\bM_{\bX}-\bM_{\bY}}_{\F},\norm{\bB_{\bX}-\bB_{\bY}}_{\F},\norm{\bC_{\bX}-\bC_{\bY}}_{\F}\right\}\le\Delta.
\end{equation}
Moreover, the definitions of the tangent and graph coordinates imply
\begin{equation}\label{eq:two-base-coordinate-sizes}
\max_{\bZ\in\{\bX,\bY\}}\max\left\{\norm{\bM_{\bZ}}_{\F},\norm{\bB_{\bZ}}_{\F},\norm{\bC_{\bZ}}_{\F}\right\}\le s_\eta, \qquad \max_{\bZ\in\{\bX,\bY\}}\max\left\{\norm{\bL_{\bZ}}_{\op},\norm{\bR_{\bZ}}_{\op}\right\}\le \rho.
\end{equation}

The correction-size hypothesis gives $s_\eta\le9\mu r \rho^2/\sigma_r(\bX_\star)$. Furthermore, Weyl's inequality and \eqref{eq:two-base-coordinate-sizes} give
\[
\sigma_{\min}(\bS_{\bm\eta,\bZ}) \ge\sigma_r(\bX_\star)-2\rho-s_\eta>0, \qquad \norm{\bS_{\bm\eta,\bZ}^{-1}}_{\op} \le\frac1{\sigma_r(\bX_\star)-2\rho-s_\eta}, \qquad \bZ\in\{\bX,\bY\}.
\]
Using the resolvent identity, \Cref{lem:factor-motion}, and \eqref{eq:two-base-coordinate-differences}, we also obtain
\begin{equation}\label{eq:two-base-updated-core-motion}
\norm{\bS_{\bm\eta,\bX}^{-1}-\bS_{\bm\eta,\bY}^{-1}}_{\F} \le\frac{5d_{\bX,\bY}+\Delta}{\bigl(\sigma_r(\bX_\star)-2\rho-s_\eta\bigr)^2}.
\end{equation}

Since $\rho<\sigma_r(\bX_\star)/4$, \Cref{lem:population-exact} applies at both base points. Hence \eqref{eq:graph-retraction-identity} gives, for $\bZ\in\{\bX,\bY\}$, we have
\begin{equation}\label{eq:two-base-retraction-remainder}
\begin{aligned}
&~\operatorname{Retr}_{\bZ}(\bm t_{\bZ}+\bm\eta_{\bZ})-\bX_\star-\bm\eta_{\bZ}\cr =&~\bB_{\bZ}\bS_{\bm\eta,\bZ}^{-1}\bR_{\bZ}^T+\bL_{\bZ}\bS_{\bm\eta,\bZ}^{-1}\bC_{\bZ}^T+\bB_{\bZ}\bS_{\bm\eta,\bZ}^{-1}\bC_{\bZ}^T -\bL_{\bZ}\bG_{\bZ}^{-1}\bM_{\bZ}\bS_{\bm\eta,\bZ}^{-1}\bR_{\bZ}^T.
\end{aligned}
\end{equation}

We compare the four terms in \eqref{eq:two-base-retraction-remainder}. In each product, we use the Frobenius norm for a correction or a difference and the operator norm for the remaining graph factors. For the first term, we have
\begin{align*}
&~\bB_{\bX}\bS_{\bm\eta,\bX}^{-1}\bR_{\bX}^T-\bB_{\bY}\bS_{\bm\eta,\bY}^{-1}\bR_{\bY}^T\cr =&~(\bB_{\bX}-\bB_{\bY})\bS_{\bm\eta,\bX}^{-1}\bR_{\bX}^T+\bB_{\bY}(\bS_{\bm\eta,\bX}^{-1}-\bS_{\bm\eta,\bY}^{-1})\bR_{\bX}^T+\bB_{\bY}\bS_{\bm\eta,\bY}^{-1}(\bR_{\bX}-\bR_{\bY})^T.
\end{align*}
Interchanging the left and right factors gives the corresponding expansion for $\bL_{\bZ}\bS_{\bm\eta,\bZ}^{-1}\bC_{\bZ}^T$. For the third term, we have
\begin{align*}
&~\bB_{\bX}\bS_{\bm\eta,\bX}^{-1}\bC_{\bX}^T-\bB_{\bY}\bS_{\bm\eta,\bY}^{-1}\bC_{\bY}^T \cr =&~(\bB_{\bX}-\bB_{\bY})\bS_{\bm\eta,\bX}^{-1}\bC_{\bX}^T+\bB_{\bY}(\bS_{\bm\eta,\bX}^{-1}-\bS_{\bm\eta,\bY}^{-1})\bC_{\bX}^T+\bB_{\bY}\bS_{\bm\eta,\bY}^{-1}(\bC_{\bX}-\bC_{\bY})^T.
\end{align*}
Finally, the five-factor difference has the telescoping expansion
\begin{align*}
&~\bL_{\bX}\bG_{\bX}^{-1}\bM_{\bX}\bS_{\bm\eta,\bX}^{-1}\bR_{\bX}^T-\bL_{\bY}\bG_{\bY}^{-1}\bM_{\bY}\bS_{\bm\eta,\bY}^{-1}\bR_{\bY}^T\cr =&~(\bL_{\bX}-\bL_{\bY})\bG_{\bX}^{-1}\bM_{\bX}\bS_{\bm\eta,\bX}^{-1}\bR_{\bX}^T+\bL_{\bY}(\bG_{\bX}^{-1}-\bG_{\bY}^{-1})\bM_{\bX}\bS_{\bm\eta,\bX}^{-1}\bR_{\bX}^T\\
&\quad+\bL_{\bY}\bG_{\bY}^{-1}(\bM_{\bX}-\bM_{\bY})\bS_{\bm\eta,\bX}^{-1}\bR_{\bX}^T+\bL_{\bY}\bG_{\bY}^{-1}\bM_{\bY}(\bS_{\bm\eta,\bX}^{-1}-\bS_{\bm\eta,\bY}^{-1})\bR_{\bX}^T\\
&\quad+\bL_{\bY}\bG_{\bY}^{-1}\bM_{\bY}\bS_{\bm\eta,\bY}^{-1}(\bR_{\bX}-\bR_{\bY})^T.
\end{align*}
Applying \Cref{lem:factor-motion} and \eqref{eq:two-base-coordinate-differences}--\eqref{eq:two-base-updated-core-motion} termwise in these expansions, and then using $\Delta=d_\eta+3s_\eta d_{\bX,\bY}/\sigma_r(\bX_\star)$, $s_\eta\le9\mu r \rho^2/\sigma_r(\bX_\star)$, and $\rho\le\sigma_r(\bX_\star)/(1000\mu r)$, yields
\[
\norm{\operatorname{Retr}_{\bX}(\bm t_{\bX}+\bm\eta_{\bX})-\operatorname{Retr}_{\bY}(\bm t_{\bY}+\bm\eta_{\bY})-(\bm\eta_{\bX}-\bm\eta_{\bY})}_{\F}\le\frac14d_\eta+19\frac{\mu r \rho^2}{\sigma_r(\bX_\star)^2}d_{\bX,\bY}.
\]
Adding $\norm{\bm\eta_{\bX}-\bm\eta_{\bY}}_{\F}=d_\eta$ proves \eqref{eq:finite-two-base-main}.
\end{proof}

\section{Sharp-norm estimates for spectral initialization}\label{app:spectral}

We prove the deterministic reconstruction estimate and the finite approximation used in \Cref{sec:initialization}. The estimates concern the reconstructed matrix and do not require separation between adjacent singular values.

\subsection[Proof of the sharp reconstruction estimate]{Proof of \Cref{lem:spectral-reconstruction}}
\label{sec:proof-spectral-reconstruction}

\begin{proof}[Proof of \Cref{lem:spectral-reconstruction}]
Suppose first that $\bZ^+\ne\bm0$. With the dilation in \eqref{eq:spectral-dilation}, set
\[
\widehat{\mathscr Q} :=\frac1{\sqrt2}
\begin{bmatrix}
\widehat{\bU}&\widehat{\bU}\\
\widehat{\bV}&-\widehat{\bV}
\end{bmatrix}
, \qquad \mathscr D:=\operatorname{diag}(\widehat{\bm\Sigma},-\widehat{\bm\Sigma}), \qquad \bm P:=\mathscr Q_\star\mathscr Q_\star^T.
\]
The columns of $\widehat{\mathscr Q}$ are orthonormal, and $\norm{\mathscr D^{-1}}_{\op}\le8/\tau$. The one-sided singular equation in \eqref{eq:spectral-one-sided-pair} and the residual bound in \eqref{eq:spectral-pair-residual} give
\begin{equation}\label{eq:spectral-dilated-equation}
\left(
\begin{bmatrix}
\bm0&\bX_\star\\
\bX_\star^T&\bm0
\end{bmatrix}
+\mathscr W\right)\widehat{\mathscr Q} =\widehat{\mathscr Q}\mathscr D+\mathscr F, \qquad \norm{\mathscr F}_{\op}\le\frac{\tau}{64}\sqrt{\frac{\mu r}{n}}.
\end{equation}
The coefficient of the true singular spaces is
\[
\bC:=
\begin{bmatrix}
\bm0&\bm\Sigma_\star\\
\bm\Sigma_\star&\bm0
\end{bmatrix}
\mathscr Q_\star^T\widehat{\mathscr Q}\mathscr D^{-1}.
\]
By \eqref{eq:spectral-dilated-equation}, we have
\[
\mathscr Q_\star\bC =\widehat{\mathscr Q}-\mathscr W\widehat{\mathscr Q}\mathscr D^{-1} +\mathscr F\mathscr D^{-1}, \qquad \norm{\bC}_{\op} \le1+\frac8\tau\left(\frac{\tau}{512}+\frac{\tau}{64}\sqrt{\frac{\mu r}{n}}\right) \le\frac54.
\]
In particular, this estimate does not use the ratio $\sigma_1(\bX_\star)/\tau$.

Since $\norm{\mathscr W}_{\op}\norm{\mathscr D^{-1}}_{\op}\le1/64$, the same equation has the convergent expansion
\begin{equation}\label{eq:spectral-reconstruction-series}
\widehat{\mathscr Q} =\sum_{j\ge0}\mathscr W^j\mathscr Q_\star\bC\mathscr D^{-j} -\sum_{j\ge0}\mathscr W^j\mathscr F\mathscr D^{-j-1}.
\end{equation}
For $j\ge1$, \eqref{eq:spectral-power-bounds} implies
\[
\norm{(\bI-\bm P)\mathscr W^j\mathscr Q_\star}_{2,\infty} \le2\sqrt{\frac{\mu r}{n}}\left(\frac{\tau}{64}\right)^j +\sqrt{\frac{\mu r}{n}}\norm{\mathscr W}_{\op}^j \le3\sqrt{\frac{\mu r}{n}}\left(\frac{\tau}{64}\right)^j.
\]
Also, $\norm{(\bI-\bm P)\mathscr W^j\mathscr F}_{2,\infty}\le\norm{\mathscr W}_{\op}^j\norm{\mathscr F}_{\op}$. Applying these estimates to \eqref{eq:spectral-reconstruction-series} gives
\begin{align}
\norm{(\bI-\bm P)\widehat{\mathscr Q}\mathscr D}_{2,\infty} &\le\sqrt{\frac{\mu r}{n}}\left( \frac{15\tau}{256(1-1/8)}+\frac{\tau}{64(1-1/64)}\right) \le\frac{3\tau}{32}\sqrt{\frac{\mu r}{n}},\label{eq:spectral-offspace-weighted}\\
\norm{(\bI-\bm P)\widehat{\mathscr Q}}_{2,\infty} &\le\frac34\sqrt{\frac{\mu r}{n}},\qquad \norm{(\bI-\bm P)\widehat{\mathscr Q}\mathscr D\widehat{\mathscr Q}^T(\bI-\bm P)}_{\infty} \le\frac{9\tau\mu r}{128n}.\label{eq:spectral-offspace-entry}
\end{align}

The matrix $\widehat{\mathscr Q}\mathscr D\widehat{\mathscr Q}^T$ is the symmetric dilation of $\bZ^+$. By \eqref{eq:spectral-pair-residual}, we have
\begin{equation}\label{eq:spectral-reconstruction-op}
\norm{\bZ^+-\bX_\star}_{\op} \le\frac{7\tau}{32}+\frac{\tau}{512} =\frac{113\tau}{512}.
\end{equation}
Decompose its error on the left by $\bm P$ and $\bI-\bm P$. Since $\norm{\mathscr Q_\star}_{2,\infty}\le\sqrt{\mu r/n}$, \eqref{eq:spectral-offspace-weighted}--\eqref{eq:spectral-reconstruction-op} yield
\[
\max\{\norm{\bZ^+-\bX_\star}_{2,\infty},\norm{(\bZ^+-\bX_\star)^T}_{2,\infty}\} \le\left(\frac{113}{512}+\frac3{32}\right)\tau\sqrt{\frac{\mu r}{n}} =\frac{161\tau}{512}\sqrt{\frac{\mu r}{n}}.
\]
For the entrywise norm, decompose the error on both sides by the same projections. The true-space term is bounded by $(\mu r/n)\norm{\bZ^+-\bX_\star}_{\op}$, the two mixed terms by $3\tau\mu r/(32n)$ each, and the remaining term by \eqref{eq:spectral-offspace-entry}. Therefore,
\[
\norm{\bZ^+-\bX_\star}_{\infty} \le\left(\frac{113}{512}+\frac6{32}+\frac9{128}\right)\frac{\tau\mu r}{n} =\frac{245\tau\mu r}{512n}.
\]
Combining these three bounds in \eqref{eq:sharp-main} gives $\sharpnorm{\bZ^+-\bX_\star}\le\tau/4$.

If $\bZ^+=\bm0$, then $\sigma_1(\bX_\star)\le\norm{\bY}_{\op}+\norm{\bW}_{\op}\le113\tau/512$. Incoherence gives $\sharpnorm{\bX_\star}=\sigma_1(\bX_\star)$, so the same conclusion follows.
\end{proof}

\subsection[Proof of the approximate reconstruction estimate]{Proof of \Cref{lem:spectral-approximation}}
\label{sec:proof-spectral-approximation}

\begin{proof}[Proof of \Cref{lem:spectral-approximation}]
We condition on the fixed matrix $\bY$. Let $d_1\ge\cdots\ge d_n>0$ be the eigenvalues of $\bY\bY^T+\tau^2\bI/4096$, with corresponding orthonormal eigenvectors $\bm u_1,\ldots,\bm u_n$. The hypothesis on $\sigma_{r+1}(\bY)$ gives
\begin{equation}\label{eq:threshold-tail-eigenvalue}
d_{r+1} \le \frac{\tau^2}{512^2}+\frac{\tau^2}{4096} <\frac{\tau^2}{4000}.
\end{equation}
Let $\bU_1$ contain the eigenvectors whose eigenvalues are at least $\tau^2/128$, and let $\bm D_1$ be the diagonal matrix of those eigenvalues. By \eqref{eq:threshold-tail-eigenvalue}, their number is at most $r$. Let $\bU_2$ and $\bm D_2$ contain the remaining eigenvectors and eigenvalues, so that
\[
\bY\bY^T+\frac{\tau^2}{4096}\bI =\bU_1\bm D_1\bU_1^T+\bU_2\bm D_2\bU_2^T, \qquad \norm{\bm D_2}_{\op}\le\frac{\tau^2}{128}.
\]
These decompositions are used only in the proof, not in the computation of $\Tcal$.

\emph{Step 1: the initial subspace.} Put $\bU_{[r]}=[\bm u_1,\ldots,\bm u_r]$. In the eigenbasis $[\bU_{[r]},\bU_{[r]}^\perp]$, write the Gaussian initial matrix as $[\bG_1^T,\bG_2^T]^T$, where $\bG_1\in\R^{r\times r}$. Orthogonal invariance preserves the standard Gaussian distribution. Since $\Pr\{|g|\le t\}\le t$ for $g\sim N(0,1)$,
\[
\Pr\!\left\{\min_{1\le j\le r}\operatorname{dist} \bigl((\bG_1)_{:,j},\operatorname{span}\{(\bG_1)_{:,k}:k\ne j\}\bigr)<n^{-20}\right\} \le n^{-19},
\]
and therefore
\[
\norm{\bG_1^{-1}}_{\op}\le\sqrt r\,n^{20}.
\]
Moreover,
\[
\E\norm{\bG_2}_{\F}^2\le n^2, \qquad \Pr\{\norm{\bG_2}_{\F}>n^{11}\}\le n^{-20}.
\]
Hence, with probability at least $1-2n^{-19}\ge1-n^{-12}/48$,
\begin{equation}\label{eq:threshold-gaussian-bound}
\norm{\bG_2\bG_1^{-1}}_{\op}\le n^{32}.
\end{equation}

\emph{Step 2: weighted subspace estimates.} Fix a realization satisfying \eqref{eq:threshold-gaussian-bound}. Thin QR does not change the column space, so \eqref{eq:threshold-iteration} gives
\[
\range(\bQ_t)=\range\!\left(\left(\bY\bY^T+\frac{\tau^2}{4096}\bI\right)^t\bG\right).
\]
Relative to $\bU_{[r]}\oplus\bU_{[r]}^\perp$, this space has graph matrix
\[
\operatorname{diag}(d_{r+1}^t,\ldots,d_n^t)\bG_2\bG_1^{-1} \operatorname{diag}(d_1^{-t},\ldots,d_r^{-t}).
\]
Since $\range(\bU_1)\subseteq\range(\bU_{[r]})$, \eqref{eq:threshold-tail-eigenvalue} and $d_j\ge\tau^2/128$ on $\range(\bU_1)$ imply, for $t\ge1$ and $s\in\{0,1/2,1\}$, we have
\begin{equation}\label{eq:threshold-weighted-subspace}
\norm{(\bI-\bQ_t\bQ_t^T)\bU_1\bm D_1^s}_{\op} \le\left(\frac{\tau^2}{128}\right)^s n^{32}30^{-t}.
\end{equation}
For $t=\ceil{12\log n}$ and $n\ge2$, we have
\[
n^{32}30^{-t}\le n^{-7}\le\frac1{64}\sqrt{\frac{\mu r}{n}}.
\]
Let $\bQ$ denote the final factor, and put $\bm P=\bQ\bQ^T$ and $\bE=(\bI-\bm P)\bU_1$. Then
\begin{equation}\label{eq:threshold-final-subspace}
\norm{\bE\bm D_1^s}_{\op} \le\frac1{64}\sqrt{\frac{\mu r}{n}} \left(\frac{\tau^2}{128}\right)^s, \qquad s\in\{0,1/2,1\}.
\end{equation}
If $\bU_1$ has no columns, then $\norm{\bY}_{\op}<\tau/(8\sqrt2)<\tau/8$. No singular value is retained and the conclusion is immediate. We henceforth suppose that $\bU_1$ is nonempty.

\emph{Step 3: the retained singular factors.} An orthonormal basis for $\range(\bm P\bU_1)$ is
\[
\widetilde{\bQ}_1=(\bU_1-\bE)(\bI-\bE^T\bE)^{-1/2}, \qquad \norm{(\bI-\bE^T\bE)^{-1/2}}_{\op}\le\frac43.
\]
Complete it inside $\range(\bQ)$ by $\widetilde{\bQ}_2$. Then $\widetilde{\bQ}_2^T\bU_1=\bm0$. Since $\bU_1^T\bE=\bE^T\bE$, direct substitution gives
\begin{align*}
(\bI-\bm P)\bY\bY^T\widetilde{\bQ}_1 &=\left[\bE\bm D_1-\bE\bm D_1\bE^T\bE -(\bI-\bm P)\bU_2\bm D_2\bU_2^T\bE\right](\bI-\bE^T\bE)^{-1/2},\\
\widetilde{\bQ}_2^T\bY\bY^T\widetilde{\bQ}_1 &=-\widetilde{\bQ}_2^T\bU_2\bm D_2\bU_2^T\bE(\bI-\bE^T\bE)^{-1/2}.
\end{align*}
Using \eqref{eq:threshold-final-subspace} and $\mu r/n\le1$, these identities imply
\begin{subequations}\label{eq:threshold-ritz-blocks}
\begin{align}
\norm{(\bI-\bm P)\bY\bY^T\widetilde{\bQ}_1}_{\op} &\le\frac{\tau^2}{2048}\sqrt{\frac{\mu r}{n}},\label{eq:threshold-ritz-high}\\
\norm{\widetilde{\bQ}_2^T\bY\bY^T\widetilde{\bQ}_1}_{\op} &\le\frac{\tau^2}{4096}\sqrt{\frac{\mu r}{n}}, \qquad \widetilde{\bQ}_2^T\bY\bY^T\widetilde{\bQ}_2\preceq\frac{\tau^2}{128}\bI.\label{eq:threshold-ritz-low}
\end{align}
\end{subequations}

The compressed SVD in \eqref{eq:threshold-definition} gives $\bY^T\widehat{\bU}=\widehat{\bV}\widehat{\bm\Sigma}$ and $\sigma_{\min}(\widehat{\bm\Sigma})\ge\tau/8$. Its retained left factors are Ritz vectors of $\bY\bY^T$ in $\range(\bQ)$. Write $\widehat{\bU}=\widetilde{\bQ}_1\bZ_1+\widetilde{\bQ}_2\bZ_2$. The lower block of the Ritz equation is
\[
\bZ_2\widehat{\bm\Sigma}^{\,2} -(\widetilde{\bQ}_2^T\bY\bY^T\widetilde{\bQ}_2)\bZ_2 =\widetilde{\bQ}_2^T\bY\bY^T\widetilde{\bQ}_1\bZ_1.
\]
The spectra on the two sides are separated by at least $\tau^2/128$. The integral solution of this Sylvester equation, together with \eqref{eq:threshold-ritz-low} and $\norm{\bZ_1}_{\op}\le1$, therefore gives
\[
\norm{\bZ_2}_{\op} \le\frac{128}{\tau^2}\norm{\widetilde{\bQ}_2^T\bY\bY^T\widetilde{\bQ}_1\bZ_1}_{\op} \le\frac1{32}\sqrt{\frac{\mu r}{n}}.
\]
Since $\widetilde{\bQ}_2$ is orthogonal to $\bU_1$, also $\norm{\bY\bY^T\widetilde{\bQ}_2}_{\op}\le\tau^2/128$. Combining this with \eqref{eq:threshold-ritz-high} yields
\[
\norm{(\bI-\bm P)\bY\bY^T\widehat{\bU}}_{\op} \le\frac{\tau^2}{1024}\sqrt{\frac{\mu r}{n}}.
\]
Moreover, $\bm P\bY\widehat{\bV}=\widehat{\bU}\widehat{\bm\Sigma}$. Consequently,
\[
\norm{\bY\widehat{\bV}-\widehat{\bU}\widehat{\bm\Sigma}}_{\op} =\norm{(\bI-\bm P)\bY\bY^T\widehat{\bU}\widehat{\bm\Sigma}^{-1}}_{\op} \le\frac{\tau}{128}\sqrt{\frac{\mu r}{n}} \le\frac{\tau}{64}\sqrt{\frac{\mu r}{n}}.
\]
If no singular value is retained, this residual estimate is unnecessary.

\emph{Step 4: the unreconstructed part.} The $s=1/2$ estimate in \eqref{eq:threshold-final-subspace} controls the part of $\bY$ in $\range(\bU_1)$, since $\sigma_j(\bY)\le\sqrt{d_j}$. On its orthogonal complement the singular values are smaller than $\tau/(8\sqrt2)$. Thus
\[
\norm{(\bI-\bm P)\bY}_{\op} \le\frac{\tau}{512}\sqrt{\frac{\mu r}{n}} +\frac{\tau}{8\sqrt2} <\frac{3\tau}{32}.
\]
The singular values discarded from $\bm P\bY$ are smaller than $\tau/8$. Therefore,
\[
\norm{\bY-\bZ^+}_{\op} \le\norm{(\bI-\bm P)\bY}_{\op} +\norm{\bm P\bY-\bQ\Hcal_{\ge\tau/8}(\bQ^T\bY)}_{\op} \le\frac{3\tau}{32}+\frac{\tau}{8} =\frac{7\tau}{32}.
\]
This includes the case $\bZ^+=\bm0$. Thus all approximation conclusions hold on \eqref{eq:threshold-gaussian-bound}. Since $\bY\bY^T+\tau^2\bI/4096\succ\bm0$ and $\bG$ has full column rank almost surely, every thin QR step is well defined. A fixed orthonormal completion defines the procedure on the remaining null event without changing its operation count.
\end{proof}

\section{Leave-one-out estimates for the initial RGN iterates}\label{app:loo}

We prove the sampling and correction estimates used to control the RGN iterates for $0\le k\le\bar K$. The concentration and geometric tools are given in \Cref{app:tools,app:geometry}. The passage to the quadratic Frobenius recurrence is proved in \Cref{sec:proof-local-rgn}.

\subsection{Sampling estimates for the completed operators}

We use the completed sampling operators $\Rcal^\alpha$ and the index set $\mathfrak A$ defined in \Cref{sec:proof-fixed-rgn}.  For $\alpha\in\mathfrak A$, let $\Acal^\alpha$ be the entrywise nonnegative square root of $\Rcal^\alpha$, so $(\Acal^\alpha)^*\Acal^\alpha=\Rcal^\alpha$.

We first establish the simultaneous sampling isometries for the actual and completed operators.

\subsubsection{Proof of Lemma~\ref{lem:all-true-rip}}
\label{sec:proof-all-true-rip}

\begin{proof}
Let $\delta_{ij}:=\mathbf 1_{\{(i,j)\in\widehat\Omega\}}$ and $\bm a_{ij}=\Pcal_{T_{\bX_\star}}(\bm e_i\bm e_j^T)$, where $(\bm a\otimes\bm a)(\bZ):=\inner{\bm a}{\bZ}_{\F}\bm a$. By \Cref{ass:incoherence}, $\norm{\bm a_{ij}}_{\F}^2\le2\mu r/n$.  Since the matrices $\bm e_i\bm e_j^T$ form an orthonormal basis and $\Pcal_{T_{\bX_\star}}$ is an orthogonal projector, we have
\begin{align*}
\sum_{i,j}\bm a_{ij}\otimes\bm a_{ij} &=\Ical_{T_{\bX_\star}},\\
(\bm a_{ij}\otimes\bm a_{ij})^2 &=\norm{\bm a_{ij}}_{\F}^2(\bm a_{ij}\otimes\bm a_{ij}),\\
\frac1q\sum_{i,j}\norm{\bm a_{ij}}_{\F}^2 (\bm a_{ij}\otimes\bm a_{ij}) &\preceq\frac{2\mu r}{nq}\Ical_{T_{\bX_\star}}.
\end{align*}
For the actual operator, $\Pcal_{T_{\bX_\star}}\Rcal^0\Pcal_{T_{\bX_\star}} -\Pcal_{T_{\bX_\star}}$ is the centered sum of $(\delta_{ij}/q-1)(\bm a_{ij}\otimes\bm a_{ij})$ over all coordinates. For $\Rcal^{\mathrm r,i}$, the centered sum is restricted to coordinates $(a,b)$ with $a\ne i$; for $\Rcal^{\mathrm c,j}$, it is restricted to coordinates with $b\ne j$.

For the actual centered sum, each summand and the variance operator satisfy
\begin{align*}
\norm{(\delta_{ij}/q-1)(\bm a_{ij}\otimes\bm a_{ij})}_{\F\to\F} &\le\frac{2\mu r}{nq},\\
\sum_{i,j}\E\left[(\delta_{ij}/q-1)^2 (\bm a_{ij}\otimes\bm a_{ij})^2\right] &=\frac{1-q}{q}\sum_{i,j}\norm{\bm a_{ij}}_{\F}^2 (\bm a_{ij}\otimes\bm a_{ij}) \preceq\frac{2\mu r}{nq}\Ical_{T_{\bX_\star}}.
\end{align*}
Restricting the sums to a completed row or column preserves these two inequalities. Since $q\ge c_1\mu r\log n/n$ and $c_1=2^{15}$, the Bernstein exponent at threshold $1/16$ is at least $24\log n$. Since $\dim(T_{\bX_\star})\le2nr$ and $|\mathfrak A|=2n+1$, \Cref{lem:bernstein} followed by a union bound gives
\[
\Pr\left\{ \max_{\alpha\in\mathfrak A} \norm{\Pcal_{T_{\bX_\star}}\Rcal^\alpha\Pcal_{T_{\bX_\star}} -\Pcal_{T_{\bX_\star}}}_{\F\to\F}>\frac1{16} \right\} \le4nr(2n+1)n^{-24} \le\frac{n^{-10}}{24}.
\]
\end{proof}

We next record the weighted product estimate used to transfer these sampling bounds to nearby tangent spaces.

\subsubsection{A weighted product estimate}

\begin{lemma}\label[lemma]{lem:weighted-products}
Suppose that the conclusion of \Cref{lem:degree} holds with $(\Lambda,p_\Lambda)=(\widehat{\Omega},q)$. Then, for every $\alpha\in\mathfrak A$ and all matrices $\bm F,\bm H$ with $n$ rows and the same number of columns,
\[
\norm{\Acal^\alpha(\bm F\bm H^T)}_{\F} \le\sqrt{3n}\min\left\{ \norm{\bm F}_{\F}\norm{\bm H}_{2,\infty}, \norm{\bm F}_{2,\infty}\norm{\bm H}_{\F} \right\}.
\]
\end{lemma}

\begin{proof}
For the actual operator, weighted row and column sums are the degrees divided by $q$, hence at most $2n$.  Completing one row or column adds at most one to every opposite weighted degree and replaces the completed weighted degree by $n$, so all weighted sums are at most $3n$.

If $\bm f_i^T$ and $\bm h_j^T$ are the rows of $\bm F$ and $\bm H$, and $w_{ij}^\alpha$ are the weights of $\Rcal^\alpha$, then
\begin{align*}
\norm{\Acal^\alpha(\bm F\bm H^T)}_{\F}^2 &=\sum_{i,j}w_{ij}^\alpha(\bm f_i^T\bm h_j)^2\\
&\le\sum_i\norm{\bm f_i}_2^2 \sum_jw_{ij}^\alpha\norm{\bm h_j}_2^2\\
&\le3n\norm{\bm F}_{\F}^2\norm{\bm H}_{2,\infty}^2.
\end{align*}
Interchanging rows and columns gives
\[
\norm{\Acal^\alpha(\bm F\bm H^T)}_{\F}^2 \le3n\norm{\bm F}_{2,\infty}^2\norm{\bm H}_{\F}^2.
\]
Taking square roots in the two preceding inequalities and minimizing the resulting upper bounds proves the lemma.
\end{proof}

\subsection{Proof of Lemma~\ref{lem:loo-probability}}
\label{sec:proof-loo-probability}
We first prove the conditional row estimate; the corresponding column estimate follows by transposition in the subsequent leave-one-out argument. The Proof of Lemma~\ref{lem:loo-probability} is after the following lemma.


\begin{lemma}\label[lemma]{lem:conditional-row}
Let $\delta_1,\ldots,\delta_n$ be independent Bernoulli$(p_\Lambda)$ variables, and let $\bm z\in\R^n$ and $\bV\in\R^{n\times r}$ be deterministic. Set
\[
\bm w=(w_j)_{j=1}^n, \qquad w_j=\left(\frac{\delta_j}{p_\Lambda}-1\right)z_j.
\]
If $\bV^T\bV=\bI$ and $\norm{\bV}_{2,\infty}\le2\sqrt{\mu r/n}$, then, with probability at least $1-(n+r+2)n^{-24}$,
\begin{align*}
2\sqrt{\frac{\mu r}{n}}\norm{\bm w}_2+\norm{\bV^T\bm w}_2 \le 8\sqrt{\frac{3\log n}{p_\Lambda}} \left(4\sqrt{\frac{\mu r}{n}}\norm{\bm z}_2+\norm{\bm z}_\infty\right)+\frac{64\sqrt{\mu r/n}\,\log n}{p_\Lambda}\norm{\bm z}_\infty.
\end{align*}
\end{lemma}

\begin{proof}
For $\bm Z_j:=(p_\Lambda^{-1}\delta_j-1)z_j\bm e_j$, we have $\norm{\bm Z_j}_2\le\norm{\bm z}_\infty/p_\Lambda$, $\norm{\sum_j\E(\bm Z_j\bm Z_j^T)}_{\op}\le\norm{\bm z}_\infty^2/p_\Lambda$, and $\sum_j\E\norm{\bm Z_j}_2^2\le\norm{\bm z}_2^2/p_\Lambda$. Thus the variance parameter in \Cref{lem:bernstein} satisfies $v\le\norm{\bm z}_2^2/p_\Lambda$, and applying the lemma with $t=24\log n$ gives
\begin{align*}
L\le\frac{\norm{\bm z}_\infty}{p_\Lambda},\qquad v\le\frac{\norm{\bm z}_2^2}{p_\Lambda},\qquad \norm{\bm w}_2 \le4\sqrt{\frac{3\log n}{p_\Lambda}}\norm{\bm z}_2 +\frac{16\log n}{p_\Lambda}\norm{\bm z}_\infty.
\end{align*}
For $\widetilde{\bm Z}_j:=(p_\Lambda^{-1}\delta_j-1)z_j\bV^T\bm e_j$, the incoherence bound $\norm{\bV}_{2,\infty}\le2\sqrt{\mu r/n}$ gives $\norm{\widetilde{\bm Z}_j}_2\le2\sqrt{\mu r/n}\,\norm{\bm z}_\infty/p_\Lambda$. Moreover, 
$$
\norm{\sum_j\E(\widetilde{\bm Z}_j\widetilde{\bm Z}_j^T)}_{\op}\le\norm{\bm z}_\infty^2/p_\Lambda\qquad\text{and}\qquad\sum_j\E\norm{\widetilde{\bm Z}_j}_2^2\le4\mu r\norm{\bm z}_2^2/(np_\Lambda).
$$
Thus we may take $v\le\norm{\bm z}_\infty^2/p_\Lambda+4\mu r\norm{\bm z}_2^2/(np_\Lambda)$ in \Cref{lem:bernstein}. Applying the lemma with $t=24\log n$ gives
\[
\norm{\bV^T\bm w}_2 \le4\sqrt{\frac{3\log n}{p_\Lambda}} \left(2\sqrt{\frac{\mu r}{n}}\norm{\bm z}_2+\norm{\bm z}_\infty\right) +\frac{32\sqrt{\mu r/n}\,\log n}{p_\Lambda}\norm{\bm z}_\infty.
\]
The Bernstein bounds for $\norm{\bm w}_2$ and $\norm{\bV^T\bm w}_2$ fail with probability at most $(n+1)n^{-24}$ and $(r+1)n^{-24}$, respectively. Therefore, by a union bound, both estimates hold with probability at least $1-(n+r+2)n^{-24}$.  On the intersection of the two Bernstein events, adding the two estimates and collecting the $\norm{\bm z}_2$ and $\norm{\bm z}_\infty$ terms gives
\begin{align*}
2\sqrt{\frac{\mu r}{n}}\norm{\bm w}_2+\norm{\bV^T\bm w}_2 &\le16\sqrt{\frac{3\mu r\log n}{np_\Lambda}}\norm{\bm z}_2 +4\sqrt{\frac{3\log n}{p_\Lambda}}\norm{\bm z}_\infty +\frac{64\sqrt{\mu r/n}\,\log n}{p_\Lambda}\norm{\bm z}_\infty\\
&\le8\sqrt{\frac{3\log n}{p_\Lambda}} \left(4\sqrt{\frac{\mu r}{n}}\norm{\bm z}_2 +\norm{\bm z}_\infty\right) +\frac{64\sqrt{\mu r/n}\,\log n}{p_\Lambda}\norm{\bm z}_\infty,
\end{align*}
which is the claimed inequality.
\end{proof}

We use the stopped auxiliary sequences and the quantities $\rho_k$, $d_k^{\rm loo}$, and $h_k^{\rm corr}$ defined in \Cref{sec:proof-fixed-rgn}. With the fixed compact-SVD convention in \Cref{lem:loo-probability}, write
\[
\bX_k^\alpha=\bU_k^\alpha\bS_k^\alpha(\bV_k^\alpha)^T,
\]
and set $\bN_k^\alpha:=\bN_{\bX_k^\alpha}$ as in \eqref{eq:tangent-normal-components}. For each comparison, the left and right singular bases are aligned as in \Cref{app:geometry}.Now we give the proof of Lemma~\ref{lem:loo-probability}.

\begin{proof}[Proof of Lemma~\ref{lem:loo-probability}]
Fix a row $i$ and condition on $\bX_0$ and the Bernoulli variables outside row $i$.  The stopped row-$i$ trajectory, its residual row, and its right singular factors are then fixed, whereas the row-$i$ Bernoulli variables remain independent.

Before stopping, the initial error bound and the accepted updates give $\sharpnorm{\bX_k^{\mathrm r,i}-\bX_\star}\le\rho_k$. Therefore, the proof of \Cref{lem:current-incoherence} gives, for every such $k<\bar K$,
\begin{align*}
\sharpnorm{\bX_k^{\mathrm r,i}-\bX_\star} &\le\rho_k\le\rho_0 \le\frac{\sigma_r(\bX_\star)}8. \\
\norm{\bV_k^{\mathrm r,i}}_{2,\infty} &\le\frac{10}{7}\sqrt{\frac{\mu r}{n}} <2\sqrt{\frac{\mu r}{n}}.
\end{align*}
After stopping, the sequence is equal to $\bX_\star$, and \Cref{ass:incoherence} gives $\norm{\bV_\star}_{2,\infty}\le\sqrt{\mu r/n}<2\sqrt{\mu r/n}$.

Thus \Cref{lem:conditional-row} applies at every $k<\bar K$.  Applying \Cref{lem:conditional-row} to the transposed column-$j$ process, conditional on the variables outside column $j$, gives the column inequalities in \Cref{lem:loo-probability}.  Taking expectations over the conditioned variables and a union bound over at most $2n\bar K$ pairs gives failure probability at most $n^{-10}/48$ for $\bar K\le n$.

The row and column bounds in \Cref{lem:loo-probability} are obtained before alignment and are invariant under the subsequent left and right orthogonal Procrustes transformations. Consequently, conditional on every admissible realization of $\bX_0$, all row and column inequalities in the lemma hold simultaneously with failure probability at most $n^{-10}/48$.
\end{proof}

\subsection{Local solvability and correction size}
\label{sec:proof-local-conditioning}

We next use the simultaneous sampling events to prove local invertibility and bound the tangent correction.

\begin{proof}[Proof of \Cref{lem:local-conditioning}]
Since $\mu r\ge1$ and $e_\sharp\le\sigma_r(\bX_\star)/(1000\mu r)$, \Cref{lem:current-incoherence,lem:weighted-products} and \eqref{eq:tangent-projector} give
\begin{equation}\label{eq:local-conditioning-projector-bounds}
\begin{aligned}
\norm{\bU\bU^T-\bU_\star\bU_\star^T}_{2,\infty}+\norm{\bV\bV^T-\bV_\star\bV_\star^T}_{2,\infty} &\le\frac{64}{7}\sqrt{\frac{\mu r}{n}}\frac{e_\sharp}{\sigma_r(\bX_\star)},\\
\norm{\Acal^\alpha(\Pcal_{T_{\bX}}-\Pcal_{T_{\bX_\star}})}_{\F\to\F} &\le\frac{64\sqrt3}{7}\sqrt{\mu r}\frac{e_\sharp}{\sigma_r(\bX_\star)},\\
\norm{\Pcal_{T_{\bX}}-\Pcal_{T_{\bX_\star}}}_{\F\to\F} &\le\frac{16}{7}\frac{e_\sharp}{\sigma_r(\bX_\star)}.
\end{aligned}
\end{equation}
Fix $\bm\zeta\in T_{\bX}\Mcal_r$. By \Cref{lem:all-true-rip}, the triangle inequality, and \eqref{eq:local-conditioning-projector-bounds}, we have
\begin{align*}
\norm{\Acal^\alpha\bm\zeta}_{\F} &\ge\norm{\Acal^\alpha\Pcal_{T_{\bX_\star}}\bm\zeta}_{\F}-\norm{\Acal^\alpha(\Pcal_{T_{\bX}}-\Pcal_{T_{\bX_\star}})\bm\zeta}_{\F}\\
&\ge\sqrt{\frac{15}{16}}\norm{\Pcal_{T_{\bX_\star}}\bm\zeta}_{\F}-\frac{64\sqrt3}{7}\sqrt{\mu r}\frac{e_\sharp}{\sigma_r(\bX_\star)}\norm{\bm\zeta}_{\F}\\
&\ge\sqrt{\frac9{10}}\norm{\bm\zeta}_{\F}.
\end{align*}
Using the upper inequality in \Cref{lem:all-true-rip}, the triangle inequality, and \eqref{eq:local-conditioning-projector-bounds}, we also have
\begin{align*}
\norm{\Acal^\alpha\bm\zeta}_{\F} &\le\norm{\Acal^\alpha\Pcal_{T_{\bX_\star}}\bm\zeta}_{\F}+\norm{\Acal^\alpha(\Pcal_{T_{\bX}}-\Pcal_{T_{\bX_\star}})\bm\zeta}_{\F}\\
&\le\sqrt{\frac{17}{16}}\norm{\Pcal_{T_{\bX_\star}}\bm\zeta}_{\F}+\frac{64\sqrt3}{7}\sqrt{\mu r}\frac{e_\sharp}{\sigma_r(\bX_\star)}\norm{\bm\zeta}_{\F}\\
&\le\sqrt{\frac{11}{10}}\norm{\bm\zeta}_{\F}.
\end{align*}
Here we used
\[
\norm{\Pcal_{T_{\bX_\star}}\bm\zeta}_{\F}\ge\norm{\bm\zeta}_{\F}-\norm{(\Pcal_{T_{\bX}}-\Pcal_{T_{\bX_\star}})\bm\zeta}_{\F}
\]
in the lower bound and $\norm{\Pcal_{T_{\bX_\star}}\bm\zeta}_{\F}\le\norm{\bm\zeta}_{\F}$ in the upper bound. Squaring the inequalities $\norm{\Acal^\alpha\bm\zeta}_{\F}\ge\sqrt{9/10}\norm{\bm\zeta}_{\F}$ and $\norm{\Acal^\alpha\bm\zeta}_{\F}\le\sqrt{11/10}\norm{\bm\zeta}_{\F}$, and using $(\Acal^\alpha)^*\Acal^\alpha=\Rcal^\alpha$, proves \eqref{eq:local-conditioning-isometry}.

We next prove the correction estimate. The definitions of $\bm t_{\bX}$, $\bm\xi_{\bX}^\alpha$, and $\bm\eta_{\bX}^\alpha$ give
\[
\Pcal_{T_{\bX}}\Rcal^\alpha\Pcal_{T_{\bX}}\bm\eta_{\bX}^\alpha=\Pcal_{T_{\bX}}\Rcal^\alpha\bN_{\bX}.
\]
Write $\bX=\bU\bm\Sigma\bV^T$ as a compact singular value decomposition. By \eqref{eq:local-conditioning-isometry}, \Cref{lem:normal,lem:weighted-products}, and Weyl's inequality, we obtain
\begin{align*}
\sigma_{\min}(\bG_{\bX})&\ge\sigma_r(\bX_\star)-2e_\sharp\ge\frac9{10}\sigma_r(\bX_\star),\\
\norm{\bm\eta_{\bX}^\alpha}_{\F} &\le\frac{10}{9}\sqrt{\frac{11}{10}}\norm{\Acal^\alpha\bN_{\bX}}_{\F}\\
&\le\frac{10}{9}\sqrt{\frac{11}{10}}\sqrt{3n}\norm{\bL_{\bX}\bG_{\bX}^{-1}}_{\F}\norm{\bR_{\bX}}_{2,\infty}\\
&\le9\mu r\frac{e_\sharp^2}{\sigma_r(\bX_\star)}.
\end{align*}
Therefore, \eqref{eq:finite-correction-F-main} follows.
\end{proof}

\subsection{Comparison of the tangent corrections}

We first bound the variation in the sampled normal component between two nearby matrices.


\begin{lemma}\label[lemma]{lem:sampled-normal-motion}
Let $\alpha\in\mathfrak A$ and $\bX,\bY\in\Mcal_r$, and use $d_{\bX,\bY}$ from \Cref{lem:factor-motion}.  Suppose that \Cref{ass:incoherence} holds and
\[
\max\left\{\sharpnorm{\bX-\bX_\star}, \sharpnorm{\bY-\bX_\star}\right\}\le \rho \le\frac{\sigma_r(\bX_\star)}{1000\mu r}.
\]
If $d_{\bX,\bY}\le2\sqrt{\mu r/n}\,\rho$, then, on the joint events in \Cref{lem:all-true-rip,lem:degree}, it holds
\[
\norm{\Pcal_{T_{\bX}}\Rcal^\alpha(\bN_{\bX}-\bN_{\bY})}_{\F} \le32\sqrt{\frac{\mu r}{n}} \mu r\frac{\rho^2}{\sigma_r(\bX_\star)}.
\]
\end{lemma}

\begin{proof}
By the graph-factor representation in \Cref{lem:normal}, we have
\begin{equation}\label{eq:normal-motion-expansion}
\begin{aligned}
\bN_{\bX}-\bN_{\bY} ={}&(\bL_{\bX}-\bL_{\bY})\bG_{\bX}^{-1}\bR_{\bX}^T+\bL_{\bY}(\bG_{\bX}^{-1}-\bG_{\bY}^{-1})\bR_{\bX}^T\\
&+\bL_{\bY}\bG_{\bY}^{-1}(\bR_{\bX}-\bR_{\bY})^T.
\end{aligned}
\end{equation}
Applying \Cref{lem:weighted-products,lem:factor-motion,lem:normal} termwise in \eqref{eq:normal-motion-expansion}, and using $\rho\le\sigma_r(\bX_\star)/(1000\mu r)$, gives
\[
\norm{\Acal^\alpha(\bN_{\bX}-\bN_{\bY})}_{\F} \le15\sqrt{\mu r}\frac{\rho\,d_{\bX,\bY}}{\sigma_r(\bX_\star)}.
\]
Since $d_{\bX,\bY}\le2\sqrt{\mu r/n}\,\rho$, it follows that
\[
\norm{\Acal^\alpha(\bN_{\bX}-\bN_{\bY})}_{\F} \le30\mu r\sqrt{\frac{\mu r}{n}}\frac{\rho^2}{\sigma_r(\bX_\star)}.
\]
Finally, \eqref{eq:local-conditioning-isometry} and $(\Acal^\alpha)^*\Acal^\alpha=\Rcal^\alpha$ give
\[
\norm{\Pcal_{T_{\bX}}\Rcal^\alpha(\bN_{\bX}-\bN_{\bY})}_{\F}\le\sqrt{\frac{11}{10}}\norm{\Acal^\alpha(\bN_{\bX}-\bN_{\bY})}_{\F}\le32\sqrt{\frac{\mu r}{n}}\mu r\frac{\rho^2}{\sigma_r(\bX_\star)}.
\]
This proves the claim.
\end{proof}

We now estimate the tangent transport, which bound the error incurred by projecting a tangent correction onto a nearby tangent space.
\begin{lemma}\label[lemma]{lem:tangent-transport}
Let $\alpha\in\mathfrak A$, $h\ge0$, $\bX,\bY\in\Mcal_r$, and $\bZ\in T_{\bY}\Mcal_r$. Suppose that $\bX$ and $\bY$ satisfy the hypotheses of \Cref{lem:sampled-normal-motion}, that $d_{\bX,\bY}\le2\sqrt{\mu r/n}\,\rho$, and that
\[
\norm{\bZ}_{\F}\le9\mu r\frac{\rho^2}{\sigma_r(\bX_\star)}, \qquad \max\left\{\norm{\bZ}_{2,\infty},\norm{\bZ^T}_{2,\infty}\right\} \le2h+38\mu r\sqrt{\frac{\mu r}{n}}\frac{\rho^2}{\sigma_r(\bX_\star)}.
\]
Then, on the joint events in \Cref{lem:all-true-rip,lem:degree},
\[
\norm{\Pcal_{T_{\bX}}\Rcal^\alpha(\Pcal_{T_{\bY}}-\Pcal_{T_{\bX}})\bZ}_{\F} \le\frac12\mu r\sqrt{\frac{\mu r}{n}}\frac{\rho^2}{\sigma_r(\bX_\star)}+\frac1{40}h.
\]
\end{lemma}

\begin{proof}
Write $\bY=\bU_{\bY}\bS_{\bY}\bV_{\bY}^T$ and use the tangent decomposition
\[
\bZ=\bU_{\bY}\bM\bV_{\bY}^T+\bB\bV_{\bY}^T+\bU_{\bY}\bC^T, \qquad \bU_{\bY}^T\bB=0, \quad \bV_{\bY}^T\bC=0.
\]
Since the three terms are mutually orthogonal, we have
\[
\max\left\{\norm{\bM}_{\F},\norm{\bB}_{\F},\norm{\bC}_{\F}\right\}\le\norm{\bZ}_{\F}.
\]
Moreover, $\bB=\bZ\bV_{\bY}-\bU_{\bY}\bM$ and $\bC=\bZ^T\bU_{\bY}-\bV_{\bY}\bM^T$. Therefore, \Cref{lem:current-incoherence} and the assumed bounds on $\bZ$ give
\[
\max\left\{\norm{\bB}_{2,\infty},\norm{\bC}_{2,\infty}\right\} \le2h+56\mu r\sqrt{\frac{\mu r}{n}}\frac{\rho^2}{\sigma_r(\bX_\star)}.
\]

Define
\[
\bm\Delta_U:=(\bI-\bU_{\bX}\bU_{\bX}^T)\bU_{\bY}, \qquad \bm\Delta_V:=(\bI-\bV_{\bX}\bV_{\bX}^T)\bV_{\bY}.
\]
Since we have
\[
\bm\Delta_U=(\bI-\bU_{\bX}\bU_{\bX}^T)(\bY-\bX)\bV_{\bY}\bS_{\bY}^{-1}, \qquad \bm\Delta_V=(\bI-\bV_{\bX}\bV_{\bX}^T)(\bY-\bX)^T\bU_{\bY}\bS_{\bY}^{-T},
\]
Weyl's inequality gives
\begin{equation}\label{eq:tangent-transport-factor-motion}
\max\left\{\norm{\bm\Delta_U}_{\F},\norm{\bm\Delta_V}_{\F}\right\} \le\frac{d_{\bX,\bY}}{\sigma_r(\bX_\star)-\rho}.
\end{equation}
For every $i$, the identity $\bm\Delta_U=(\bI-\bU_{\bX}\bU_{\bX}^T)(\bY-\bX)\bV_{\bY}\bS_{\bY}^{-1}$ gives
\begin{align*}
\norm{\bm e_i^T\bm\Delta_U}_2 &\le\frac{\norm{\bm e_i^T(\bY-\bX)}_2+\norm{\bm e_i^T\bU_{\bX}}_2\norm{\bY-\bX}_{\op}}{\sigma_r(\bY)}
\le9\sqrt{\frac{\mu r}{n}}\frac{\rho}{\sigma_r(\bX_\star)},
\end{align*}
where the last inequality follows from the definition of the sharp norm, \Cref{lem:current-incoherence}, and $\sigma_r(\bY)\ge\sigma_r(\bX_\star)-\rho$. Applying the same calculation to $\bm\Delta_V=(\bI-\bV_{\bX}\bV_{\bX}^T)(\bY-\bX)^T\bU_{\bY}\bS_{\bY}^{-T}$ gives
\begin{equation}\label{eq:tangent-transport-factor-row}
\max\left\{\norm{\bm\Delta_U}_{2,\infty},\norm{\bm\Delta_V}_{2,\infty}\right\} \le9\sqrt{\frac{\mu r}{n}}\frac{\rho}{\sigma_r(\bX_\star)}.
\end{equation}
Furthermore, since $\bU_{\bY}^T\bB=0$ and $\bV_{\bY}^T\bC=0$, we have
\begin{align*}
\norm{(\bI-\bU_{\bX}\bU_{\bX}^T)\bB}_{2,\infty} &\le\norm{\bB}_{2,\infty}+\norm{\bU_{\bX}}_{2,\infty}\norm{(\bU_{\bX}-\bU_{\bY})^T\bB}_{\F},\\
\norm{(\bI-\bV_{\bX}\bV_{\bX}^T)\bC}_{2,\infty} &\le\norm{\bC}_{2,\infty}+\norm{\bV_{\bX}}_{2,\infty}\norm{(\bV_{\bX}-\bV_{\bY})^T\bC}_{\F}.
\end{align*}
Combining these inequalities with \Cref{lem:factor-motion}, $d_{\bX,\bY}\le2\sqrt{\mu r/n}\,\rho$, and the Frobenius bound on $\bZ$, we obtain
\begin{equation}\label{eq:tangent-transport-projected-row}
\max\left\{\norm{(\bI-\bU_{\bX}\bU_{\bX}^T)\bB}_{2,\infty},\norm{(\bI-\bV_{\bX}\bV_{\bX}^T)\bC}_{2,\infty}\right\} \le2h+57\mu r\sqrt{\frac{\mu r}{n}}\frac{\rho^2}{\sigma_r(\bX_\star)}.
\end{equation}

Since $\bZ\in T_{\bY}\Mcal_r$, substituting its tangent decomposition into \eqref{eq:tangent-projector} gives
\begin{equation}\label{eq:tangent-transport-decomposition}
\begin{aligned}
(\Pcal_{T_{\bY}}-\Pcal_{T_{\bX}})\bZ =&\bm\Delta_U\bM\bm\Delta_V^T+(\bI-\bU_{\bX}\bU_{\bX}^T)\bB\bm\Delta_V^T+\bm\Delta_U\bC^T(\bI-\bV_{\bX}\bV_{\bX}^T).
\end{aligned}
\end{equation}
Applying \Cref{lem:weighted-products} to the first term and using \eqref{eq:tangent-transport-factor-motion}--\eqref{eq:tangent-transport-factor-row}, we have
\[
\norm{\Acal^\alpha(\bm\Delta_U\bM\bm\Delta_V^T)}_{\F} \le9\sqrt{3\mu r}\frac{\rho}{\sigma_r(\bX_\star)}\norm{\bm\Delta_U}_{\F}\norm{\bM}_{\F}.
\]
Applying \Cref{lem:weighted-products} to the second and third terms in \eqref{eq:tangent-transport-decomposition}, and then using \eqref{eq:tangent-transport-factor-motion} and \eqref{eq:tangent-transport-projected-row}, we obtain
\begin{align*}
&~\norm{\Acal^\alpha((\bI-\bU_{\bX}\bU_{\bX}^T)\bB\bm\Delta_V^T)}_{\F} +\norm{\Acal^\alpha(\bm\Delta_U\bC^T(\bI-\bV_{\bX}\bV_{\bX}^T))}_{\F}\\
\le&~2\sqrt{3n}\frac{d_{\bX,\bY}}{\sigma_r(\bX_\star)-\rho} \left(2h+57\mu r\sqrt{\frac{\mu r}{n}}\frac{\rho^2}{\sigma_r(\bX_\star)}\right).
\end{align*}
Substituting $d_{\bX,\bY}\le2\sqrt{\mu r/n}\,\rho$, $\rho\le\sigma_r(\bX_\star)/(1000\mu r)$, $\norm{\bM}_{\F}\le\norm{\bZ}_{\F}$, and $\mu r\le n$ into the bounds for the three terms in \eqref{eq:tangent-transport-decomposition} gives
\begin{equation}\label{eq:tangent-transport-weighted}
\norm{\Acal^\alpha(\Pcal_{T_{\bY}}-\Pcal_{T_{\bX}})\bZ}_{\F} \le\frac25\mu r\sqrt{\frac{\mu r}{n}}\frac{\rho^2}{\sigma_r(\bX_\star)}+\frac1{45}h.
\end{equation}
Finally, \eqref{eq:local-conditioning-isometry}, $(\Acal^\alpha)^*\Acal^\alpha=\Rcal^\alpha$, and \eqref{eq:tangent-transport-weighted} give
\[
\norm{\Pcal_{T_{\bX}}\Rcal^\alpha(\Pcal_{T_{\bY}}-\Pcal_{T_{\bX}})\bZ}_{\F} \le\sqrt{\frac{11}{10}}\norm{\Acal^\alpha(\Pcal_{T_{\bY}}-\Pcal_{T_{\bX}})\bZ}_{\F} \le\frac12\mu r\sqrt{\frac{\mu r}{n}}\frac{\rho^2}{\sigma_r(\bX_\star)}+\frac1{40}h.
\]
This proves the result.
\end{proof}

\subsection{Proof of \Cref{lem:correction-proximity}}
\label{sec:proof-correction-proximity}

We now combine the normal-motion and tangent-transport estimates to compare the actual and leave-one-out tangent corrections.

\begin{proof}[Proof of \Cref{lem:correction-proximity}]
\emph{First estimate: coordinatewise bounds.} Fix $\alpha$ and a row $i$.  By the definitions of $d_k^{\rm loo}$ and $h_k^{\rm corr}$ and by \Cref{lem:factor-motion} in the aligned gauges, we have
\begin{align*}
\norm{\bm\eta_k^\alpha-\bm\eta_k^{\mathrm r,i}}_{\F} \le2h_k^{\rm corr},~~ \norm{\bX_k^\alpha-\bX_k^{\mathrm r,i}}_{\F} \le2d_k^{\rm loo},~~ \norm{\bV_k^\alpha-\bV_k^{\mathrm r,i}}_{\op} \le3\frac{d_k^{\rm loo}}{\sigma_r(\bX_\star)}.
\end{align*}
Set $\bX=\bX_k^{\mathrm r,i}=\bU\bS\bV^T$, $\bN=\bN_k^{\mathrm r,i}$, and $\bm\eta=\bm\eta_k^{\mathrm r,i}$.  Since $\Rcal^{\mathrm r,i}$ is the identity on row $i$ and $\bN\bV=0$, the correction equation and, for every $\bm a\in\R^r$, its test against $\bm e_i\bm a^T\bV^T$ give
\begin{align*}
\Pcal_{T_{\bX}}\Rcal^{\mathrm r,i}(\bN-\bm\eta)&=0,\quad \bm e_i\bm a^T\bV^T =\bU\big[(\bU^T\bm e_i)\bm a^T\big]\bV^T +\big[(\bI-\bU\bU^T)\bm e_i\bm a^T\big]\bV^T\in T_{\bX}\Mcal_r,\\
\bm e_i^T(\bN-\bm\eta)\bV&=0,\quad \bm e_i^T\bm\eta_k^{\mathrm r,i}\bV_k^{\mathrm r,i}=0.
\end{align*}
Since $\bm e_i^T\bm\eta_k^{\mathrm r,i}\bV_k^{\mathrm r,i}=0$, the $i$th row of $\bm\eta_k^{\mathrm r,i}$ contains only its $\bU\bC^T$ component. Hence \Cref{lem:current-incoherence}, \eqref{eq:finite-correction-F-main}, $\mu r/n\le1$ from \Cref{ass:incoherence}, and $\rho_k/\sigma_r(\bX_\star)\le1/(1000\mu r)$ give
\begin{align}
\norm{\bm e_i^T\bm\eta_k^{\mathrm r,i}}_2 &\le2\sqrt{\frac{\mu r}{n}}\norm{\bm\eta_k^{\mathrm r,i}}_{\F} \le18\sqrt{\frac{\mu r}{n}}\mu r \frac{\rho_k^2}{\sigma_r(\bX_\star)},\notag\\
\norm{\bm e_i^T\bm\eta_k^\alpha\bV_k^\alpha}_2 &\le2h_k^{\rm corr}+54\sqrt{\frac{\mu r}{n}}\mu r \frac{\rho_k^2}{\sigma_r(\bX_\star)} \frac{d_k^{\rm loo}}{\sigma_r(\bX_\star)}\notag\\
&\le2h_k^{\rm corr}+\frac18\sqrt{\frac{\mu r}{n}}\mu r \frac{\rho_k^2}{\sigma_r(\bX_\star)}. \label{eq:correction-row-coordinate}
\end{align}
Replacing $(\mathrm r,i,\bU,\bV)$ by $(\mathrm c,j,\bV,\bU)$ and transposing gives the column form of \eqref{eq:correction-row-coordinate}. Moreover, comparing each row with its row-completed correction, and each column with its column-completed correction, gives
\begin{equation}\label{eq:correction-row-norm}
\max\left\{\norm{\bm\eta_k^\alpha}_{2,\infty},\norm{(\bm\eta_k^\alpha)^T}_{2,\infty}\right\} \le2h_k^{\rm corr}+18\sqrt{\frac{\mu r}{n}}\mu r\frac{\rho_k^2}{\sigma_r(\bX_\star)},\qquad \alpha\in\mathfrak A.
\end{equation}

For a tangent matrix $\bm\eta$ at $\bX=\bU\bm\Sigma\bV^T$, we have
\[
\bm\eta=(\bm\eta\bV)\bV^T+\bU(\bm\eta^T\bU)^T -\bU(\bU^T\bm\eta\bV)\bV^T.
\]
Equation~\eqref{eq:correction-row-coordinate} and its transpose, together with the tangent decomposition, \Cref{lem:current-incoherence}, and $\norm{\bm\eta}_{\F}\le9\mu r\rho_k^2/\sigma_r(\bX_\star)$, give
\[
\norm{\bm\eta}_{\op} \le9\mu r\frac{\rho_k^2}{\sigma_r(\bX_\star)}, \qquad \sharpnorm{\bm\eta} \le2\sqrt{\frac{n}{\mu r}}h_k^{\rm corr} +19\mu r\frac{\rho_k^2}{\sigma_r(\bX_\star)}.
\]
This proves \eqref{eq:finite-coordinate-main}.

\emph{Second estimate: comparison of the corrections.} Fix a row $i$, and set
\[
\bX=\bX_k^0,\qquad \bY=\bX_k^{\mathrm r,i},\qquad \bm\eta=\bm\eta_k^0,\qquad \overline{\bm\eta}=\bm\eta_k^{\mathrm r,i},\qquad \bZ=\bN_{\bY}-\overline{\bm\eta}.
\]
The two correction equations are
\begin{subequations}\label{eq:correction-equations}
\begin{align}
\Pcal_{T_{\bX}}\Rcal^0\Pcal_{T_{\bX}}\bm\eta &=\Pcal_{T_{\bX}}\Rcal^0\bN_{\bX}, \label{eq:correction-actual-equation}\\
\Pcal_{T_{\bY}}\Rcal^{\mathrm r,i}\Pcal_{T_{\bY}}\overline{\bm\eta} &=\Pcal_{T_{\bY}}\Rcal^{\mathrm r,i}\bN_{\bY}. \label{eq:correction-row-equation}
\end{align}
\end{subequations}
Since $\bm\eta\in T_{\bX}$ and $\overline{\bm\eta}\in T_{\bY}$, \eqref{eq:correction-row-equation} gives $\Pcal_{T_{\bY}}\Rcal^{\mathrm r,i}\bZ=0$.  Subtracting \eqref{eq:correction-row-equation} from \eqref{eq:correction-actual-equation} and adding and subtracting $\bN_{\bY}$ and $\Rcal^{\mathrm r,i}$ give
\begin{equation}\label{eq:correction-proximity-decomposition}
\begin{aligned}
(\Pcal_{T_{\bX}}\Rcal^0\Pcal_{T_{\bX}}) (\bm\eta-\Pcal_{T_{\bX}}\overline{\bm\eta}) &= \underbrace{\Pcal_{T_{\bX}}\Rcal^0(\bN_{\bX}-\bN_{\bY})}_{\bm T_1} +\underbrace{\Pcal_{T_{\bX}}(\Rcal^0-\Rcal^{\mathrm r,i})\bZ}_{\bm T_2}\\
&\quad+\underbrace{(\Pcal_{T_{\bX}}-\Pcal_{T_{\bY}})\Rcal^{\mathrm r,i}\bZ}_{\bm T_3} +\underbrace{\Pcal_{T_{\bX}}\Rcal^0(\Pcal_{T_{\bY}}-\Pcal_{T_{\bX}})\overline{\bm\eta}}_{\bm T_4}.
\end{aligned}
\end{equation}
Replacing $(\mathrm r,i)$ by $(\mathrm c,j)$ and transposing gives the column form of \eqref{eq:correction-proximity-decomposition}. By \Cref{lem:local-conditioning}, the inverse of the left-hand operator has norm at most $10/9$.  We bound $\bm T_1,\ldots,\bm T_4$ separately, as in \cite[Lemma~11]{MaWangChiChen2020}.

For $\bm T_1$, \Cref{lem:sampled-normal-motion} gives $\displaystyle \norm{\bm T_1}_{\F} \le32\sqrt{\frac{\mu r}{n}}\mu r \frac{\rho_k^2}{\sigma_r(\bX_\star)}$.

For $\bm T_2$, write the completed row of $\bZ$ as $\bm z_i^T$, set $\delta_{ij}:=\mathbf 1_{\{(i,j)\in\widehat{\Omega}\}}$, and write $w_j=(\delta_{ij}/q-1)(\bm z_i)_j$. Conditional on $\bX_0$ and the Bernoulli variables outside row $i$, $\bm z_i$ and $\bV_k^{\mathrm r,i}$ are fixed, whereas $\{\delta_{ij}\}_{j=1}^n$ remain independent Bernoulli$(q)$ variables. Hence the conclusion of \Cref{lem:loo-probability} gives \eqref{eq:simultaneous-row-event} with $\bm z=\bm z_i$ and $\bV=\bV_k^{\mathrm r,i}$.

Write $\bX_k^{\mathrm r,i}=\bU\bm\Sigma\bV^T$, $\bN=\bN_k^{\mathrm r,i}$, and $\overline{\bm\eta}=\bm\eta_k^{\mathrm r,i}$.  Testing the correction equation against $\bm e_i\bm a^T\bV^T\in T_{\bX_k^{\mathrm r,i}}\Mcal_r$ and using $\bN\bV=0$ give
\[
\bm e_i^T(\bN-\overline{\bm\eta})\bV=0, \qquad \bm e_i^T\overline{\bm\eta}\bV=0.
\]
Set $\bC_{\rm off}:=(\bI-\bV\bV^T)\overline{\bm\eta}^{T}\bU$. Using \Cref{lem:normal} for $\bN_k^{\mathrm r,i}$, the identity $\bm e_i^T\overline{\bm\eta}\bV=0$ and \Cref{lem:local-conditioning} for $\bm e_i^T\overline{\bm\eta}$, and the transposed form of \eqref{eq:correction-row-coordinate} for $\bC_{\rm off}$, the tangent decomposition of $\overline{\bm\eta}$ gives
\begin{subequations}\label{eq:loo-residual-bounds}
\begin{align}
\norm{\bm e_i^T\bN_k^{\mathrm r,i}}_2 &\le\frac{16}{3}\sqrt{\frac{\mu r}{n}} \frac{\rho_k^2}{\sigma_r(\bX_\star)}, \notag\\
\norm{\bN_k^{\mathrm r,i}}_\infty &\le\frac{16}{3}\frac{4\mu r}{n} \frac{\rho_k^2}{\sigma_r(\bX_\star)},\notag\\
\norm{\bm e_i^T\overline{\bm\eta}}_2 &\le2\sqrt{\frac{\mu r}{n}} \norm{\overline{\bm\eta}}_{\F} \le18\sqrt{\frac{\mu r}{n}}\mu r \frac{\rho_k^2}{\sigma_r(\bX_\star)},\notag\\
\norm{\bm e_j^T\bC_{\rm off}}_2 &\le\norm{\bm e_j^T\overline{\bm\eta}^T\bU}_2 +\norm{\bm e_j^T\bV}_2 \norm{\bU^T\overline{\bm\eta}\bV}_{\op}\notag\\
&\le2h_k^{\rm corr} +\frac18\sqrt{\frac{\mu r}{n}}\mu r \frac{\rho_k^2}{\sigma_r(\bX_\star)} +18\sqrt{\frac{\mu r}{n}}\mu r \frac{\rho_k^2}{\sigma_r(\bX_\star)},\notag\\
\norm{\bm z_i}_2 &\le24\mu r\sqrt{\frac{\mu r}{n}} \frac{\rho_k^2}{\sigma_r(\bX_\star)}, \label{eq:loo-residual-row}\\
\norm{\bm z_i}_\infty &\le\frac{60(\mu r)^2}{n}\frac{\rho_k^2}{\sigma_r(\bX_\star)} +4\sqrt{\frac{\mu r}{n}}\,h_k^{\rm corr}. \label{eq:loo-residual-entry}
\end{align}
\end{subequations}
Substituting \eqref{eq:loo-residual-row} and \eqref{eq:loo-residual-entry} into \eqref{eq:simultaneous-row-event}, and using $q\ge c_4\mu r\log n/n$ with $c_4=2^{20}$, gives
\[
2\sqrt{\frac{\mu r}{n}}\norm{\bm w}_2 +\norm{(\bV_k^{\mathrm r,i})^T\bm w}_2 \le13\sqrt{\frac{\mu r}{n}} \mu r \frac{\rho_k^2}{\sigma_r(\bX_\star)}+\frac{17}{50}h_k^{\rm corr}.
\]
Equation~\eqref{eq:tangent-projector}, \Cref{lem:factor-motion}, and $\rho_k\le\sigma_r(\bX_\star)/(1000\mu r)$ give
\begin{align*}
\norm{\Pcal_{T_{\bY}}(\bm e_i\bm w^T)}_{\F}^2 &=\norm{\bU_{\bY}^T\bm e_i}_2^2\norm{\bm w}_2^2 +\norm{(\bI-\bU_{\bY}\bU_{\bY}^T)\bm e_i}_2^2 \norm{\bV_{\bY}^T\bm w}_2^2,\\
3\frac{d_k^{\rm loo}}{\sigma_r(\bX_\star)}\norm{\bm w}_2 &\le\frac{3\rho_k}{\sigma_r(\bX_\star)} \left(2\sqrt{\frac{\mu r}{n}}\norm{\bm w}_2 +\norm{\bV_{\bY}^T\bm w}_2\right),\\
\norm{\bm T_2}_{\F} &\le14\mu r\sqrt{\frac{\mu r}{n}} \frac{\rho_k^2}{\sigma_r(\bX_\star)}+\frac7{20}h_k^{\rm corr}.
\end{align*}

For $\bm T_3$, we use the row-coordinate estimate \eqref{eq:correction-row-coordinate} and its transpose. Together with the tangent decomposition and the Frobenius bound on $\overline{\bm\eta}$, they give
\[
\norm{\overline{\bm\eta}}_\infty \le8\sqrt{\frac{\mu r}{n}}h_k^{\rm corr}+37\frac{(\mu r)^2}{n}\frac{\rho_k^2}{\sigma_r(\bX_\star)}.
\]
Consequently, \eqref{eq:correction-row-equation}, \Cref{lem:normal,lem:degree,lem:factor-motion}, $d_k^{\rm loo}\le2\sqrt{\mu r/n}\,\rho_k$, and $\rho_k\le\sigma_r(\bX_\star)/(1000\mu r)$ give
\begin{align*}
\norm{\bZ}_\infty &\le8\sqrt{\frac{\mu r}{n}} h_k^{\rm corr} +64\frac{(\mu r)^2}{n} \frac{\rho_k^2}{\sigma_r(\bX_\star)},\\
\bU_{\bY}^T\Rcal^{\mathrm r,i}\bZ&=0, \quad \Rcal^{\mathrm r,i}\bZ\bV_{\bY}=0,\\
\norm{\Rcal^{\mathrm r,i}\bZ}_{\op} &\le3n\norm{\bZ}_\infty,\\
\norm{(\Pcal_{T_{\bX}}-\Pcal_{T_{\bY}}) \Rcal^{\mathrm r,i}\bZ}_{\F} &\le3\frac{d_k^{\rm loo}}{\sigma_r(\bX_\star)} \norm{\Rcal^{\mathrm r,i}\bZ}_{\op},\\
\norm{\bm T_3}_{\F} &\le\frac65\sqrt{\frac{\mu r}{n}} \mu r \frac{\rho_k^2}{\sigma_r(\bX_\star)}+\frac3{20}h_k^{\rm corr}.
\end{align*}
For $\bm T_4$, \eqref{eq:finite-correction-F-main} and \eqref{eq:correction-row-norm}, followed by \Cref{lem:tangent-transport}, give
\begin{align*}
\norm{\overline{\bm\eta}}_{\F} &\le9\mu r\frac{\rho_k^2}{\sigma_r(\bX_\star)},\\
\max\{\norm{\overline{\bm\eta}}_{2,\infty}, \norm{\overline{\bm\eta}^{T}}_{2,\infty}\} &\le2h_k^{\rm corr}+38\mu r\sqrt{\frac{\mu r}{n}} \frac{\rho_k^2}{\sigma_r(\bX_\star)},\\
\norm{\bm T_4}_{\F} &\le\frac12\sqrt{\frac{\mu r}{n}} \mu r \frac{\rho_k^2}{\sigma_r(\bX_\star)}+\frac1{40}h_k^{\rm corr}.
\end{align*}
To compare the two corrections themselves, we also account for the component of $\overline{\bm\eta}$ normal to $T_{\bX}$. By \Cref{lem:factor-motion,lem:local-conditioning} and \eqref{eq:correction-proximity-decomposition}, we have
\[
\norm{\bm\eta-\overline{\bm\eta}}_{\F} \le\frac{10}{9}\sum_{j=1}^4\norm{\bm T_j}_{\F} +\frac{3d_k^{\rm loo}}{\sigma_r(\bX_\star)}\norm{\overline{\bm\eta}}_{\F}.
\]
Substituting the four estimates for $\norm{\bm T_1}_{\F},\ldots,\norm{\bm T_4}_{\F}$ derived in this proof, using $d_k^{\rm loo}\le2\sqrt{\mu r/n}\,\rho_k$, and applying their transposed counterparts to each completed column yield
\[
h_k^{\rm corr} \le\left(53+54\frac{\rho_k}{\sigma_r(\bX_\star)}\right)\sqrt{\frac{\mu r}{n}}\mu r\frac{\rho_k^2}{\sigma_r(\bX_\star)}+\frac7{12}h_k^{\rm corr}.
\]
Since $\rho_k/\sigma_r(\bX_\star)\le1/(1000\mu r)$ and $\mu r\ge1$, rearranging gives
\[
h_k^{\rm corr} \le128\sqrt{\frac{\mu r}{n}}\,\mu r \frac{\rho_k^2}{\sigma_r(\bX_\star)}.
\]
This proves \eqref{eq:finite-proximity-main}. Together with \eqref{eq:finite-coordinate-main}, this completes the proof.
\end{proof}

\section{Matrix-free implementation and computational complexity}\label{app:implementation}

We describe the matrix-free realization of the RGD and RGN steps and record the resulting  complexity. Let $\bX=\bU\bS\bV^T\in\Mcal_r$ be a compact singular value decomposition. Every $\bm\xi\in T_{\bX}\Mcal_r$ can be written as
\[
\bm\xi=\bU\bM\bV^T+\bB\bV^T+\bU\bC^T, \qquad \bU^T\bB=0, \quad \bV^T\bC=0;
\]
see \cite[Sec.~2.1]{Vandereycken2013}. Define $\Phi_{\bX}(\bM,\bB,\bC):=\bU\bM\bV^T+\bB\bV^T+\bU\bC^T$. Then $\Phi_{\bX}$ is an isometry onto $T_{\bX}\Mcal_r$ with
\[
\Phi_{\bX}^*(\bZ) =\bigl(\bU^T\bZ\bV,(\bI-\bU\bU^T)\bZ\bV,(\bI-\bV\bV^T)\bZ^T\bU\bigr).
\]
For the RGD direction, orthogonality gives
\[
\norm{\bm\xi}_{\F}^2=\norm{\bM}_{\F}^2+\norm{\bB}_{\F}^2+\norm{\bC}_{\F}^2.
\]
Thus the numerator in \eqref{eq:rgd-exact-line-search} costs $O(nr)$ operations. Its denominator is obtained by first forming $\bU\bM$, then evaluating $\Phi_{\bX}(\bM,\bB,\bC)$ on $\Omega$ and summing the squared entries, which costs $O(|\Omega|r+nr^2)$ operations. The graph retraction costs $O(nr^2+r^3)$ operations. Since $r\le n$, the exact line search and the graph retraction therefore preserve the $O(|\Omega|r+nr^2)$ cost per RGD iteration.

For RGN, writing $\bm z=(\bM,\bB,\bC)$, the tangent normal equation \eqref{eq:ambient-normal-equation} is equivalent to
\begin{equation}\label{eq:matrix-free-coordinate-normal}
\mathcal H_{\bX}\bm z=-\bm g_{\bX}, \qquad \mathcal H_{\bX}:=\Phi_{\bX}^*q^{-1}\Pcal_{\widehat{\Omega}}\Phi_{\bX}, \qquad \bm g_{\bX}:=\Phi_{\bX}^*q^{-1}\Pcal_{\widehat{\Omega}}(\bX-\bX_\star).
\end{equation}
Thus no $r(2n-r)\times r(2n-r)$ normal matrix is formed. Both the formation of $\bm g_{\bX}$ and one application of $\mathcal H_{\bX}$ cost $O(|\widehat{\Omega}|r+nr^2)$ operations, and the graph retraction costs $O(nr^2+r^3)$ operations; see \cite{AbsilOseledets2015}. This gives the standard orders for fixed-rank Riemannian matrix completion; see, for example, \cite{WeiCaiChanLeung2020}.

By \Cref{lem:local-conditioning}, $\mathcal H_{\bX}$ is positive definite along the RGN iterates covered by \Cref{thm:fixed}. Therefore, exact CG applied to \eqref{eq:matrix-free-coordinate-normal} is well defined and terminates in at most $\dim(T_{\bX}\Mcal_r)=r(2n-r)$ iterations \cite{Greenbaum1997}. In particular, $1\le J_k\le r(2n-r)$ whenever an RGN correction is computed, and the $k$th RGN iteration costs
\[
O\!\left(J_k(|\widehat{\Omega}|r+nr^2)\right).
\]

\subsection{Proof of \Cref{prop:reconstruction-cost}}
\label{proof:reconstruction-cost}
\begin{proof}
Set
\[
\bY=\bZ+q^{-1}\Pcal_\Lambda(\bX_\star-\bZ) =\bU\bm\Sigma\bV^T+\bS,
\]
where $\rank(\bZ)\le r$ and $\bS$ has at most $m$ nonzero entries. For $\bQ\in\R^{n\times r}$,
\[
\bY^T\bQ =\bV\bm\Sigma(\bU^T\bQ)+\bS^T\bQ, \qquad \bY\bQ =\bU\bm\Sigma(\bV^T\bQ)+\bS\bQ.
\]
Thus the two matrix products in each step of \eqref{eq:threshold-iteration}, together with the thin QR factorization, require $O(mr+nr^2)$ operations. Since there are $\ceil{12\log n}$ such steps, the subspace iteration costs $O\!\left((mr+nr^2)\log n\right)$. For the final spectral reconstruction, compute
\[
\bY^T\bQ=\widetilde{\bQ}\bR, \qquad \bR^T=\bU_R\bm\Sigma_R\bV_R^T.
\]
Then
\[
\Hcal_{\ge\tau/8}(\bQ^T\bY) = \bU_R\Hcal_{\ge\tau/8}(\bm\Sigma_R) (\widetilde{\bQ}\bV_R)^T,
\]
and hence
\[
\Tcal_\tau(\bY) = \bQ\bU_R\Hcal_{\ge\tau/8}(\bm\Sigma_R) (\widetilde{\bQ}\bV_R)^T.
\]
The thin QR factorization and the $r\times r$ singular value decomposition require $O(mr+nr^2)$ additional operations. Evaluating the entries of $\bZ$ on $\Lambda$ and summing the squared residuals costs $O(mr)$ operations. This proves the result.
\end{proof}

\subsection{Computational complexity of RGD and RGN}
\begin{theorem}\label[theorem]{thm:complexity}
Let $0<\varepsilon\le1/4$. Under the conditions of \Cref{thm:rgd-global}, RGD attains relative Frobenius accuracy $\varepsilon$ with  complexity
\[
O\!\left( \mu nr^2\log n\log(n\kappa) \left[ \log n+\log\frac1\varepsilon \right] \right).
\]
Under the conditions of \Cref{thm:end}, RGN attains relative Frobenius accuracy $\varepsilon$ with arithmetic complexity
\[
O\!\left( \mu nr^2\log n \left[ \log n\log(2\mu r\kappa) + J\log\log\frac1\varepsilon \right] \right),
\]
where $J_k\le J\le r(2n-r)$ for all RGN iterations used to attain the prescribed accuracy.
\end{theorem}

\begin{proof}
\emph{Initialization.}
Let $m_\ell=|\Omega^{(\ell)}|$ and let $B=K+1$ for RGD and $B=K+2$ for RGN. By \Cref{prop:reconstruction-cost}, at most $K$ reconstruction steps require
\[
O\!\left( \left[ r\sum_{\ell=1}^{K}m_\ell+Knr^2 \right]\log n \right)
\]
operations. The residual tests cost at most $O(r\sum_{\ell=1}^{K}m_\ell)$ operations and have lower order. Since the observation components are independent Bernoulli$(q)$ sets, $\sum_{\ell=0}^{B-1}m_\ell=O(Bqn^2)$ with high probability. Moreover, $Bq=O(p)$ and $qn\gtrsim\mu r\log n$, so $Knr^2=O(Bqn^2r)$. Therefore the initialization cost is
\begin{equation}\label{eq:initialization-cost-simple}
O(pn^2r\log n).
\end{equation}

\emph{RGD.} By the matrix-free implementation above, one RGD iteration costs $O(|\Omega|r+nr^2)$. At the sampling rate of \Cref{thm:rgd-global}, $|\Omega|=O\!\left(\mu nr\log n\log(n\kappa)\right)$ with high probability. Combining the resulting per-iteration cost, the linear convergence in \Cref{thm:rgd-global}, and \eqref{eq:initialization-cost-simple} gives
\[
O\!\left( \mu nr^2\log n\log(n\kappa) \left[ \log n+\log\frac1\varepsilon \right] \right).
\]

\emph{RGN.} By the matrix-free implementation above, the $k$th RGN iteration costs $O\!\left(J_k(|\widehat{\Omega}|r+nr^2)\right)$. At the sampling rate of \Cref{thm:end}, $|\widehat{\Omega}|=O(\mu nr\log n)$ with high probability. Combining this estimate, the convergence rate in \Cref{thm:end}, and \eqref{eq:initialization-cost-simple} gives
\[
O\!\left( \mu nr^2\log n \left[ \log n\log(2\mu r\kappa) + J\log\log\frac1\varepsilon \right] \right).
\]
\end{proof}

\bibliographystyle{plain}
\bibliography{references}

\begin{thebibliography}{10}

\bibitem{AbsilOseledets2015}
P.-A. Absil and Ivan~V. Oseledets.
\newblock Low-rank retractions: A survey and new results.
\newblock {\em Computational Optimization and Applications}, 62(1):5--29, 2015.

\bibitem{BlanchardTannerWei2015}
Jeffrey~D. Blanchard, Jared Tanner, and Ke~Wei.
\newblock {CGIHT}: Conjugate gradient iterative hard thresholding for
  compressed sensing and matrix completion.
\newblock {\em Information and Inference: A Journal of the IMA}, 4(4):289--327,
  2015.

\bibitem{CaiCaiYou2023}
Hanqin Cai, Jian-Feng Cai, and Juntao You.
\newblock Structured gradient descent for fast robust low-rank {Hankel} matrix
  completion.
\newblock {\em SIAM Journal on Scientific Computing}, 45(3):A1172--A1198, 2023.

\bibitem{CaiHuangLuYou2025}
HanQin Cai, Longxiu Huang, Xiliang Lu, and Juntao You.
\newblock Accelerating ill-conditioned {Hankel} matrix recovery via structured
  {Newton}-like descent.
\newblock {\em Inverse Problems}, 41(7):075015, 2025.

\bibitem{CaiWuXia2025}
Jian-Feng Cai, Tong Wu, and Ruizhe Xia.
\newblock Fast non-convex matrix sensing with optimal sample complexity.
\newblock In {\em Proceedings of the Forty-first Conference on Uncertainty in
  Artificial Intelligence}, pages 497--520, 2025.

\bibitem{CandesRecht2009}
Emmanuel~J. Cand{\`e}s and Benjamin Recht.
\newblock Exact matrix completion via convex optimization.
\newblock {\em Foundations of Computational Mathematics}, 9(6):717--772, 2009.

\bibitem{CandesTao2010}
Emmanuel~J. Cand{\`e}s and Terence Tao.
\newblock The power of convex relaxation: Near-optimal matrix completion.
\newblock {\em IEEE Transactions on Information Theory}, 56(5):2053--2080,
  2010.

\bibitem{Chen2015}
Yudong Chen.
\newblock Incoherence-optimal matrix completion.
\newblock {\em IEEE Transactions on Information Theory}, 61(5):2909--2923,
  2015.

\bibitem{ChenBhojanapalliSanghaviWard2015}
Yudong Chen, Srinadh Bhojanapalli, Sujay Sanghavi, and Rachel Ward.
\newblock Completing any low-rank matrix, provably.
\newblock {\em Journal of Machine Learning Research}, 16(94):2999--3034, 2015.

\bibitem{DingChen2020}
Lijun Ding and Yudong Chen.
\newblock Leave-one-out approach for matrix completion: Primal and dual
  analysis.
\newblock {\em IEEE Transactions on Information Theory}, 66(11):7274--7301,
  2020.

\bibitem{Greenbaum1997}
Anne Greenbaum.
\newblock {\em Iterative Methods for Solving Linear Systems}, volume~17 of {\em
  Frontiers in Applied Mathematics}.
\newblock Society for Industrial and Applied Mathematics, Philadelphia, PA,
  1997.

\bibitem{HalkoMartinssonTropp2011}
Nathan Halko, Per-Gunnar Martinsson, and Joel~A. Tropp.
\newblock Finding structure with randomness: Probabilistic algorithms for
  constructing approximate matrix decompositions.
\newblock {\em SIAM Review}, 53(2):217--288, 2011.

\bibitem{Hardt2014}
Moritz Hardt.
\newblock Understanding alternating minimization for matrix completion.
\newblock In {\em 2014 IEEE 55th Annual Symposium on Foundations of Computer
  Science}, pages 651--660, 2014.

\bibitem{HardtWootters2014}
Moritz Hardt and Mary Wootters.
\newblock Fast matrix completion without the condition number.
\newblock In {\em Proceedings of the 27th Conference on Learning Theory}, pages
  638--678, 2014.

\bibitem{JainMekaDhillon2010}
Prateek Jain, Raghu Meka, and Inderjit~S. Dhillon.
\newblock Guaranteed rank minimization via singular value projection.
\newblock In {\em Advances in Neural Information Processing Systems},
  volume~23, pages 937--945, 2010.

\bibitem{JainNetrapalli2015}
Prateek Jain and Praneeth Netrapalli.
\newblock Fast exact matrix completion with finite samples.
\newblock In {\em Proceedings of the 28th Conference on Learning Theory}, pages
  1007--1034, 2015.

\bibitem{JainNetrapalliSanghavi2013}
Prateek Jain, Praneeth Netrapalli, and Sujay Sanghavi.
\newblock Low-rank matrix completion using alternating minimization.
\newblock In {\em Proceedings of the Forty-Fifth Annual ACM Symposium on Theory
  of Computing}, pages 665--674, 2013.

\bibitem{KeshavanMontanariOh2010}
Raghunandan~H. Keshavan, Andrea Montanari, and Sewoong Oh.
\newblock Matrix completion from a few entries.
\newblock {\em IEEE Transactions on Information Theory}, 56(6):2980--2998,
  2010.

\bibitem{KummerleVerdun2021}
Christian K{\"u}mmerle and Claudio~M. Verdun.
\newblock A scalable second order method for ill-conditioned matrix completion
  from few samples.
\newblock In {\em Proceedings of the 38th International Conference on Machine
  Learning}, pages 5872--5883, 2021.

\bibitem{LiuVandenberghe2009}
Zhang Liu and Lieven Vandenberghe.
\newblock Interior-point method for nuclear norm approximation with application
  to system identification.
\newblock {\em SIAM Journal on Matrix Analysis and Applications},
  31(3):1235--1256, 2009.

\bibitem{MaWangChiChen2020}
Cong Ma, Kaizheng Wang, Yuejie Chi, and Yuxin Chen.
\newblock Implicit regularization in nonconvex statistical estimation: Gradient
  descent converges linearly for phase retrieval, matrix completion, and blind
  deconvolution.
\newblock {\em Foundations of Computational Mathematics}, 20(3):451--632, 2020.

\bibitem{NgoSaad2012}
Thanh~T. Ngo and Yousef Saad.
\newblock Scaled gradients on {Grassmann} manifolds for matrix completion.
\newblock In {\em Advances in Neural Information Processing Systems},
  volume~25, pages 1412--1420, 2012.

\bibitem{SunLuo2016}
Ruoyu Sun and Zhi-Quan Luo.
\newblock Guaranteed matrix completion via non-convex factorization.
\newblock {\em IEEE Transactions on Information Theory}, 62(11):6535--6579,
  2016.

\bibitem{TannerWei2013}
Jared Tanner and Ke~Wei.
\newblock Normalized iterative hard thresholding for matrix completion.
\newblock {\em SIAM Journal on Scientific Computing}, 35(5):S104--S125, 2013.

\bibitem{TongMaChi2021}
Tian Tong, Cong Ma, and Yuejie Chi.
\newblock Accelerating ill-conditioned low-rank matrix estimation via scaled
  gradient descent.
\newblock {\em Journal of Machine Learning Research}, 22(150):1--63, 2021.

\bibitem{Tropp2012}
Joel~A. Tropp.
\newblock User-friendly tail bounds for sums of random matrices.
\newblock {\em Foundations of Computational Mathematics}, 12(4):389--434, 2012.

\bibitem{VakninLauferNadler2025}
Eilon Vaknin~Laufer and Boaz Nadler.
\newblock {RGNMR}: A {Gauss--Newton} method for robust matrix completion with
  theoretical guarantees.
\newblock In {\em Advances in Neural Information Processing Systems},
  volume~38, pages 66931--66973, 2025.

\bibitem{Vandereycken2013}
Bart Vandereycken.
\newblock Low-rank matrix completion by {Riemannian} optimization.
\newblock {\em SIAM Journal on Optimization}, 23(2):1214--1236, 2013.

\bibitem{WangZhangGu2017}
Lingxiao Wang, Xiao Zhang, and Quanquan Gu.
\newblock A unified computational and statistical framework for nonconvex
  low-rank matrix estimation.
\newblock In {\em Proceedings of the 20th International Conference on
  Artificial Intelligence and Statistics}, pages 981--990, 2017.

\bibitem{WangWei2026}
Tianming Wang and Ke~Wei.
\newblock Leave-one-out analysis for nonconvex robust matrix completion with
  general thresholding functions.
\newblock {\em Numerical Algorithms}, 2026.

\bibitem{WeiCaiChanLeung2020}
Ke~Wei, Jian-Feng Cai, Tony~F. Chan, and Shingyu Leung.
\newblock Guarantees of {Riemannian} optimization for low rank matrix
  completion.
\newblock {\em Inverse Problems and Imaging}, 14(2):233--265, 2020.

\bibitem{XuShenChiMa2023}
Xingyu Xu, Yandi Shen, Yuejie Chi, and Cong Ma.
\newblock The power of preconditioning in overparameterized low-rank matrix
  sensing.
\newblock In {\em Proceedings of the 40th International Conference on Machine
  Learning}, pages 38611--38654, 2023.

\bibitem{ZhuYuan2025}
Xiaojing Zhu and Fengyi Yuan.
\newblock A {Riemannian} regularized {Gauss--Newton} method for low-rank matrix
  completion.
\newblock {\em AIMS Mathematics}, 10(12):28556--28582, 2025.

\bibitem{ZilberNadler2022}
Pini Zilber and Boaz Nadler.
\newblock {GNMR}: A provable one-line algorithm for low rank matrix recovery.
\newblock {\em SIAM Journal on Mathematics of Data Science}, 4(2):909--934,
  2022.

\end{thebibliography}

\end{document}